\documentclass[11pt]{article}
\usepackage{amsmath,amssymb,amsthm,amscd}

\newtheorem{theorem}{Theorem}[section]
\newtheorem{lemma}[theorem]{Lemma}
\newtheorem{proposition}[theorem]{Proposition}
\newtheorem{corollary}[theorem]{Corollary}
\newtheorem{example}[theorem]{Example}

\theoremstyle{plain}

\theoremstyle{definition}
\newtheorem{definition}[theorem]{Definition}
\newtheorem{remark}[theorem]{Remark}

\numberwithin{equation}{section}

\renewcommand{\labelenumi}{\textup{(\theenumi)}}

\title{Simple purely infinite $C^*$-algebras associated with normal subshifts}

\author{Kengo Matsumoto \\
Department of Mathematics \\
Joetsu University of Education \\
Joetsu, 943-8512, Japan
}

\begin{document}
\maketitle

\date{}

%\email{kengo{@@}juen.ac.jp}
\begin{abstract}
We will introduce a notion of normal subshifts.
A subshift $(\Lambda,\sigma)$ is said to be normal
if it satisfies a certain synchronizing property called 
$\lambda$-synchronizing and is infinite as a set.
We have lots of purely infinite simple $C^*$-algebras 
from normal subshifts including irreducible infinite sofic shifts,
Dyck shifts, $\beta$-shifts, and so on.
Eventual conjugacy of one-sided normal subshifts
and topological conjugacy of two-sided normal subshifts 
are characterized in terms of the associated $C^*$-algebras
and the associated stabilized $C^*$-algebras 
with its diagonals and gauge actions,
respectively.
\end{abstract}



%Mathematics Subject Classification 2000:

%%%%%%%%%%%%%%%%%%%%%%%%%%%%%%%%%%%%%%%%%%%%%%%%%%%%%%                        

\newcommand{\R}{\mathbb{R}}
\newcommand{\T}{\mathbb{T}}
\newcommand{\Z}{\mathbb{Z}}
\newcommand{\N}{\mathbb{N}}
\newcommand{\Zp}{{\mathbb{Z}}_+}

\def\LLL{{ {\frak L}^{\lambda(\Lambda)} }}
\def\LLTL{{ {\frak L}^{\lambda(\widetilde{\Lambda})} }}

\def\OA{{ {\mathcal{O}}_A }}

\def\OL{{ {\mathcal{O}}_{\frak L} }}
\def\DL{{ {\mathcal{D}}_{\frak L}}}
\def\AL{{ {\mathcal{A}}_{\frak L}}}
\def\OLmin{{ {\mathcal{O}}_{{\Lambda}^{\operatorname{min}}} }}
\def\LLmin{{ {{\frak L}_\Lambda^{\operatorname{min}}} }}
\def\OLamonemin{{ {\mathcal{O}}_{{\Lambda_1}^{\operatorname{min}}} }}
\def\LLamonemin{{ {{\frak L}_{\Lambda_1}^{\operatorname{min}}} }}
\def\OLamtwomin{{ {\mathcal{O}}_{{\Lambda_2}^{\operatorname{min}}} }}
\def\LLamtwomin{{ {{\frak L}_{\Lambda_2}^{\operatorname{min}}} }}
\def\LLamptwomin{{ {{\frak L}_{\Lambda'_2}^{\operatorname{min}}} }}

\def\DLLamtwomin{{ {\mathcal{D}}_{{\frak L}_{\Lambda_2}^{\operatorname{min}} } }}
\def\DLamtwo{{ {\mathcal{D}}_{{\Lambda_2} } }}
\def\DLamtwomin{{ {\mathcal{D}}_{{\Lambda_2}^{\operatorname{min}}} }}

\def\OLamptwomin{{ {\mathcal{O}}_{{\Lambda'_2}^{\operatorname{min}}} }}
\def\DLLamptwomin{{ {\mathcal{D}}_{{\frak L}_{\Lambda'_2}^{\operatorname{min}}} }}
\def\DLamptwo{{ {\mathcal{D}}_{{\Lambda'_2}  } }}
\def\DLamptwomin{{ {\mathcal{D}}_{{\Lambda'_2}^{\operatorname{min}}} }}
\def\ALLamptwomin{{ {\mathcal{A}}_{{\frak L}_{\Lambda'_2}^{\operatorname{min}}} }}

\def\OhatLamtwomin{{ \widehat{\mathcal{O}}_{{\Lambda'_2}^{\operatorname{min}}} }}
\def\DhatLamtwo{{ \widehat{\mathcal{D}}_{{\Lambda'_2}} }}
\def\DhatLamtwomin{{ \widehat{\mathcal{D}}_{{\Lambda'_2}^{\operatorname{min}}} }}
\def\DhatLLamtwomin{{ 
\widehat{\mathcal{D}}_{{\frak L}_{\Lambda'_2}^{\operatorname{min}}} 
}}

\def\ALLamptwomin{{ {\mathcal{A}}_{{\frak L}_{\Lambda'_2}^{\operatorname{min}}} }}
\def\AhatLLamtwomin{{ 
\widehat{\mathcal{A}}_{{\frak L}_{\Lambda'_2}^{\operatorname{min}}} 
}}

\def\rhoLamtwo{ \rho^{{\Lambda_2}} }
\def\rhoLamptwo{ \rho^{{\Lambda'_2}} }
\def\rhohatLamtwo{{ {\hat{\rho}}^{{\Lambda'_2}} }}
\def\rhohatLamtwomin{{ {\hat{\rho}}^{{\Lambda'_2}^{\operatorname{min}}} }}
\def\rhoLamtwomin{ {\rho}^{{\Lambda_2}^{\operatorname{min}}} }
\def\rhoLamptwomin{ {\rho}^{{\Lambda'_2}^{\operatorname{min}}} }
    
\def\FLamtwomin{{ {\mathcal{F}}_{{\Lambda_2}^{\operatorname{min}}} }}
\def\FhatLamtwomin{{ \widehat{\mathcal{F}}_{{\Lambda'_2}^{\operatorname{min}}} }}
\def\FLLamptwomin{{ {\mathcal{F}}_{{\frak L}_{\Lambda'_2}^{\operatorname{min}}} }}

\def\FL{{{\mathcal{F}}_{\frak L}}}
\def\FLmin{{ {\mathcal{F}}_{{\Lambda}^{\operatorname{min}}} }}

\def\DLmin{{ {\mathcal{D}}_{{\frak L}_\Lambda^{\operatorname{min}}} }}

\newcommand{\K}{\mathcal{K}}
\newcommand{\C}{\mathcal{C}}
\def\DLam{{ {\mathcal{D}}_{\Lambda}}}

\def\L{{\frak L}}

\def\OLL{{ {\cal O}_{\lambda(\Lambda)}  }}
\def\OLTL{{ {\cal O}_{\lambda(\widetilde{\Lambda})}  }}
\def\ALL{{ {\cal A}_{\lambda(\Lambda)}  }}
\def\ALTL{{ {\cal A}_{\lambda(\widetilde{\Lambda})}  }}
\def\DLL{{ {\cal D}_{\lambda(\Lambda)}  }}
\def\DLTL{{ {\cal D}_{\lambda(\widetilde{\Lambda})}  }}

\def\FKL{{ {\cal F}_k^{l} }}
\def\A{{ {\mathcal{A}} }}
\def\D{{ {\mathcal{D}} }}
\def\F{{ {\mathcal{F}} }}
\def\Ext{{{\operatorname{Ext}}}}
\def\Im{{{\operatorname{Im}}}}
\def\Hom{{{\operatorname{Hom}}}}
\def\Ker{{{\operatorname{Ker}}}}
\def\dim{{{\operatorname{dim}}}}
\def\min{{{\operatorname{min}}}}
\def\id{{{\operatorname{id}}}}
\def\OLF{{{\cal O}_{{\frak L}^{Ch(D_F)}}}}
\def\OLN{{{\cal O}_{{\frak L}^{Ch(D_N)}}}}
\def\OLA{{{\cal O}_{{\frak L}^{Ch(D_A)}}}}
\def\LCHDA{{{{\frak L}^{Ch(D_A)}}}}
\def\LCHDF{{{{\frak L}^{Ch(D_F)}}}}
\def\LCHLA{{{{\frak L}^{Ch(\Lambda_A)}}}}
\def\LWA{{{{\frak L}^{W(\Lambda_A)}}}}

%%%%%%%%%%%%%%%%%%%%%%%%%%%%%%%%%%%%%%% 

\bigskip

Contents:

\begin{enumerate}
\renewcommand{\theenumi}{\arabic{enumi}}
\renewcommand{\labelenumi}{\textup{\theenumi}}
\item Introduction
\item $\lambda$-synchronization and normal subshifts
\item Structure and simplicity of $\OLmin$
\item Irreducible sofic shifts
\item Other examples of normal subshifts
\item Continuous orbit equivalence
\item One-sided topological conjugacy  
\item One-sided eventual conjugacy
\item Two-sided topological conjugacy
\end{enumerate}

\newpage

\section{Introduction}

In \cite{KMAAA2013} (see also \cite{MaIsrael2013}, \cite{MaJAMS2013}), 
W. Krieger and the author introduced the notion of  $\lambda$-synchronization 
for subshifts. 
The class of $\lambda$-synchronizing  subshifts contains a lot of important and interesting subshifts such as irreducible shifts of finite type, irreducible sofic shifts, synchronizing subshifts, Dyck shifts, $\beta$-shifts, substitution minimal shifts and so on.
In this paper,  we will introduce a notion of normal subshifts.
A subshift $\Lambda$ is said to be {\it normal}\/
if it is  a $\lambda$-synchronizing subshift  with its infinite cardinality as a set. 
 The class of normal subshifts is closed under topological conjugacy, 
and consists of irreducible $\lambda$-synchronizing subshifts 
excluded trivial subshifts. 
An important property of $\lambda$-synchronization is to have a minimal $\lambda$-graph system as its presentation.  
A $\lambda$-graph system $\frak L$ introduced in \cite{MaDocMath1999}
is a generalization of finite labeled graph.
It presents a subshift.
Any $\lambda$-graph system presents a subshift, conversely 
any subshift can be presented by a $\lambda$-graph system in a canonical way. 
 The $\lambda$-graph system from a subshift in the canonical construction 
is called the canonical $\lambda$-graph system for the subshift.
Not only the canonical $\lambda$-graph system for the subshift,
there are many $\lambda$-graph systems in general that present the subshift.
The canonical $\lambda$-graph system corresponds to its left Krieger cover graph.
We in fact see that the  canonical $\lambda$-graph system for a sofic shift
is the $\lambda$-graph system associated to the  left Krieger cover graph.
Hence the canonical $\lambda$-graph system in general does not have certain irreducibility
unless the subshift is an irreducible  shift of finite type.
An irreducible sofic shift has an irreducible minimal presentation as a labeled graph.
The presentation is called the left (or right) Fischer cover graph. 
It is an irreducible ergodic component of its left Krieger cover graph.
To catch Fisher cover analogue of general subshifts, 
we introduced in \cite{KMAAA2013} the notion of $\lambda$-synchronization of subshift.
It was shown that any $\lambda$-synchronizing subshift 
has a minimal presentation of 
$\lambda$-graph system corresponding to Fisher cover (\cite{KMAAA2013}).
In \cite{MaDocMath2002},
the author introduced a $C^*$-algebra associated with a $\lambda$-graph system
as a generalization of Cuntz--Krieger algebras.
The $C^*$-algebra is written $\OL$ for a $\lambda$-graph system $\frak L$
and has a universal property subject to certain operator relations encoded by structure of the $\lambda$-graph system $\frak L$.
If a $\lambda$-graph system is the canonical $\lambda$-graph system 
${\frak L}^\Lambda$ for a subshift $\Lambda$,
the $C^*$-algebra in general is far from simple, namely has nontrivial ideals, 
unless the subshift is a shift of finite type or special kinds of subshifts,
because the canonical $\lambda$-graph system corresponds to left Krieger cover,
that is not irreducible in general.

On the other hand, 
 if a subshift is normal, that is, $\lambda$-synchronizing,
we may construct a minimal $\lambda$-graph system as its presentation
called the $\lambda$-synchronizing $\lambda$-graph system
written $\LLmin$.
It is called the minimal presentation (see \cite{MaIsrael2013}), 
so that the associated $C^*$-algebra are simple and purely infinite in many cases
(see \cite{MaJAMS2013}). 
 For a normal subshift $\Lambda$, we write the $C^*$-algebra
as $\OLmin$. 
Let us denote by $X_\Lambda$ the associated right one-sided subshift
for a two-sided subshift $\Lambda$.
As in the previous papers \cite{KMAAA2013} and \cite{MaJAMS2013},
the $C^*$-algebra $\OLmin$ has a natural action of the circle group
$\T$ called  gauge action written $\rho^\Lambda$.
The fixed point algebra $\FLmin$ of $\OLmin$ under $\rho^\Lambda$
is an AF-algebra having its diagonal
algebra denoted by $\DLmin$.
The commutative $C^*$-algebra $C(X_\Lambda)$
of continuous functions on the right one-sided subshift $X_\Lambda$ 
is naturally regarded as a subalgebra of $\DLmin$ 
denoted by $\DLam$.
We know that the relative commutant $\DLam^\prime \cap \OLmin$
of  $\DLam$ in $\OLmin$ coincides with
$\DLmin$ (Proposition \ref{prop:relative}).
Hence we have a triplet 
$(\OLmin, \DLam, \rho^\Lambda)$
from a normal subshift $\Lambda$.

In the first half of the paper, we will summarize the $\lambda$-synchronization of subshifts and describe simplicity condition of the $C^*$-algebras
$\OLmin$ so that we have
\begin{theorem}\label{thm:main1}
Let $\Lambda$ be a normal subshift.
If $\Lambda$ is $\lambda$-irreducible, then the $C^*$-algebra
$\OLmin$ is simple.
If in addition $\Lambda$ satisfies  $\lambda$-condition (I), then  
 the $C^*$-algebra
$\OLmin$ is simple and purely infinite.
\end{theorem}
As a corollary,
we have
\begin{corollary}[{Proposition \ref{prop:4.2}}]
Let $\Lambda$ be an irreducible sofic shift such that $\Lambda$ is not of finite set.
The $C^*$-algebra $\OLmin$ is simple, purely infinite.
It is isomorphic to the Cuntz--Krieger algebra for the transition matrix of the
left Fischer cover graph of the sofic shift $\Lambda$.     
\end{corollary}

We will present several examples  of simple purely infinite $C^*$-algebras associated 
with normal subshifts in Section 5.
They are the $C^*$-algebras associated with Dyck shifts, Markov-Dyck shifts, Motzkin shifts and $\beta$-shifts.

In the second half of the paper,
we will study relationship between several kinds of topological conjugacy of normal subshifts  
and structure of the associated $C^*$-algebras.
Let ${\frak L}_1, {\frak L}_2$
be left-resolving $\lambda$-graph systems that present 
the subshifts $\Lambda_1, \Lambda_2$, respectively. 
In \cite{MaDynam2020},
the notion of $({\frak L}_1,{\frak L}_2)$-continuous orbit equivalence
between their one-sided subshifts 
$(X_{\Lambda_1}, \sigma_{\Lambda_1})$ and
$(X_{\Lambda_2}, \sigma_{\Lambda_2})$
was introduced.
The author then proved that 
$(X_{\Lambda_1},\sigma_{\Lambda_1})$
and
$(X_{\Lambda_2},\sigma_{\Lambda_2})$
are 
$({\frak L}_1,{\frak L}_2)$-continuously orbit equivalent 
if and only if 
there exists an isomorphism
$\Phi:{\mathcal{O}}_{{\frak L}_1}\longrightarrow {\mathcal{O}}_{{\frak L}_2}
$ 
of $C^*$-algebras such that 
$
\Phi({\mathcal{D}}_{{\Lambda}_1})={\mathcal{D}}_{{\Lambda}_2}
$
where 
${\mathcal{D}}_{{\Lambda}_i}$ 
is a canonical commutative $C^*$-subalgebra of 
${\mathcal{O}}_{{\frak L}_i}$
isomorphic to $C(X_{\Lambda_i})$ for $ i=1,2$.
We will see that, under the condition that 
$(X_{\Lambda_1},\sigma_{\Lambda_1})$
and
$(X_{\Lambda_2},\sigma_{\Lambda_2})$
are 
$({\frak L}_1,{\frak L}_2)$-continuously orbit equivalent,
if $\Lambda_1$ is a normal subshift 
and
${\frak L}_1$ is its minimal presentation,
then 
$\Lambda_2$ is a normal subshift 
and
${\frak L}_2$ is its minimal presentation
(Lemma \ref{lem:coe6.2}).
We then define 
the one-sided subshifts 
$(X_{\Lambda_1}, \sigma_{\Lambda_1})$ and
$(X_{\Lambda_2}, \sigma_{\Lambda_2})$
for normal subshifts $\Lambda_1$ and $\Lambda_2$ 
to be {\it continuously orbit equivalent}\/ if they are 
$(\LLamonemin,\LLamtwomin)$-continuously orbit equivalent
(Definition \ref{def:coefornormal}).
We thus know that for normal subshifts
$\Lambda_1$ and $\Lambda_2$,
their 
one-sided subshifts 
$(X_{\Lambda_1},\sigma_{\Lambda_1})$
and
$(X_{\Lambda_2},\sigma_{\Lambda_2})$
are continuously orbit equivalent
if and only if there exists an isomorphism
$\Phi: \OLamonemin
\longrightarrow \OLamtwomin
$ of $C^*$-algebras
such that 
$\Phi({\mathcal{D}}_{\Lambda_1})={\mathcal{D}}_{\Lambda_2}
$ 
(Proposition \ref{prop:onesidednormalcoe}).

%%%%%%%%%%%%%%%%%%%%%%%%%%%%%%%%%%%%%%%%%%%%%%%

In \cite{MaDynam2020}, the author also introduced 
the notion of $({\frak L}_1,{\frak L}_2)$-eventual conjugacy 
between their one-sided subshifts 
$(X_{\Lambda_1}, \sigma_{\Lambda_1})$ and
$(X_{\Lambda_2}, \sigma_{\Lambda_2})$
and proved that 
$(X_{\Lambda_1},\sigma_{\Lambda_1})$
and
$(X_{\Lambda_2},\sigma_{\Lambda_2})$
are 
$({\frak L}_1,{\frak L}_2)$-eventually conjugate
if and only if 
there exists an isomorphism
$\Phi:{\mathcal{O}}_{{\frak L}_1}\longrightarrow {\mathcal{O}}_{{\frak L}_2}
$ 
of $C^*$-algebras such that 
\begin{equation*}
%\Phi({\mathcal{D}}_{{\frak L}_1})={\mathcal{D}}_{{\frak L}_2},\qquad
\Phi({\mathcal{D}}_{{\Lambda}_1})={\mathcal{D}}_{{\Lambda}_2}
\quad
\text{ and }
\quad
\Phi \circ \rho_t^{{\frak L}_1} =\rho_t^{{\frak L}_2}\circ \Phi,
\quad
t \in \T,
\end{equation*}
where 
${\mathcal{D}}_{{\Lambda}_i}$ 
is a canonical commutative $C^*$-subalgebra of 
${\mathcal{O}}_{{\frak L}_i}$
isomorphic to $C(X_{\Lambda_i})$,
and 
$\rho_t^{{\frak L}_i}$ is the gauge action on
${\mathcal{O}}_{{\frak L}_i}$ for $ i=1,2$.

Let us denote by $\K$ the $C^*$-algebra of compact operators on the separable 
infinite dimensional Hilbert space $\ell^2(\N)$ and
$\C$ its commutative $C^*$-subalgebra of diagonal operators.
For two-sided topological conjugacy, the notion of  
$({\frak L}_1,{\frak L}_2)$-conjugacy
between two-sided subshifts 
$({\Lambda_1}, \sigma_{\Lambda_1}),$
$({\Lambda_2}, \sigma_{\Lambda_2})$
were introduced in 
\cite{MaYMJ2010} and \cite{MaDynam2020}. 
It was proved in \cite{MaDynam2020} that 
$(\Lambda_1, \sigma_1)$ and 
$(\Lambda_2, \sigma_2)$
are $({\frak L}_1,{\frak L}_2)$-conjugate
if and only if   
there exists an isomorphism
$\widetilde{\Phi}:{\mathcal{O}}_{{\frak L}_1}\otimes\K
\longrightarrow 
{\mathcal{O}}_{{\frak L}_2}\otimes\K
$ of $C^*$-algebras
such that 
\begin{equation*}
%\widetilde{\Phi}({\mathcal{D}}_{{\frak L}_1}\otimes\C)={\mathcal{D}}_{{\frak L}_2}\otimes\C, \qquad
\widetilde{\Phi}({\mathcal{D}}_{{\Lambda}_1}\otimes\C)
={\mathcal{D}}_{{\Lambda}_2}\otimes\C, \qquad
\widetilde{\Phi}\circ (\rho^{{\frak L}_1}_t\otimes\id)
 =(\rho^{{\frak L}_2}_t\otimes\id)\circ \widetilde{\Phi},\quad t \in \T.
\end{equation*} 

In \cite{KMAAA2013}, it was proved that $\lambda$-synchronization 
is invariant under topological conjugacy of two-sided subshifts.
Hence if a normal subshift $\Lambda_1$
is topologically conjugate to another  subshift $\Lambda_2$, then
 $\Lambda_2$ is normal.
We will first show the following theorems concerning one-sided conjugacies. 
\begin{theorem}\label{thm:main1.3}
Let $\Lambda_1$ and $\Lambda_2$ be normal subshifts.
If their one-sided subshifts 
$(X_{\Lambda_1},\sigma_{\Lambda_1})$ 
and 
$(X_{\Lambda_2},\sigma_{\Lambda_2})$ 
are topologically conjugate, then there exists an isomorphism
$\Phi:\OLamonemin\longrightarrow\OLamtwomin$ of $C^*$-algebras such that 
$\Phi({\mathcal{D}}_{\Lambda_1}) ={\mathcal{D}}_{\Lambda_2}$
and
$\Phi\circ\rho^{\Lambda_1}_t = \rho^{\Lambda_2}_t\circ\Phi, t \in \T$.
\end{theorem}
Theorem \ref{thm:main1.3} is a generalization of Cuntz--Krieger's theorem
\cite[Proposition 2.17]{CK}. 
Related results are seen in \cite{BC2017},\cite{BC2019},\cite{MaPAMS2017}, etc.)
 
The following theorem is a generalization of the results for 
irreducible topological Markov shifts in \cite{MaPAMS2017} (cf.  \cite{BC2017},\cite{BC2019}).

\begin{theorem}\label{thm:main1.4}
Let $\Lambda_1$ and $\Lambda_2$ be normal subshifts.
Their one-sided subshifts 
$(X_{\Lambda_1},\sigma_{\Lambda_1})$ 
and 
$(X_{\Lambda_2},\sigma_{\Lambda_2})$ 
are eventually  conjugate if and only if there exists an isomorphism
$\Phi:\OLamonemin\longrightarrow\OLamtwomin$ of $C^*$-algebras such that 
$\Phi({\mathcal{D}}_{\Lambda_1}) ={\mathcal{D}}_{\Lambda_2}$
and
$\Phi\circ\rho^{\Lambda_1}_t = \rho^{\Lambda_2}_t\circ\Phi, t \in \T$.
\end{theorem}
The if part of Theorem \ref{thm:main1.4} follows from a result in 
\cite{MaDynam2020}.
The proof of its only if part is a main body in the second half of this paper.
To prove the only if part, we provide an auxiliay subshft written $\Lambda'_2$
whose one-sided subshift $X_{\Lambda'_2}$ is topologically conjugate to $X_{\Lambda_1}$.
We will then prove that there exists an isomorphism of $C^*$-algebras
$\Phi_2:
{\mathcal{O}}_{{\Lambda'_2}^{\operatorname{min}}}
\longrightarrow
\OLamtwomin
$ 
satisfying 
$\Phi_2({\mathcal{D}}_{\Lambda'_2}) ={\mathcal{D}}_{\Lambda_2}$
and
$\Phi_2\circ\rho^{\Lambda'_2}_t = \rho^{\Lambda_2}_t\circ\Phi_2, t \in \T$,
so that
we will obtain Theorem \ref{thm:main1.4}
by using Theorem \ref{thm:main1.3}.

We will second show the following theorem concerning two-sided conjugacies,
that is a generalization of the case of topological Markov shifts 
proved by Cuntz--Krieger \cite{CK} and Carlsen--Rout \cite{CR}. 
\begin{theorem}\label{thm:main1.5}
Let $\Lambda_1$ and $\Lambda_2$ be normal subshifts.
The two-sided subshifts 
$({\Lambda_1},\sigma_{\Lambda_1})$ 
and 
$({\Lambda_2},\sigma_{\Lambda_2})$ 
are topologically   conjugate if and only if there exists an isomorphism
$\widetilde{\Phi}:\OLamonemin\otimes\K\longrightarrow\OLamtwomin\otimes\K$ of $C^*$-algebras such that 
$\widetilde{\Phi}({\mathcal{D}}_{\Lambda_1}\otimes\C) ={\mathcal{D}}_{\Lambda_2}\otimes\C$
and
$\widetilde{\Phi}\circ(\rho^{\Lambda_1}_t\otimes\id) 
= (\rho^{\Lambda_2}_t\otimes\id) \circ \widetilde{\Phi}, t \in \T$.
\end{theorem}
The $C^*$-algebraic characterizations of eventual conjugacy and topological conjugacy
appeared in Theorem \ref{thm:main1.4} and Theorem \ref{thm:main1.5}
are rephrased in terms of the associated groupoids as seen in 
 \cite[Theorem 1.3]{MaDynam2020}
and \cite[Theorem 1.4]{MaDynam2020}, respetively.

We may apply the above theorems to irreducible sofic shifts.
Let $\Lambda$ be an irreducible sofic shift such that $\Lambda$ is infinite.
Let $G_\Lambda^F$
be its left Fischer cover graph, that is the unique left-resolving irreducible
minimal finite labeled graph that presents $\Lambda$ (\cite{Fischer}, cf. \cite{LM}).
%Let $\LLmin$ be its minimal presentation of $\lambda$-graph system for the sofic subshift $\Lambda$.
Then the $C^*$-algebra
$\OLmin$ is a simple purely infinite $C^*$-algebra such that 
$\OLmin$ is isomorphic to the Cuntz--Krieger algebra 
$ {\mathcal{O}}_{\widehat{A}}
$
for the transition matrix of the topological Markov shift 
defined by the Fischer cover 
$G_\Lambda^F$ (Proposition \ref{prop:4.2}).
By  Proposition \ref{prop:onesidednormalcoe}, Theorem \ref{thm:main1.4} and Theorem \ref{thm:main1.5},
we have the following result.
\begin{corollary}%[{Theorem \ref{thm:soficconj}}]
Let $\Lambda_1$ and $\Lambda_2$ be two irredicible sofic shifts such that $\Lambda_i, i=1,2$ are infinite.
\begin{enumerate}
\renewcommand{\theenumi}{\roman{enumi}}
\renewcommand{\labelenumi}{\textup{(\theenumi)}}
\item
Their one-sided sofic shifts
$(X_{\Lambda_1},\sigma_{\Lambda_1})$ 
and 
$(X_{\Lambda_2},\sigma_{\Lambda_2})$ 
are continuously orbit equivalent if and only if there exists an isomorphism
$\Phi:\OLamonemin\longrightarrow\OLamtwomin$ of simple $C^*$-algebras such that 
$\Phi({\mathcal{D}}_{\Lambda_1}) ={\mathcal{D}}_{\Lambda_2}$.
\item
Their one-sided sofic shifts
$(X_{\Lambda_1},\sigma_{\Lambda_1})$ 
and 
$(X_{\Lambda_2},\sigma_{\Lambda_2})$ 
are eventually  conjugate if and only if there exists an isomorphism
$\Phi:\OLamonemin\longrightarrow\OLamtwomin$ of simple $C^*$-algebras such that 
$\Phi({\mathcal{D}}_{\Lambda_1}) ={\mathcal{D}}_{\Lambda_2}$
and
$\Phi\circ\rho^{\Lambda_1}_t = \rho^{\Lambda_2}_t\circ\Phi, t \in \T$.
\item
Their two-sided sofic shifts 
$({\Lambda_1},\sigma_{\Lambda_1})$ 
and 
$({\Lambda_2},\sigma_{\Lambda_2})$ 
are topologically conjugate if and only if there exists an isomorphism
$\widetilde{\Phi}:\OLamonemin\otimes\K\longrightarrow\OLamtwomin\otimes\K$ of simple $C^*$-algebras such that 
$\widetilde{\Phi}({\mathcal{D}}_{\Lambda_1}\otimes\C) ={\mathcal{D}}_{\Lambda_2}\otimes\C$
and
$\widetilde{\Phi}\circ(\rho^{\Lambda_1}_t\otimes\id) 
= (\rho^{\Lambda_2}_t\otimes\id) \circ \widetilde{\Phi}, t \in \T$.
\end{enumerate}
\end{corollary}

\medskip

We have to remark that in a recent paper \cite{BC2019} by Brix--Carlsen,
similar results to the present paper are seen. 
The  $C^*$-algebras treated by Brix--Carlsen are different from
our $C^*$-algebras. In fact, their $C^*$-algebras in \cite{BC2019} are not simple 
in many cases unless the subshifts are irreducible shifts of finite type,    
whereas our $C^*$-algebras in the present paper are simple in many cases
including infinite irreducible sofic shifts.

In what follows,
the set of nonnegative integers and the set of positive integers
are denoted by $\Zp$ and $\N$, respectively.

%\newpage

%%%%%%%%%%%%%%%%%%%%%%%%%%%%%%%%%%%%%%%%%%%%%%%%%%%
%%%%%%%%%%%%%%%%%%%%%%%%%%%%%%%%%%%%%%%%%%%%%%%%%%%%%%%
\section{$\lambda$-synchronization and normal subshifts}
%%%%%%%%%%%%%%%%%%%%%%%%%%%%%%%%%%%%%%%%%%%%%%%%%%%%%%%

{\bf 1. $\lambda$-synchronization of subshifts.}

Let $\Sigma$ be a finite set with its discrete topology.
Denote by $\Sigma^\Z$ (resp. $\Sigma^\N$) the set of bi-infinite  (resp. right one-sided)
sequences of $\Sigma$.
We endow $\Sigma^\Z$ (resp. $\Sigma^\N$) 
 with infinite product topology, so that they are compact Hausdorff spaces.
The shift homeomorphism $\sigma: \Sigma^\Z\longrightarrow \Sigma^\Z$
is defined by $\sigma((x_n)_{n \in \Z}) =(x_{n+1})_{n \in \Z}$.
A continuous surjection
$\sigma:\Sigma^\N \longrightarrow \Sigma^\N$ is similarly defined.
Let $\Lambda\subset\Sigma^\Z$ be a closed $\sigma$-invariant subset, that is,
$\sigma(\Lambda) = \Lambda$. 
We denote the restriction $\sigma|_{\Lambda}$ of $\sigma$ to $\Lambda$ 
by $\sigma_{\Lambda}$.
The topological dynamical system
$(\Lambda,\sigma_\Lambda)$ is called a subshift over alphabet $\Sigma$.
It is often written as $\Lambda$ for short.
Let $X_\Lambda$ be the set of right infinite sequence $(x_n)_{n\in \N}$
of $\Sigma$ such that $(x_n)_{n\in \Z} \in \Lambda$.
The set $X_\Lambda$ is a closed subset of $\Sigma^\N$ such that 
$\sigma(X_\Lambda) = X_\Lambda$.
We similarly denote $\sigma|_{X_{\Lambda}}$ by $\sigma_{\Lambda}$.
The topological dynamical system
$(X_\Lambda,\sigma_\Lambda)$ is called the right one-sided subshift for $\Lambda$.
For an introduction to the theory of subshifts, we refer to
text books of symbolic dynamical systems  \cite{Kitchens}, \cite{LM}.
 For $l \in \Zp$, 
denote  by $B_l(\Lambda)$
the admissible words
$\{(x_1,\dots,x_l) \in \Sigma^l \mid (x_n)_{n\in \Z} \in \Lambda\}$
 of $\Lambda$ with its length $l$.
Denote by 
$B_*(\Lambda)$ the set $\cup_{l=0}^\infty B_l(\Lambda)$
of admissible words of $\Lambda$, where 
$B_0(\Lambda)$ denotes the empty word.
The length $m$ of a word 
$\mu =(\mu_1,\dots,\mu_m)$
is denoted by $|\mu|$.
For two words
$\mu =(\mu_1,\dots,\mu_m), \nu =(\nu_1,\dots,\nu_n) \in B_*(\Lambda)$
denote by $\mu\nu$ the concatenation
$(\mu_1,\dots,\mu_m,\nu_1,\dots,\nu_n)$.
For $\mu =(\mu_1,\dots,\mu_m) \in B_*(\Lambda)$ and $x=(x_n)_{n\in \N}\in X_\Lambda$,
we put
$\mu x = (\mu_1,\dots,\mu_m,x_1,x_2,\dots)\in \Sigma^\N$. 
For a word $\mu =(\mu_1,\dots,\mu_m) \in B_m(\Lambda)$,
the cylinder set $U_\mu \subset X_\Lambda$ is defined by
$$
U_\mu = \{(x_n)_{n\in \N} \in X_\Lambda \mid x_1=\mu_1,\dots,x_m =\mu_m \}.
$$
For $x=(x_n)_{n\in \N} \in X_\Lambda$
and $k, l\in \N$ with $k \le l$,
we put 
$x_{[k,l]} = (x_k,\dots,x_l) \in B_{l-k+1}(\Lambda),$
$ x_{[k,l)} = (x_k,\dots,x_{l-1}) \in B_{l-k}(\Lambda)$
and
$x_{[k,\infty)} = (x_k, x_{k+1},\dots) \in X_\Lambda.$

A subshift $\Lambda$ is said to be irreducible
if for any $\mu, \nu \in B_*(\Lambda)$, there exists a word 
$\eta \in B_*(\Lambda)$ such that 
$\mu\eta\nu \in B_*(\Lambda)$ (cf. \cite{LM}).
We note the following lemma.
Although it is well-known, the author has not been able to find a suitable reference,
so that the proof is given.
\begin{lemma}[{cf. \cite[p. 142]{Kurka}}]\label{lem:Cantor}
If a subshift $\Lambda$ is irreducible and the cardinality of $\Lambda$ is infinite,
then the  subshift $\Lambda$ and its right one-sided subshift $X_\Lambda$ 
are both homeomorphic to a Cantor set.
\end{lemma}
\begin{proof}
We will show that $X_\Lambda$ does not have any isolated point. 
Since $\Lambda$ is irreducible,
one may find a point $z \in X_\Lambda$ such that 
its orbit $\{\sigma_\Lambda^n(z) \mid n \in \Zp\}$ is dense in $X_\Lambda$.
For any point $x \in X_\Lambda$ and word $\mu\in B_m(\Lambda)$
with $x \in U_\mu$,
there exists $n_1 \in \Zp$ such that 
$\sigma_\Lambda^{n_1}(z) \in U_\mu$
As $\{\sigma_\Lambda^n(\sigma_\Lambda^{n_1}(z)) \mid n \in \N \}$ 
is also dense in $X_\Lambda$,
there exists $n_2 \in \N$ such that 
$\sigma_\Lambda^{n_2}(\sigma_\Lambda^{n_1}(z)) \in U_\mu$.
If 
$\sigma_\Lambda^{n_2}(\sigma_\Lambda^{n_1}(z)) =\sigma_\Lambda^{n_1}(z),
$
then 
$\sigma_\Lambda^{n_1}(z)
$ is periodic, so that  $\{\sigma_\Lambda^n(z) \mid n \in \Zp\}$ 
is finite, and $X_\Lambda$ becomes a finite set, a contradiction.
Therefore   $\sigma_\Lambda^{n_2 +n_1}(z)\ne \sigma_\Lambda^{n_1}(z)$,
and hence $U_\mu$ contains two distinct points
$\sigma_\Lambda^{n_2 +n_1}(z), \sigma_\Lambda^{n_1}(z)$
so that $x$ is not isolated.
As $X_\Lambda$ is totally disconnected compact metric space,
it is homeomorphic to a Cantor set.
Similarly we know that $\Lambda$ does not have any isolated points.
\end{proof}

We define predecessor sets and follower sets of a word 
$\mu\in B_m(\Lambda)$ as follows:
\begin{equation*}
\Gamma_l^-(\mu) =\{\nu \in B_l(\Lambda) \mid \nu\mu\in B_{l+m}(\Lambda)\},
\qquad
\Gamma_l^+(\mu) =\{\nu \in B_l(\Lambda) \mid \mu\nu\in B_{l+m}(\Lambda)\}
\end{equation*}
and
$\Gamma_*^-(\mu) =\bigcup_{l=0}^\infty \Gamma_l^-(\mu),
\Gamma_*^+(\mu) =\bigcup_{l=0}^\infty \Gamma_l^+(\mu).
$

Following \cite{KMAAA2013}, \cite{MaIsrael2013}, \cite{MaJAMS2013},
a word $\mu\in B_*(\Lambda)$ for $l \in \Zp$ is said to be
$l$-{\it synchronizing}\/
if the equality $\Gamma_l^-(\mu) = \Gamma_l^-(\mu\omega)$
holds for all $\omega \in \Gamma_*^+(\mu)$.
Let us denote by $S_l(\Lambda)$ the set of $l$-synchronizing words of $\Lambda$.

\begin{definition}[{\cite{KMAAA2013}, \cite{MaIsrael2013}, \cite{MaJAMS2013}}] \label{def:synchro} 
An irreducible subshift $\Lambda$ is said to be $\lambda$-synchronizing if
for any word $\eta \in B_l(\Lambda)$ and positive integer $k >l$,
there exists $\nu \in S_k(\Lambda)$ such that 
$\eta \nu \in S_{k-l}(\Lambda)$.
\end{definition}
As in \cite{KMAAA2013}, \cite{MaIsrael2013} and \cite{MaJAMS2013}, 
the following subshifts are $\lambda$-synchronizing:

\medskip

$\bullet$
irreducible shifts of finite type, 

$\bullet$
irreducible sofic shifts, 

$\bullet$
synchronizing systems, 

$\bullet$
Dyck shifts,

$\bullet$
Motzkin shifts, 

$\bullet$
irreducible Markov-Dyck shifts, 

$\bullet$
primitive substitution subshifts, 

$\bullet$
$\beta$-shifts for every $\beta>1$, etc.

\medskip

There is an example of a coded system that is not $\lambda$-synchronizing
(cf. \cite{KMAAA2013}).

Following \cite{MaIsrael2013}, 
two admissible words $\mu, \nu \in B_*(\Lambda)$ are said to be $l$-past equivalent
if $\Gamma_l^-(\mu) = \Gamma_l^-(\nu)$. In this case we write 
$\mu \underset{l}{\sim} \nu$. 
\begin{definition}
A $\lambda$-synchronizing subshift $\Lambda$ is said to be $\lambda$-{\it transitive}\/
if for any two admissible words $\mu, \nu \in S_l(\Lambda)$, there exists 
$k_{\mu,\nu} \in \N$ such that for any $\eta \in S_{l +k_{\mu,\nu}}(\Lambda)$
satisfying $\nu \underset{l}{\sim} \eta$, there exists $\xi \in B_{k_{\mu,\nu}}(\Lambda)$
such that $\mu \underset{l}{\sim} \xi\eta$. 
\end{definition}
In \cite{KMAAA2013}, the term "synchronized irreducible" was used for the above $\lambda$-transitivity.
\begin{definition}
A subshift $\Lambda$ is said to be {\it normal}\/ if it is $\lambda$-synchronizing and 
its cardinality $| \Lambda|$ is not finite. 
\end{definition}
Hence the class of normal subshifts contains a lot of important nontrivial subshifts.

\medskip

\noindent
{\bf 2.  $\lambda$-graph systems}

A $\lambda$-graph system $\frak L$ over alphabet $\Sigma$ 
consists of a quadruple
$(V,E,\lambda,\iota),$ 
where 
$(V,E,\lambda)$
is a labeled Bratteli diagram with 
its vertex set
$V=\cup_{l \in \Zp} V_l$,
edge set 
$E = \cup_{l \in \Zp} E_{l,l+1}$
and labeling map
$\lambda: E \longrightarrow \Sigma$.
For an edge $e \in E_{l,l+1}$, 
denote by $s(e)\in V_l$ and $ t(e) \in V_{l+1}$ 
its source vertex and terminal vertex, respectively. 
The additional object $\iota$ is a surjection
$\iota(=\iota_{l,l+1}): V_{l+1}\longrightarrow V_l$
for each $\l \in \Zp$.  
The quadruple
$(V,E,\lambda,\iota)$
is needed to satisfy the following local property. 
Put for 
$u \in V_{l-1}$ and $v \in V_{l+1},$
\begin{align*}
E^{\iota}_{l,l+1}(u, v)
& = \{ e \in E_{l,l+1} \ \mid \ t(e) = v, \, \iota(s(e)) = u \},\\
E_{\iota}^{l-1,l}(u, v)
& = \{ e \in E_{l-1,l} \ \mid \ s(e) = u, \, t(e) = \iota(v) \}.
\end{align*}
The local property requires a bijective correspondence preserving their labels between 
$E^{\iota}_{l,l+1}(u, v)$
and
$E_{\iota}^{l-1,l}(u, v)$
for every  pair of vertices $u, v$.
For $k < l$, we put
$$
E_{k,l} =\{(e_1, \dots, e_{l-k}) \in E_{k,k+1}\times\cdots\times E_{l-1,l} \mid
t(e_i) =s(e_{i+1}), i=1,\dots, l-k-1\}.
$$ 
A member of $E_{k,l}$ is called a labeled path.
For $\gamma = (e_1, \dots, e_{l-k}) \in E_{k,l}$, we put
$s(\gamma):= s(e) \in V_k$, $t(\gamma) := t(e_{l-k}) \in V_l$
and
$\lambda(\gamma) := (\lambda(e_1), \dots, \lambda(e_{l-k})) \in \Sigma^{l-k}.
$
For $v\in V_l$, we put
\begin{equation}
\Gamma_l^-(v) = \{(\lambda(e_1), \dots, \lambda(e_l)) \in \Sigma^{l} \mid 
(e_1, \dots, e_l) \in E_{0,l}, t(e_l) = v \}. \label{eq:Gammalv}
\end{equation}

A $\lambda$-graph system $\frak L$ is said to be {\it predecessor-separated}
if $\Gamma_l^-(v) \ne
\Gamma_l^-(u)$ for every distinct pair $u,v \in V_l$.  
A $\lambda$-graph system $\frak L$ is said to be {\it left-resolving}
 if  $e, f \in E_{l,l+1}$ satisfy $t(e) = t(f), \lambda(e) = \lambda(f)$, 
then $e=f$.

Let us denote by $\Lambda_{\frak L}$ the two-sided subshift
over $\Sigma$, whose admissible words $B_*(\Lambda_{\frak L})$
are defined by the set of words appearing in the finite labeled sequences 
in the labeled Bratteli diagram $(V,E,\lambda)$ of the $\lambda$-graph system 
${\frak L} =(V,E,\lambda,\iota)$.
We say that a subshift $\Lambda$ is presented by a  $\lambda$-graph system 
${\frak L}$ or ${\frak L}$ presents $\Lambda$
 if $ \Lambda = \Lambda_{\frak L}.$

Let $\mathcal{G} =(\mathcal{V}, \mathcal{E},\lambda)$ be 
a predecessor-separated left-resolving finite labeled graph over alphabet 
$\Sigma$ with finite vertex set $\mathcal{V}$, finite edge set $\mathcal{E}$
and labeling $\lambda: \mathcal{E} \longrightarrow \Sigma$.
It naturally gives rise to a $\lambda$-graph system
${\frak L}_{\mathcal{G}}$ by setting $V_l = \mathcal{V}, E_{l,l+1} =\mathcal{E}$
for all $l \in \Zp$ and $\iota = \id$.
The presented subshift $\Lambda_{{\frak L}_{\mathcal{G}}}$
by the $\lambda$-graph system ${\frak L}_{\mathcal{G}}$
is noting but the sofic shift $\Lambda_{\mathcal{G}}$
presented by the finite labeled graph $\mathcal{G}$.
Detail studies of $\lambda$-graph system are in \cite{MaDocMath1999}.
\begin{definition}[{\cite{MaJAMS2013}}] \label{def:2.5}
 Let ${\frak L} =(V,E,\lambda,\iota)$ 
be a $\lambda$-graph system over $\Sigma$.
 \begin{enumerate}
\renewcommand{\theenumi}{\roman{enumi}}
\renewcommand{\labelenumi}{\textup{(\theenumi)}}
\item
${\frak L}$ is said to be $\iota$-{\it irreducible}\/ if  for any two vertices 
$u,v \in V_l$ and a labeled path $\gamma$ leaving $u$,
there exist labeled paths $\eta$ and $\gamma'$ of length $n$ 
such that $\eta$ leaves $v$ and satisfies $\iota^n(t(\eta)) = u$,
and $\gamma'$ leaves $t(\eta)$ satisfies
$\iota^n(t(\gamma')) = t(\gamma)$ and $\lambda(\gamma') = \lambda(\gamma)$ 
\item
${\frak L}$ is said to be $\lambda$-{\it irreducible}\/ if  
for any ordered pair $u,v \in V_l$ of vertices,
there exists $L(u,v) \in \N$ such that 
for any vertex $w \in V_{l+L(u,v)}$ satisfying
$\iota^{L(u,v)}(w) = u$, there exists a labeled path
$\gamma $ such that $s(\gamma) = v$ and $t(\gamma) = w$.
\end{enumerate}
\end{definition}

\begin{lemma}\label{lem:irreducible}
Let ${\frak L} =(V,E,\lambda,\iota)$ 
be a $\lambda$-graph system that presents a subshift  $\Lambda$.
Consider the following three conditions.
 \begin{enumerate}
\renewcommand{\theenumi}{\roman{enumi}}
\renewcommand{\labelenumi}{\textup{(\theenumi)}}
\item ${\frak L}$ is  $\lambda$-irreducible.
\item ${\frak L}$ is  $\iota$-irreducible.
\item $\Lambda$ is irreducible.
\end{enumerate}
Then we have 
(i) $\Longrightarrow$  (ii) $\Longrightarrow$ (iii). 
\end{lemma}
\begin{proof}
(i) $\Longrightarrow$ (ii):
Assume that ${\frak L}$ is $\lambda$-irreducible.
Let $u,v \in V_l$ be two vertices 
 and $\gamma$ a labeled path  leaving $u$,
Take $L(u,v) \in \N$ satisfying the $\lambda$-irreducibility condition
in Definition \ref{def:2.5} (ii). 
Let $k$ denote the length of the path $\gamma$
and
$u_{\gamma} =t(\gamma)\in V_{l+k}$.
Take $u' \in V_{l+k+L(u,v)}$
such that $\iota^{L(u,v)}(u') = u_\gamma$. 
By the local property of $\lambda$-graph system,
one may find $w \in V_{l+L(u,v)}$ and
a labeled path $\gamma' $ such that 
$$
\iota^{L(u,v)}(w) = u, \qquad s(\gamma') = w,\qquad t(\gamma') = u'.
$$
By the $\lambda$-irreducibility, there exists a labeled path
$\eta$ such that 
$s(\eta) = v, t(\eta) = w$.

(ii) $\Longrightarrow$ (iii): The assertion comes from \cite[Lemma 3.5]{MaIsrael2013}.
\end{proof}
\begin{remark}
\begin{enumerate}
\renewcommand{\theenumi}{\roman{enumi}}
\renewcommand{\labelenumi}{\textup{(\theenumi)}}
\item If $\frak L$ is a $\lambda$-graph system ${\frak L}_{\mathcal{G}}$
associated to a left-resolving finite labeled graph $\mathcal{G}$,
then the presented subshift $\Lambda_{{\frak L}_{\mathcal{G}}}$ 
by ${\frak L}_{\mathcal{G}}$ is a sofic shift defined by $\mathcal{G}$.
It is easy to see that for the $\lambda$-graph system
${\frak L}_{\mathcal{G}}$,  all of the conditions (i), (ii) and (iii) in Lemma \ref{lem:irreducible}
are mutually equivalent.
\item Let $\Lambda_C$ be the coded system defined by the code
$C =\{ a^n b^n \mid n=1,2,\dots \}$ for alphabet $\Sigma =\{a, b\}$
(see \cite{BH}).
Then the subshift $\Lambda_C$ has a synchronizing word
$\omega = aba$, 
so that it is an irreducible  synchronizing subshift.
Hence $\Lambda_C$ is a $\lambda$-synchronizing
(\cite{KMAAA2013}). 
Let ${\frak L}^{\lambda(\Lambda_C)}$ 
be its $\lambda$-synchronizing $\lambda$-graph system as in \cite{MaIsrael2013}. 
By \cite[Lemma 3.6]{MaIsrael2013}, irreducibility of $\Lambda_C$ implies
$\iota$-irreducibility, so that ${\frak L}^{\lambda(\Lambda_C)}$ is $\iota$-irreducible.
However, it is not difficult to see that 
${\frak L}^{\lambda(\Lambda_C)}$  is not $\lambda$-irreducible.
Hence there is an example of $\lambda$-graph system  such that the implication
(ii) $\Longrightarrow$ (i) above does not hold. 
\item Let $\Lambda_{\operatorname{ev}}$ be the even shift,
that is defined to be  a sofic shift over $\{0, 1\}$ whose admissible words 
are $ 1\overbrace{0\cdots 0}^{\text{even}}1$.
Let 
${\frak L}^{\Lambda_{\operatorname{ev}}}$ be the canonical $\lambda$-graph system for
$\Lambda^{\operatorname{ev}}$ (see \cite{MaDocMath1999}).
The subshift $\Lambda_{\operatorname{ev}}$ is irreducible, 
whereas  ${\frak L}^\Lambda_{\operatorname{ev}}$ is not $\iota$-irreducible.
Hence there is an example of $\lambda$-graph system  such that the implication
(iii) $\Longrightarrow$ (ii) above does not hold. 
\end{enumerate}
\end{remark}

\medskip

\noindent
{\bf 3.  $\lambda$-synchronizing $\lambda$-graph systems}

Let ${\frak L} =(V,E,\lambda,\iota)$ 
be a $\lambda$-graph system that presents a subshift  $\Lambda$.
Let
$v \in V_l$ and $\mu \in B_m(\Lambda), m\in \N$.
Following \cite{MaJAMS2013}, we say that 
$v$ {\it launches}\/ $\mu$ if the following two conditions 
are both satisfied:
\begin{enumerate}
\renewcommand{\theenumi}{\roman{enumi}}
\renewcommand{\labelenumi}{\textup{(\theenumi)}}
\item
There exists a labeled path $\gamma \in E_{l,l+m}$ such that 
$s(\gamma) = v, \lambda(\gamma) = \mu.$
\item There are no other vertices in $V_l$ than $v$ for $\mu$ leaving.
\end{enumerate}
The vertex $v$ is called the launching vertex for $\mu$.
\begin{definition}[{\cite{MaJAMS2013}}]
A $\lambda$-graph system
 ${\frak L} =(V,E,\lambda,\iota)$ is said to be $\lambda$-{\it synchronizing}\/
if any vertex of $V$ is a launching vertex for some word of $\Lambda$.
\end{definition}
A $\lambda$-synchronizing $\lambda$-graph system is $\iota$-irreducible
if and only if the presented subshift $\Lambda$ is irreducible (\cite[Proposition 3.7]{MaJAMS2013}).
It was shown that if $\frak L$ is $\iota$-irreducible and $\lambda$-synchronizing,
then the presented subshift $\Lambda$ is $\lambda$-synchronizing.
Conversely, as in \cite{MaJAMS2013}, one may construct 
a left-resolving, predecessor-separated $\iota$-irreducible
$\lambda$-synchronizing $\lambda$-graph system from 
a $\lambda$-synchronizing subshift $\Lambda$.
We briefly review its construction.
Let $\Lambda$ be a $\lambda$-synchronizing subshift. 
Recall that $S_l(\Lambda)$ denotes the set of $l$-synchronizing words of $\Lambda$.
Denote by $V^{\lambda(\Lambda)}_l$ the set of $l$-past equivalence classes of 
$S_l(\Lambda)$,
where $V_0^{\lambda(\Lambda)} =\{v_0\}$ a singleton.
Let us denote by $[\mu]_l$ the equivalence class of $\mu \in S_l(\Lambda)$.
For $\nu \in S_{l+1}(\Lambda)$ and $\alpha \in \Gamma_1^-(\nu)$,
an edge from $[\alpha\nu]_l \in V_l^{\lambda(\Lambda)}$ to 
 $[\nu]_{l+1} \in V_{l+1}^{\lambda(\Lambda)}$ 
with its label $\alpha$ is defined.
The set of such edges is denoted by $E_{l,l+1}^{\lambda(\Lambda)}$.
The labeling map from $E_{l,l+1}^{\lambda(\Lambda)}$ to $\Sigma$
is denoted by $\lambda^{\lambda(\Lambda)}$.
As $S_{l+1}(\Lambda) \subset S_l(\Lambda)$,
we have a natural map 
$\iota^{\lambda(\Lambda)}: 
[\nu]_{l+1} \in V_{l+1}^{\lambda(\Lambda)}\longrightarrow 
[\nu]_{l} \in V_{l}^{\lambda(\Lambda)}.
$
The quadruplet 
$(V^{\lambda(\Lambda)}, E^{\lambda(\Lambda)}, 
\lambda^{\lambda(\Lambda)}, \iota^{\lambda(\Lambda)})
$ 
defines 
a left-resolving, predecessor-separated, $\iota$-irreducible $\lambda$-graph system that presents the subshift $\Lambda$
(\cite[Proposition 3.2]{MaJAMS2013}).
The $\lambda$-graph system was denoted by ${\frak L}^{\lambda(\Lambda)}$
in \cite[Proposition 3.2]{MaJAMS2013}
and called the canonical $\lambda$-synchronizing $\lambda$-graph system
for $\Lambda$.
The following proposition was proved in \cite[Theorem 3.9]{MaJAMS2013}.
\begin{proposition}[{\cite[Theorem 3.9]{MaJAMS2013}}]
Let  $\Lambda$ be a $\lambda$-synchronizing subshift.
Then there uniquely exists a left-resolving, predecessor-separated, $\iota$-irreducible,
$\lambda$-synchronizing  
$\lambda$-graph system that presents the subshift $\Lambda$. 
The unique $\lambda$-synchronizing $\lambda$-graph system is 
the canonical $\lambda$-synchronizing $\lambda$-graph system 
${\frak L}^{\lambda(\Lambda)}$
for $\Lambda$.
\end{proposition}
\begin{lemma}\label{lem:2.10}
Let  $\Lambda$ be a $\lambda$-synchronizing subshift.
\begin{enumerate}
\renewcommand{\theenumi}{\roman{enumi}}
\renewcommand{\labelenumi}{\textup{(\theenumi)}}
\item
$\Lambda$ is irreducible 
if and only if ${\frak L}^{\lambda(\Lambda)}$ is $\iota$-irreducible.
\item
$\Lambda$ is $\lambda$-transitive
if and only if ${\frak L}^{\lambda(\Lambda)}$ is $\lambda$-irreducible.
\end{enumerate}
\end{lemma}
\begin{proof}
(i) The assertion comes from \cite[Proposition 3.7]{MaJAMS2013}.

(ii)
The equivalence between $\lambda$-transitivity of $\Lambda$
and $\lambda$-irreducibility of ${\frak L}^{\lambda(\Lambda)}$ is direct by definition.
\end{proof}
\begin{definition}[{\cite{MaJAMS2013}}]
A $\lambda$-graph system
 ${\frak L}$ is said to be {\it minimal} if ${\frak L}$ 
has no proper $\lambda$-graph subsystem of ${\frak L}$.
\end{definition}
It was proved that for a $\lambda$-synchronizing subshift $\Lambda$,
the canonical $\lambda$-synchronizing $\lambda$-graph system 
${\frak L}^{\lambda(\Lambda)}$ is minimal.

In what follows, 
for a $\lambda$-synchronizing subshift $\Lambda$,
the canonical $\lambda$-synchronizing $\lambda$-graph system 
${\frak L}^{\lambda(\Lambda)}$ is denoted by
$\LLmin$.
Recall that a subshift $\Lambda$ is said to be {\it normal}\/ if
it is $\lambda$-synchronizing and its cardinality $|\Lambda|$ is not finite
as a set. 
We call the $\lambda$-graph system
$\LLmin$ for a normal subshift $\Lambda$ the {\it minimal presentation of}\/ 
of a normal subshift $\Lambda$.
We often write
$\LLmin =(V^{\min}, E^{\min},\lambda^{\min},\iota^{\min})$
or 
$(V^{\Lambda^\min}, E^{\Lambda^\min},\lambda^{\Lambda^\min},\iota^{\Lambda^\min})$
%{{ {{\frak L}_\Lambda^{\operatorname{min}}} }}

\medskip

\noindent
{\bf 4.  Condition (I) for $\lambda$-graph systems}

Let ${\frak L}$ be a $\lambda$-graph system over $\Sigma$
and $\Lambda$ the presented subshift $\Lambda_{\frak L}$.
The condition (I) for a $\lambda$-graph system was introduced in
\cite{MaDocMath2002} that yields uniqueness of certain operator relations
of canonical generators of the associated $C^*$-algebra $\OL$.
\begin{definition}\label{def:cindition(I)}
A  $\lambda$-graph system ${\frak L}$ is said to satisfy {\it condition (I)} 
if for any vertex $v \in V_l$, 
the follower set $\Gamma_\infty^+(v)$ of $v$ defined by
$$
\Gamma_\infty^+(v) :=
\{(\lambda(e_1),\lambda(e_2),\dots ) \in X_\Lambda \mid 
s(e_1) = v, \, e_i \in E_{l+i-1,l+i}, \, t(e_i) = s(e_{i+1}), i=1,2,\dots \}
$$
contains at least two distinct sequences.
\end{definition}
In \cite[Lemma 5.1]{MaJMSJ1999}, the following lemma is shown for the case of the canonical $\lambda$-graph system 
${\frak L}^\Lambda$ for $\Lambda$.
\begin{lemma}[{cf. \cite[Lemma 5.1]{MaJMSJ1999}}] \label{lem:threecond}
Let  ${\frak L}$ be a left-resolving $\lambda$-graph system.
Consider the following three conditions:
\begin{enumerate}
\renewcommand{\theenumi}{\roman{enumi}}
\renewcommand{\labelenumi}{\textup{(\theenumi)}}
\item
${\frak L}$ satisfies condition (I).
\item For $l \in \Zp$, $v \in V_l$,
$(x_n)_{n\in \N} \in \Gamma^+_\infty(v)$ and  $m \in \N$,
there exists $(y_n)_{n\in \N} \in \Gamma^+_\infty(v)$ such that 
\begin{equation*}
x_j = y_j \text{ for all } j=1,2,\dots,m \text{ and }
x_N \ne y_N \text{ for some } N >m. 
\end{equation*} 
\item For $k,l\in \N$ with $k\le l$, 
there exists $y(i) \in \Gamma^+_\infty(v_i^l)$ for each 
$i=1,2,\dots,m(l)$ such that 
\begin{equation*}
\sigma_\Lambda^m(y(i)) \ne y(j) \text{ for all } i,j =1,2,\dots,m(l) \text{ and } m=1,2,\dots,k. 
\end{equation*}
\end{enumerate}
Then we have implications: (i) $\Longleftrightarrow $ (ii) $\Longrightarrow$ (iii).
If in particular, ${\frak L}$ is the minimal $\lambda$-graph system
$\LLmin$ for a normal subshift $\Lambda$, then the three conditions are
all equivalent. 
\end{lemma}
\begin{proof}
(i) $\Longrightarrow $ (ii):
For $x= (\lambda(e_n))_{n\in \N}\in \Gamma^+_\infty(v_i^l)$, 
put $ v_j^{l+m}=t(e_m) \in V_{l+m}.$
Since $\Gamma_\infty^+(v_j^{l+m})$ contains 
at least two distinct sequences, one may find
$y \in \Gamma_\infty^+(v_i^l)$ such that 
$x_j = y_j $ 
for all 
$ j=1,2,\dots,m $
 and
$x_N \ne y_N $ for some $ N >m. $

(ii) $\Longrightarrow$ (i): The assertion is clear.

(ii) $\Longrightarrow$ (iii):
Take and fix $k \le l$. 
We will first see that for a vertex $v_i^l \in V_l$, 
\begin{equation} \label{eq:maruA}
\text{there exists } y \in \Gamma^+_\infty(v_i^l) \text{ such that }
\sigma_\Lambda^n(y) \ne y \text{ for } 1 \le n \le k.
\end{equation}
Take $x \in \Gamma^+_\infty(v_i^l).$
If $\sigma_\Lambda(x) =x$, 
we may find $y \in \Gamma_\infty^+(v_i^l)$ such that 
$\sigma_\Lambda(y) \ne y$ by the assertion (ii).
We may assume that $\sigma_\Lambda(x) \ne x$.
Now suppose that 
 $\sigma_\Lambda^n(x) \ne x$
for all $n \in \N$ with $1\le n \le K$ for some $K \in \N$. 
We will show that 
\begin{equation} \label{eq:assertionB}
\text{there exists } y \in \Gamma^+_\infty(v_i^l) \text{ such that }
\sigma_\Lambda^n(y) \ne y \text{ for } 1 \le n \le K+1.
\end{equation}
Let $x = (x_i)_{i \in \N}$.
 As $\sigma_\Lambda^n(x) \ne x$
for all $n \in \N$ with $1\le n \le K$,
there exists $k_n\in \N$ such that 
$x_{k_n} \ne x_{n+k_n}$
for each $n \in \N$ with $1\le n \le K$.
Put 
$$
M = \max\{n + k_n \mid n=1,2,\dots, K \}
$$
so that $M \ge K+1$.
Suppose that $\sigma_\Lambda^{K+1}(x) =x$.
By the condition (ii) for $m =M$,
there exists 
 $y =(y_n)_{n\in \N} \in \Gamma^+_\infty(v_i^l)$ such that 
\begin{gather} 
x_j = y_j \text{ for all } j=1,2,\dots,M \text{ and } \label{eq:y2} \\ 
x_N \ne y_N \text{ for some } N >M.  \label{eq:y3}
\end{gather} 
As $x_{k_n} \ne x_{n+k_n}$ for each $n \in \N$ with $1\le n \le K$,
the equality \eqref{eq:y2} implies
$y_{k_n} \ne y_{n+k_n}$ for all $n$ with $1 \le n \le K.$ 
Hence we have 
\begin{equation} \label{eq:y4}
\sigma_\Lambda^n(y) \ne y \text{ for } 1 \le n \le K.
\end{equation}
Now $\sigma_\Lambda^{K+1}(x) =x$ so that 
$x_{K + 1+ i} = x_i$ for all $i \in \N$.
If  $\sigma_\Lambda^{K+1}(y) =y$, 
the equality \eqref{eq:y2} implies
$x_j = y_j$ for all $j \in \N$, 
a contradiction to \eqref{eq:y3}.
Hence we see that 
 $\sigma_\Lambda^{K+1}(y) \ne y$
so that by \eqref{eq:y4}, 
we obtain that 
$\sigma_\Lambda^{n}(y) \ne y$ for all $n \in \N$ with $1 \le n\le K+1$
and thus the assertion \eqref{eq:maruA}.

We will next show the following:
 for $i=1,2,\dots,m(l)$ and 
$k,l\in \N$ with $k\le l$, 
there exists $y_i^l \in \Gamma^+_\infty(v_i^l)$ such that 
\begin{equation}\label{eq:maruB}
\sigma_\Lambda^n(y_j^l) \ne y_i^l \text{ for all } i,j =1,2,\dots,m(l) \text{ and } n=1,2,\dots,k. 
\end{equation}
For $i=1$, by \eqref{eq:maruA}, 
there exists
$y_1^l \in \Gamma^+_\infty(v_1^l)$
such that
$
\sigma^n(y_1^l) \ne y_1^l 
$
for $ 1 \le n \le k.$
By the condition (ii),  it is easy to see that 
the set of $\Gamma^+_\infty(v_i^l)$ satisfying \eqref{eq:maruA} for each $i=1,2,\dots, m(l)$
is infinite.
We will show that for a fixed $k \le l$,

there exists $y_i^l \in \Gamma^+_\infty(v_i^l)$ for each 
$i=1,2,\dots,m\le m(l)$
such that  
\begin{equation}\label{eq:maruBprime}
\sigma_\Lambda^n(y_j^l) \ne y_i^l \text{ for all } i,j =1,2,\dots,m \text{ and } n=1,2,\dots,k 
\end{equation}
by induction on  $m$ with $1 \le m \le m(l)$.

As in the preceding argument, \eqref{eq:maruBprime} holds for $m=1$. 
Now assume that \eqref{eq:maruBprime} holds for all $i \le m$.
We will then prove that \eqref{eq:maruBprime} holds for all $i \le m+1$.
It is easy to see that the set 
\begin{equation*}
Y_i =
\{ y \in \Gamma^+(v_i^l) \mid \sigma_\Lambda^n(y) \ne y \text{ for } 1 \le n \le k \}
\end{equation*}
is infinite by the above argument. 
In particular, $Y_{m+1}$ is infinite.
Take
$y_i^l \in \Gamma^+_\infty(v_i^l)$
for $i=1,2,\dots,m$
such that 
\begin{equation*}
\sigma_\Lambda^n(y_j^l) \ne y_i^l \text{ for all } i,j =1,2,\dots,m \text{ and } n=1,2,\dots,k. 
\end{equation*}
We may take and fix the above 
$y_i^l \in \Gamma^+_\infty(v_i^l)$
for $i=1,2,\dots,m$ by the induction hypothesis.
Consider the following set for the 
$y_i^l, \, i=1,2,\dots,m$:
\begin{align*}
Z 
& =  \{ y \in \Gamma^+(v_{m+1}^l) \mid \sigma_\Lambda^n(y_j^l) =y 
\text{ for some }j =1,2,\dots,m \text{ and } n =1,2,\dots,k \} \\
& \bigcup  \{ y \in \Gamma^+(v_{m+1}^l) \mid \sigma_\Lambda^n(y) =y_j^l 
\text{ for some }j =1,2,\dots,m \text{ and } n =1,2,\dots,k \}. \\
 \end{align*}
As $Z$ is a finite set and $Y_{m+1}$ is an infinite set, 
the set $Y_{m+1}\cap Z^c$ is infinite.
Hence we may find  
an element $y_{m+1}^l \in Y_{m+1}\cap Z^c$ satisfying
\begin{equation*}
 \sigma_\Lambda^n(y_{m+1}^l) \ne y_{m+1}^l, \qquad
  \sigma_\Lambda^n(y_j^l) \ne y_{m+1}^l, \qquad
\sigma_\Lambda^n(y_{m+1}^l) \ne y_j^l
\end{equation*}
for all $j =1,2,\dots,m$
 and $ n =1,2,\dots,k.$
Therefore the assertion \eqref{eq:maruBprime} holds for $m+1$,
so that the induction completes.
We thus obtain the assertion  (iii).

(iii) $\Longrightarrow$ (i) : 
Assume that  ${\frak L}$ is the minimal $\lambda$-graph system $\LLmin$ 
for a normal subshift $\Lambda$.
Suppose that ${\frak L}$ does not satisfy condition (I), so that 
there exists a vertex $v_i^l \in V_l$ such that 
$\Gamma^+_\infty(v_i^l) =\{y \}$ a singleton for some $y \in X_\Lambda$.
Now we are assuming that ${\frak L}$ is minimal and hence $\lambda$-synchronizing,
so that 
there exists $N_0\in \N$ such that 
$v_i^l$  launches $y_{[1,N_0]}$. 
Let $v_j^{l+1} \in V_{l+1}$ be such that 
$\iota(v_j^{l+1}) = v_i^l$.
For any $y' \in \Gamma^+_\infty(v_j^{l+1})$,
the local property of $\lambda$-graph system 
$\frak L$ ensures us that $y' \in \Gamma^+_\infty(v_i^{l})$ and hence 
$y' = y$.
Hence we have
$\Gamma^+_\infty(v_j^{l+1})= \Gamma^+_\infty(v_i^{l})$ whenever
$v_j^{l+1} \in V_{l+1}$ with  
$\iota(v_j^{l+1}) = v_i^l$.
Since ${\frak L}$ is $\lambda$-synchronizing,
$y$ never leaves any other vertex than $v_j^{l+1}$ in $V_{l+1}$.
Hence a vertex $v_j^{l+1} \in V_{l+1}$
satisfying $\iota(v_j^{l+1}) = v_i^l$ is unique.
We may write $j$ as $i(l+1)$, so that
$\Gamma^+(v_{i(l+1)}^{l+1}) = \{ y \}$.
Similarly we have a unique sequence 
of vettices $v_{i(l+n)}^{l+n}, n=1,2,\dots $
satisfying 
$v_{i(l+n)}^{l+n} \in V_{l+n},\, 
\iota(v_{i(l+n)}^{l+n}) =v_{i(l+n-1)}^{l+n-1}$
for 
$n=1,2,\dots.$
Now by the assumption (iii),
we have
$\sigma_\Lambda(y) \ne y$,
and hence there exists 
$j_1= 1,2,\dots,m(l+1)$ such that 
$\sigma_\Lambda(y) \in \Gamma^+_\infty( v_{j_1}^{l+1})$.
Hence we have 
$j_1 \ne i(l+1)$.
As $y = y_1\sigma(y)$ and $\Gamma^+_\infty(v_i^l) = \{ y\}$,
we have
$\Gamma^+_\infty(v_{j_1}^{l+1}) = \{ \sigma_\infty(y)\}.$
Together with 
$\Gamma^+_\infty(v_{i(l+1)}^{l+1}) = \{ y\},$
we have a contradiction to the condition (iii). 
\end{proof}

\begin{proposition}\label{prop:normal(I)}
Let $\LLmin$  be the minimal presentation of a normal subshift 
$\Lambda$. Then 
the $\lambda$-graph system $\LLmin$
satisfies condition (I).
\end{proposition}
\begin{proof}
By Lemma \ref{lem:Cantor},
$X_\Lambda$ is homeomorphic to a Cantor set.
For $v_i^l \in V_l^{\min}$,
there exists an $l$-synchronizing word $\mu \in S_l(\Lambda)$
for which $v_i^l$ launches $\mu$.
Hence we have $U_\mu \subset \Gamma^+_\infty(v_i^l)$
the cylinder set for the word $\mu$.
 As $X_\Lambda$ is homeomorphic to a Cantor set, the cylinder set 
 $U_\mu$ contains at least two points, so that 
 $\LLmin$ satisfies condition (I).  
\end{proof}
The following definition have been already introduced in previously published  papers.
The first one was introduced in \cite{MathScand2005}, that is stronger 
than condition (I) for $\lambda$-graph system in Definition \ref{def:cindition(I)}.
The second one was introduced in \cite{KMAAA2013} that was named as synchronizing condition (I)
\cite[(5.1)]{KMAAA2013}. 
\begin{definition}
\begin{enumerate}
\renewcommand{\theenumi}{\roman{enumi}}
\renewcommand{\labelenumi}{\textup{(\theenumi)}}
\item
A $\lambda$-graph system ${\frak L}$ is said to satisfy $\lambda$-{\it condition (I)}\/
if for any vertex $v_i^l \in V_l$, 
there exists a vertex $v_j^{L'} \in V_{L'}$ for some $L' >l$ such that 
there exist labeled paths 
$\gamma_1, \gamma_2$ in $\frak L$ satisfying  
\begin{equation*}
s(\gamma_1)= s(\gamma_2) = v_i^l,\qquad
t(\gamma_1)= t(\gamma_2) = v_j^{L'},\qquad
\lambda(\gamma_1) \ne \lambda(\gamma_2).
\end{equation*}
\item
A normal subshift $\Lambda$ is said to satisfy $\lambda$-{\it condition (I)}\/
 if for any $l \in \N$ and $\mu \in S_l(\Lambda)$,
 there exist $\xi_1, \xi_2 \in B_k(\Lambda)$ and $\nu \in S_{l+K}(\Lambda)$
for some $K\in \N$ such that 
\begin{equation*}
\xi_1, \xi_2 \in \Gamma_K^-(\nu), \qquad 
\xi_1\ne \xi_2, \qquad 
[\xi_1 \nu]_l =[\xi_2 \nu]_l = [\mu]_l.
\end{equation*}
\end{enumerate}
\end{definition}
The $\lambda$-condition (I) for a normal subshift had been called synchronizing condition (I)
in \cite{KMAAA2013}. 
Hence we know the following lemma that was already shown in \cite{KMAAA2013}.
\begin{lemma}[{\cite[Lemma 5.1]{KMAAA2013}}] \label{lem:lambdacond(I)}
Let $\Lambda$ be a normal subshift.
Then the  following two conditions are equivalent.
 \begin{enumerate}
\renewcommand{\theenumi}{\roman{enumi}}
\renewcommand{\labelenumi}{\textup{(\theenumi)}}
\item $\Lambda$ satisfies $\lambda$-condition (I).
\item $\LLmin$ satisfies $\lambda$-condition (I).
\end{enumerate}
\end{lemma}

%\newpage

%%%%%%%%%%%%% KOKOMADE %%%%%%%%%%%%%%%%%%%%%%%%%

%%%%%%%%%%%%%%%%%%%%%%%%%%%%%%%%%%%%%%%%%%%%%%%%%%
\section{Structure and simplicity of $\OLmin$}
%%%%%%%%%%%%%%%%%%%%%%%%%%%%%%%%%%%%%%%%%%%%%%%%%%%%%%%%

{\bf 1. Construction of the $C^*$-algebras associated with $\lambda$-graph systems.}

Following \cite{MaDocMath2002},
let us recall the construction of the $C^*$-algebra $\OL$ 
associated with a left-resolving $\lambda$-graph system ${\frak L}$.
The $C^*$-algebra was first defined as a groupoid $C^*$-algebra $C^*(G_{\frak L})$ of an \'etale amenable groupoid
$G_{\frak L}$ defined by a continuous graph $E_{\frak L}$ in the sense of V. Deaconu
 (cf. \cite{De}, \cite{De2}).
 Let ${\frak L}=(V,E,\lambda,\iota)$ be a left-resolving $\lambda$-graph system over $\Sigma$
 and $\Lambda$ its presented subshift.
 The vertex set $\Omega_{\frak L}$ of the continuous graph is defined by 
the compact Hausdorff space of the projective limit:
\begin{equation*}
\Omega_{\frak L} = \{ (u^l)_{l \in \Zp}\in \prod_{l \in \Zp} V_l
 \mid \ \iota_{l,l+1}(u^{l+1}) = u^l, l\in \Zp \}.
\end{equation*}
of the system
$\iota_{l,l+1}: V_{l+1}\rightarrow V_l, l \in \Zp$
 of continuous surjections.
 It is endowed by its projective limit topology.
We call each element of  
$\Omega_{\frak L}$ a vertex or an $\iota$-orbit.
 The continuous graph 
 $E_{\frak L}$ for ${\frak L}$ is defined by
 the set of triplets 
$(u, \alpha,w) \in \Omega_{\frak L} \times \Sigma \times \Omega_{\frak L}$
where $u = (u^l)_{l \in \Zp}, w = (w^l)_{l \in \Zp}\in \Omega_{\frak L}$
such that
there exists an edge $e_{l,l+1} \in E_{l,l+1}$ satisfying
\begin{equation*}
u^l = s(e_{l,l+1}),\quad w^{l+1} = t(e_{l,l+1}) \quad \text{ and } \quad
  \alpha = \lambda(e_{l,l+1})
  \quad \text{ for each } l \in \Zp 
\end{equation*}
(\cite[Proposition 2.1]{MaDocMath2002}, cf. \cite{De}, \cite{De2}).
Let us denote by $X_{\frak L}$ 
 the set of one-sided paths of $E_{\frak L}$:
\begin{align*}
X_{\frak L}= \{ (\alpha_i,u_i)_{i\in\N} \in \prod_{i\in\N}
(\Sigma \times \Omega_{\frak L}) \mid \
& (u_{0},\alpha_1,u_{1}) \in E_{\frak L} \text{ for some } u_0 \in \Omega_{\frak L} \\ 
 \text{ and } 
& (u_{i},\alpha_{i+1},u_{i+1}) \in E_{\frak L} 
 \text{ for all } i\in \N \}.
\end{align*}
We endow $X_{\frak L}$ with the relative topology from the infinite product topology
of $\Pi_{i\in \N}(\Sigma \times \Omega_{\frak L})$,
that makes $X_{\frak L}$ a zero-dimensional compact Hausdorff space.   
The continuous surjection of the shift map 
$
\sigma_{\frak L} :(\alpha_i,u_i)_{i\in\N}\in X_{\frak L} 
\rightarrow 
(\alpha_{i+1},u_{i+1})_{i\in\N}\in
X_{\frak L}
$
is defined on $X_{\frak L}$.
Since  
the $\lambda$-graph system ${\frak L}$ is left-resolving,
so that 
 $\sigma_{\frak L}$ is a local homeomorphism on $X_{\frak L}$
(\cite[Lemma 2.2]{MaDocMath2002}).
Let us define a factor map 
$\pi_{\frak L}:(\alpha_i,u_i)_{i\in\N}\in X_{\frak L}
\longrightarrow 
(\alpha_i)_{i\in \N} \in \Sigma^\N.
$
The image
$\pi_{\frak L}(X_{\frak L})$ in $\Sigma^\N$
is the shift space $X_\Lambda$ of the one-sided subshift
$(X_\Lambda,\sigma_\Lambda)$
with shift transformation
$\sigma_\Lambda((\alpha_i)_{i\in \N}) =(\alpha_{i+1})_{i\in \N}.
$ 
We then have 
$\pi_{\frak L}\circ\sigma_{\frak L}
= \sigma_\Lambda\circ\pi_{\frak L}.$ 

For the shift dynamical system
$(X_{\frak L},\sigma_{\frak L})$, 
one may construct a locally compact \'etale groupoid 
$G_{\frak L}$, called a Deaconu--Renault groupoid
 as in the following way.
We put
\begin{equation*}
G_{\frak L} = \{ (x,n,z) \in X_{\frak L} \times {\Bbb Z} \times X_{\frak L} \ | \
\text{ there exist } k,l\in \Zp ; \ \sigma^k_{\frak L}(x) = \sigma^l_{\frak L}(z), 
n=k-l \}
\end{equation*}
 (cf. \cite{De}, \cite{De2}, \cite{Renault}, \cite{Renault2}, \cite{Renault3}). 
The unit space
$G_{\frak L}^0 
= \{ (x,0,x) \in G_{\frak L} \ | \ x \in X_{\frak L} \}
$ is identified with  the space
$
 X_{\frak L}
$
through the map
$x \in X_{\frak L}\longrightarrow 
(x,0,x) \in G_{\frak L}^0.
$
The range map and the domain  map of $G_{\frak L}$ are defined by
$
r(x,n,z) = x
$
and
$d(x,n,z) =z
$ for
$
(x,n,z) \in G_{\frak L}.
$
The multiplication and the inverse operation are defined by
$
(x,n,z)(z,m,w) = (x,n+m,w)
$
and
$
(x,n,z)^{-1} = (z,-n,x).
$
An open neighborhood basis  of $G_{\frak L}$ 
is given by
\begin{equation*}
Z(U,k,l,V) 
= \{ (x,k-l,z) \in G_{\frak L} \ | \  x \in U, z \in V, 
\sigma^k_{\frak L}(x) = \sigma^l_{\frak L}(z) \}
\end{equation*}
for open sets $U,V$ of $X_{\frak L}$ and $k,l$
nonnegative integers such that 
$\sigma^k_{\frak L}|_{U}$ and $\sigma^l_{\frak L}|_{V}$ 
are homeomorphisms with the same open range.
We then have an \'etale amenable groupoid
$G_{\frak L}$.  
We will describe the construction of the groupoid $C^*$-algebra $C^*(G_{\frak L})$ 
for the groupoid $G_{\frak L}$ as in the following way 
(\cite{Renault}, \cite{Renault2}, \cite{Renault3}, cf. \cite{De}, \cite{De2}).
Let us denote by 
$C_c(G_{\frak L}) $ the set of compactly supported continuous functions on
$G_{\frak L}$ that has a natural product structure and $*$-involution of 
$*$-algebra given by
\begin{align*}
(f*g)(s) &  
           = 
 \sum_{t_1,t_2  \in G_{\frak L},\ s = t_1 t_2} f(t_1) g(t_2)
          = 
 \sum_{t \in G_{\frak L}, \ r(t) = r(s)} f(t) g(t^{-1}s) , \\
  f^*(s) & = \overline{f(s^{-1})} 
  \qquad \text{ for } \quad f,g \in C_c(G_{\frak L}), \quad s \in G_{\frak L}.     
\end{align*}
Let us denote by $C_0(G_{\frak L}^0) $ 
the $C^*$-algebra  of  continuous functions on
$G_{\frak L}^0$ that vanish at infinity. 
The algebra
$C_c(G_{\frak L}) $
has a structure of  
$C_0(G_{\frak L}^0) $-right module with 
a $C_0(G_{\frak L}^0) $-valued inner product  by
\begin{equation*}
 (\eta f )(x,n,z) 
 =  \eta(x,n,z)f(z), 
 \qquad   
   < \xi, \eta >(z) 
 =  \sum_{ 
 { (x,n,z) \in G_{\frak L} }}
   \overline{ \xi (x,n,z)} \eta (x,n,z),
   \end{equation*}
for 
$\xi, \eta  \in C_c(G_{\frak L}), \, f \in C_0(G_{\frak L}^0),
   \, (x,n,z) \in G_{\frak L}, \, z \in X_{\frak L}.
$
The completion of the inner product 
$C_0(G_{\frak L}^0) $-right module
$C_c(G_{\frak L})$
is denoted by  $\ell^2(G_{\frak L})$, that
is a Hilbert $C^*$-right module over the commutative $C^*$-algebra 
$C_0(G_{\frak L}^0) $.
Let us  denote by
$B(\ell^2(G_{\frak L}))$
the $C^*$-algebra of all bounded adjointable
$C_0(G_{\frak L}^0) $-module maps on $\ell^2(G_{\frak L}).$
Let $\pi $ be the $*$-homomorphism of 
$C_c(G_{\frak L})$ into $B(\ell^2(G_{\frak L}))$
defined by
$\pi (f)\eta = f * \eta$
for $f, \eta \in 
C_c(G_{\frak L}).$
The (reduced) $C^*$-algebra of the groupoid $G_{\frak L}$
is defined by the closure of $\pi (C_c(G_{\frak L}))$ in
$B(\ell^2(G_{\frak L}))$,
that we denote by
$C^*_r(G_{\frak L}).$
General theory of $C^*$-algebras of groupoids says that 
for a Deaconu-- Renault groupoid $G$, 
the reduced $C^*$-algebra $C^*_r(G)$ and
the universal $C^*$-algebra $C^*(G)$ 
are canonically isomorphic and hence they are identified
(see for instance \cite[Proposition 2.4]{Renault2000}).
We denote them by $C^*(G)$.
\begin{definition}[{\cite{MaDocMath2002}}] 
The $C^*$-algebra $\OL$ associated with  a left-resolving $\lambda$-graph system 
$\frak L$ is defined to be the  $C^*$-algebra $C^*(G_{\frak L})$
of the groupoid $G_{\frak L}.$
\end{definition}
The vertex set
$V_l$ at level $l$ of ${\frak L}$ is denoted by
$\{{v}_1^l,\dots,{v}_{m(l)}^l \}$.
For 
$x = (\alpha_n,u_n)_{n\in \N} \in X_{\frak L},$
we put
$\lambda(x)_n = \alpha_n \in \Sigma,$
$v(x)_n = u_n \in \Omega_{\frak L}$
for $n \in \N,$ 
respectively.
The $\iota$-orbit
$v(x)_n$ is written as
$v(x)_n = {(v(x)^l_n)}_{l\in \Zp} 
\in \Omega_{\frak L}.$
Now $\frak L$ is left-resolving
so that there exists a unique vertex 
$v(x)_0 \in \Omega_{\frak L}$ satisfying
$(v(x)_0,\alpha_1,u_1) \in E_{\frak L}.$
Define 
$U(\alpha) \subset G_{\frak L}$ 
for 
$\alpha  \in \Sigma,$
and
$U(v_i^l)\subset G_{\frak L}$
for $v_i^l \in V_l$ by
\begin{gather*}
 U(\alpha) 
 = \{ (x, 1, z) \in G_{\frak L} \mid
 \sigma_{\frak L}(x) = z,   
\lambda(x)_1 =\alpha \}, \quad \text{ and}\\
 U({v}_i^l) 
 = \{ (x,0,x) \in G_{\frak L} \ | \ v(x)_0^l = {v}_i^l\}
\end{gather*}
where
$v(x)_0 = (v(x)^l_0)_{l\in \Zp}\in \Omega_{\frak L}.$
They are clopen sets of $G_{\frak L}$.
We define
\begin{equation*}
S_{\alpha} =\pi( \chi_{U(\alpha)} ),\qquad E_i^l = \pi(\chi_{U({v}_i^l) })
\qquad \text{ in }\quad \pi(C_c(G_{\frak L}))
\end{equation*}
where 
$ \chi_{F}
\in C_c(G_{\frak L})
$ denotes 
the characteristic function of a clopen set
$F$ on the groupoid 
$G_{\frak L}.$

The transition matrix system $(A_{l,l+1}, I_{l,l+1})_{l\in \Zp}$
for  the $\lambda$-graph system ${\frak{L}}$
determines the structure of the $\lambda$-graph system ${\frak{L}}$
that are defined by 
\begin{align*}
A_{l,l+1}(i,\alpha,j)
 & =
{\begin{cases}
1 &    \text{ if there exists } e \in E_{l,l+1}; 
 \ s(e) = v_i^l, \lambda(e) = \alpha,
                       t(e) = v_j^{l+1}, \\
0  & \text{ otherwise,}
\end{cases}} \\
I_{l,l+1}(i,j)
 & =
{\begin{cases}
1 &      \text{ if } \ \iota_{l,l+1}(v_j^{l+1}) = v_i^l, \\
0 & \text{ otherwise}
\end{cases}} 
\end{align*} 
for
$
i=1,2,\dots,m(l),\ j=1,2,\dots,m(l+1), \ \alpha \in \Sigma.
$ 
More generally for $v_i^l \in V_l, v_k^{l+n} \in V_{l+n}$
and $\nu =(\nu_1,\dots,\nu_n) \in B_n(\Lambda)$,
we define
\begin{align*}
A_{l,l+n}(i,\nu, k)
 & =
{\begin{cases}
1 &    \text{ if there exists } \gamma \in E_{l,l+n}; 
 \ s(\gamma) = v_i^l, \lambda(\gamma) = \nu,
                       t(\gamma) = v_k^{l+n}, \\
0  & \text{ otherwise,}
\end{cases}} \\
I_{l,l+n}(i,k)
 & =
{\begin{cases}
1 &      \text{ if } \ (\iota_{l,l+1}\circ\cdots\circ\iota_{l+n-1,l+n}) (v_k^{l+n}) = v_i^l, \\
0 & \text{ otherwise}
\end{cases}} 
\end{align*} 
so that 
\begin{align*}
A_{l,l+n}(i,\nu, k)
 & = \sum_{j_1,\dots,j_{n-1}}
 A_{l,l+1}(i,\nu_1, j_1)\cdots A_{l+n-1,l+n}(j_{n-1},\nu_n, k), \\
I_{l,l+n}(i,k)
 & =
 \sum_{j_1,\dots,j_{n-1}}
 I_{l,l+1}(i,j_1)\cdots I_{l+n-1,l+n}(j_{n-1}, k).
 \end{align*}
For a vertex $v_i^l \in V_l$, denote by $\Gamma_l^-(v_i^l)$ 
the predecessor set of $v_i^l$ that is defined in \eqref{eq:Gammalv}
as the set of words in $B_l(\Lambda)$
that are realized by  labeled edges in ${\frak L}$ whose terminal is $v_i^l$.
Recall that  ${\frak L}$ is predecessor-separated if   
$\Gamma_l^-(v_i^l) \ne \Gamma_l^-(v_j^l)$
for distinct $i,j =1,2, \dots,m(l)$.
We had proved the following theorem.
\begin{proposition}[{\cite[Theorem 3.6, Theorem 4.3]{MaDocMath2002}}] 
\label{prop:lambdaC}
Let ${\frak{L}}$ be a left-resolving $\lambda$-graph system.
The $C^*$-algebra $\OL$
is a universal unital $C^*$-algebra
generated by
partial isometries
$S_{\alpha}$ for $\alpha \in \Sigma$
and projections
$E_i^l$ for $v_i^l \in V_l 
$ subject to the  following relations called $({\frak{L}})$:
\begin{gather*}
\sum_{\beta \in \Sigma} S_{\beta}S_{\beta}^* 
= \sum_{i=1}^{m(l)} E_i^l   =1,  
\qquad S_\alpha S_\alpha^* E_i^l  =   E_i^{l} S_\alpha S_\alpha^* \\
E_i^l  =   \sum_{j=1}^{m(l+1)}I_{l,l+1}(i,j)E_j^{l+1},
\qquad
 S_{\alpha}^*E_i^l S_{\alpha}  =  
\sum_{j=1}^{m(l+1)} A_{l,l+1}(i,\alpha,j)E_j^{l+1} 
\end{gather*}
for $\alpha \in \Sigma,$
$i=1,2,\dots,m(l), l \in \Zp. $
 If in particular ${\frak L}$ satisfies condition (I),
 then any non-zero generators 
 %$S_\alpha, \alpha \in \Sigma, E_i^l, i=1,2,\dots, m(l)$
 satisfying the above relations $({\frak L})$
generate an isomorphic copy of $\OL$.
Hence $\OL$ is a unique $C^*$-algebra subject to the relations $({\frak L})$
if ${\frak{L}}$ satisfies condition (I).    
If in addition,
 $\frak{L}$ is $\lambda$-irreducible, 
the $C^*$-algebra $\OL$ is simple and purely infinite (\cite{MathScand2005}).
\end{proposition}

It is nuclear and belongs to the UCT class (\cite[Proposition 5.6]{MaDocMath2002}).
If ${\frak L}$ is predecessor-separated, then
the projections $E_i^l$ are written by using the partial isometries
$S_\alpha, \alpha\in \Sigma$ 
in the following way: 
\begin{equation}
E_i^l = \prod_{\mu \in \Gamma_l^-(v_i^l)} S_\mu^* S_\mu
\cdot \prod_{\nu\not{\in} \Gamma_l^-(v_i^l) \cap B_l(\Lambda)}
(1- S_\nu^* S_\nu),
\qquad i=1,2, \dots,m(l) \label{eq:eil}
\end{equation}
where $S_\mu$ denotes $S_{\mu_1}\cdots S_{\mu_m}$ for 
$\mu =(\mu_1,\dots,\mu_m)\in B_*(\Lambda)$.
Hence the $C^*$-algebra $\OL$ is generated by only finite family of the
partial isometries $S_\alpha, \alpha \in \Sigma$.

By the above relation $({\frak{L}})$, 
one sees that 
the algebra of  finite linear combinations of the elements of the form
$$
S_{\mu}E_i^lS_{\nu}^* \quad \text{ for }
\quad \mu,\nu \in B_*(X_\Lambda), \quad i=1,\dots,m(l),\quad l \in \Zp
$$
forms a dense $*$-subalgebra of  $\OL$.
Let us  denote by
$\DL$ the $C^*$-subalgebra of $\OL$
generated by the projections of the form
$S_\mu E_i^l S_\mu^*, i=1, 2, \dots, m(l),\/ l \in \Zp, \, \mu\in B_*(\Lambda)$.
We also know that the algebra $\DL$ 
is canonically isomorphic to the commutative $C^*$-algebra 
$C(X_{\frak L})$ of continuous functions on $X_{\frak L}$.
The $C^*$-subalgebra of $\DL$ generated by the projections of the form
$S_\mu S_\mu^*, \mu\in B_*(\Lambda)$
is canonically isomorphic to the  commutative $C^*$-algebra 
$C(X_\Lambda)$ of continuous functions on the right one-sided subshift
$X_\Lambda$,
that is written $\DLam$.

Let us define several kinds of $C^*$-subalgebras of $\OL$
that will be useful in our further discussions.
For a subset $F \subset \OL$, we denote by $C^*(F)$
the $C^*$-subalgebra of $\OL$ generated by all elements of $F$.
Let $k,l\in \Zp$ with $k\le l$.
We deine $C^*$-subalgebras of $\OL$ by
\begin{align*}
\A_l 
& = C^*(E_i^l : i=1,2,\dots, m(l)), \\
\AL 
& = C^*(E_i^l : i=1,2,\dots, m(l), l\in \Zp), \\
\D_{k,l}
& =C^*(S_\mu E_i^l S_\mu^* : i=1,2,\dots, m(l), \mu \in B_k(\Lambda)), \\
\D_{k,{\frak L}}
& =C^*(S_\mu E_i^l S_\mu^* : i=1,2,\dots, m(l), \mu \in B_k(\Lambda), l\in \Zp), \\
%\D_{\frak L}
%& =C^*(S_\mu E_i^l S_\mu^* : i=1,2,\dots, m(l), \mu \in B_k(\Lambda), k, l\in \Zp), \\
%\D_{k}
%& =C^*(S_\mu S_\mu^* : \mu \in B_k(\Lambda)), \\
%\D_{\Lambda}
%& =C^*(S_\mu S_\mu^* :  \mu \in B_*(\Lambda)),\\ 
\F_{k,l}
& =C^*(S_\mu E_i^l S_\nu^* : i=1,2,\dots, m(l), \mu, \nu \in B_k(\Lambda)), \\
\F_{k,{\frak L}}
& =C^*(S_\mu E_i^l S_\nu^* : i=1,2,\dots, m(l), \mu, \nu \in B_k(\Lambda), l\in \Zp), \\
\F_{\frak L}
& =C^*(S_\mu E_i^l S_\nu^* : i=1,2,\dots, m(l), \mu, \nu \in B_k(\Lambda), k, l\in \Zp).
\end{align*}
As in the papers \cite{MaDocMath2002}, \cite{KMAAA2013} and \cite{MaJAMS2013}, etc.,
the $C^*$-algebra $\OL$ has a natural action of the circle group
$\T$ called gauge action 
written $\rho^{\frak L}$, that is defined by for
$t \in {\mathbb{T}}=\mathbb{R}/\Z$, 
\begin{equation} \label{eq:gauge}
\rho^{\frak L}_t(S_\alpha)= e^{2\pi\sqrt{-1}} S_\alpha, \quad \alpha \in \Sigma,\qquad
\rho^{\frak L}_t(E_i^l)= E_i^l, \quad i=1,\dots,m(l), \, l\in \Zp.
\end{equation}
 The fixed point algebra $\FL$ of $\OL$ is an AF-algebra with its diagonal
algebra $\DL$.
%The commutative $C^*$-algebra $C(X_\Lambda)$
%of continuous functions on the right one-sided subshift $X_\Lambda$ 
%is isomorphic to  $\DLam$.
Let us define
$\phi_{\frak L} : \DL \longrightarrow \DL$ by
$\phi(X) = \sum_{\alpha \in \Sigma} S_\alpha X S_\alpha^*, X \in \DL$.
The restriction of $\phi_{\frak L}$ to $\DLam$ is denoted by
$\phi_\Lambda$.
\begin{lemma}  
Let ${\frak L}$ be a left-resolving $\lambda$-graph system.
Then the following two conditions are equivalent:
\begin{enumerate}
\renewcommand{\theenumi}{\roman{enumi}}
\renewcommand{\labelenumi}{\textup{(\theenumi)}}
\item For $k,l \in \N$ with $k \le l$ and $i=1,2,\dots,m(l)$,
there exists $y(i) \in \Gamma^+_\infty(v_i^l)$ for each $i=1,2,\dots,m(l)$ such that 
\begin{equation}
\sigma_\Lambda^n(y(i)) \ne y(j) \quad
\text{ for all } i,j = 1,2,\dots, m(l), n=1,2,\dots,k. \label{eq:condition(I)3}
\end{equation}
\item For $k,l \in \N$ with $k \le l$, 
there exists a projection  $q_k^l \in \DLam$ such that 

(1) $q_k^l a \ne 0$ for all $0 \ne a \in \A_l,$

(2) $q_k^l \phi_\Lambda^n(q_k^l) = 0$ for $n=1,2,\dots k$.
\end{enumerate}
\end{lemma}
\begin{proof}
(i) $\Longrightarrow$ (ii): 
By the condition (i), 
take $y(i) \in \Gamma^+_\infty(v_i^l)$ for each $i=1,2,\dots,m(l)$  satisfying  
\eqref{eq:condition(I)3}.
Put a finite subset of $X_\Lambda$
$$
Y = \{ y(i) \mid i=1,2,\dots,m(l) \} \subset X_\Lambda.
$$
We then have $\sigma^{-n}_\Lambda(Y) \cap Y =\emptyset$ for all $n=1,2,\dots,k$.
Now $X_\Lambda$ is Hausdorff % by Lemma \ref{lem:Cantor}, 
so that we may take
a clopen set $V \subset X_\Lambda$ such that 
$Y \subset V$ and $\sigma^{-n}_\Lambda(V) \cap V =\emptyset$ for all $n=1,2,\dots,k$.
Define $q_k^l = \chi_V \in C(X_\Lambda)(=\DLam)$
the characteristic function of $V$ on $X_\Lambda$.
Since $y(i) \in Y \subset V$ 
and $y(i) \in \Gamma^+_\infty(v_i^l)$
we have $q_k^l \cdot E_i^l\ne 0$.
On the other hand, 
the condition $\sigma^{-n}_\Lambda(V) \cap V =\emptyset$ for all $n=1,2,\dots,k$
ensures us 
$q_k^l \phi_\Lambda^n(q_k^l) = 0$ for $n=1,2,\dots k$.
As the $C^*$-subalgebra $\A_l$ is the direct sum $\oplus_{i=1}^{m(l)} \mathbb{C}E_i^l$,
we see that $q_k^l a \ne 0$ for all $0 \ne a \in \A_l$.

(ii) $\Longrightarrow$ (i):
Assume the condition (ii).
 For $k,l \in \N$ with $k \le l$, 
there exists a projection  $q_k^l \in \DLam$ satisfying the conditions (1) and (2).
The condition (1) implies that 
$q_k^l E_i^l \ne 0$ for all
$ i=1,2,\dots, m(l)$.
One may take a clopen set $V \subset X_\Lambda$ such that 
$q_k^l = \chi_V$ and hence 
\begin{equation*}
V \cap \Gamma^+_\infty(v_i^l) \ne \emptyset \text{ for } i=1,2,\dots, m(l)
\quad  \text{ and } \quad
V \cap \sigma_\Lambda^{-n}(V) = \emptyset \text{ for } n=1,2,\dots, k.
\end{equation*}
Take 
$y(i) \in V \cap \Gamma^+_\infty(v_i^l)$ for each $i=1,2,\dots, m(l)$,
so that we have $\sigma^n_\Lambda(y(i)) \ne y(j)$ for all
$i,j = 1,2,\dots, m(l), n=1,2,\dots,k$.
Thus the assertion (i) holds.
\end{proof}
Since the condition (i) in the above lemma is the same as the condition (iii)
in Lemma \ref{lem:threecond},
the following lemma holds.
\begin{lemma}  
Let ${\frak L}$ be a left-resolving $\lambda$-graph system
satisfying condition (I).
Then for $k,l \in \N$ with $k \le l$, 
there exists a projection  $q_k^l \in \DLam$ such that 

(1) $q_k^l a \ne 0$ for all $0 \ne a \in \A_l$,

(2) $q_k^l \phi_\Lambda^n(q_k^l) = 0$ for $n=1,2,\dots k$.
\end{lemma}
Now we put
$Q_k^l := \phi_{\Lambda}^k(q_k^l) \in \DLam$
a projection in $\DLam$.
We note that each element of $\DL$ commutes with
elements of $\AL$. 
As we see  the identity
\begin{equation*}
S_\mu \phi^j_{\frak L}(X) = \phi^{j+|\mu|}_{\frak L}(X) S_\mu,
\qquad X \in \DL, \mu \in B_*(\Lambda), j \in \Zp,
\end{equation*}
where $|\mu|$ denotes the length of the word $\mu$,
a similar argument to \cite[2.9 Proposition]{CK} leads to the following lemma,
that was seen in \cite[Lemma 4.2]{MaDocMath2002}. 
\begin{lemma} Keep the above notation.\hspace{4cm}
\begin{enumerate}
\renewcommand{\theenumi}{\roman{enumi}}
\renewcommand{\labelenumi}{\textup{(\theenumi)}}
\item The correspondence 
$X \in \F_{k,l} \longrightarrow Q_k^l X Q_k^l \in Q_k^l \F_{k,l} Q_k^l$
extends to an isomorphism from $\F_{k,l} $ to $ Q_k^l \F_{k,l} Q_k^l$. 
\item For $X \in \FL$, we have $\| Q_k^l X - X Q_k^l \| \longrightarrow 0$ and
$\| Q_k^l X \| - \| X \| \longrightarrow 0$ as $ k,l\longrightarrow \infty$.
\item For $\mu \in B_*(\Lambda)$, 
we have $\| Q_k^l S_\mu \|, \| Q_k^l S_\mu^* Q_k^l \| \longrightarrow 0$ 
as $ k,l\longrightarrow \infty$.
\end{enumerate}
\end{lemma}

The following lemma was seen in \cite[Lemma 2.5]{MaDynam2020} and \cite[Lemma 6.5]{MaYMJ2010}
without its detail proofs. 
We will give its detail proof here.
\begin{lemma}[{cf. \cite[Lemma 3.1, Lemma 3.2]{MaJOT2000}}] \label{lem:3.6}
Let ${\frak L}$ be a left-resolving $\lambda$-graph system
satisfying condition (I).
\begin{enumerate}
\renewcommand{\theenumi}{\roman{enumi}}
\renewcommand{\labelenumi}{\textup{(\theenumi)}}
\item
$\DLam^\prime \cap \OL \subset \FL$.
\item
$\DLam^\prime \cap \FL \subset \DL$.
\end{enumerate}
\end{lemma}
\begin{proof}
(i)
Let $E:\OL\longrightarrow \FL$ be the conditional expectation 
defined by 
$$
E(X) = \int_{\T}\rho^{\frak L}_t(X) dt \qquad X \in \OL
$$
where $dt$ denotes the normalized Lebesgue measure on $\T =\R/\Z$. 
For $X \in \DLam^\prime \cap \OL$, we put
$$
X_\mu = E(S_\mu^* X), \qquad
X_{-\mu} = E(X S_\mu) \quad \text{ for } \mu \in B_*(\Lambda).
$$  
We will show that $X_\mu = X_{-\mu} =0$ for $\mu \in B_k(\Lambda)$ with $k \ge 1$.
%We have
%$$X_\mu S_\mu S_\mu^*= E(S_\mu^* X  S_\mu S_\mu^*)= E(S_\mu^*   S_\mu S_\mu^* X) = X_\mu.$$
For $f \in \DLam$, we have
$$
X_\mu S_\mu f S_\mu^*= E(S_\mu^* X  S_\mu f S_\mu^*)= E(S_\mu^* S_\mu f S_\mu^* X) 
= E( f S_\mu^* X) = f X_\mu.
$$
%Let $|\mu|$ be the length of word $\mu$. We have
It follows that
$$
X_\mu \phi_{\frak L}^{|\mu|}( f) 
= X_\mu S_\mu S_\mu^* \sum_{\nu \in B_{|\mu|}(\Lambda)} S_\nu f S_\nu^*
= X_\mu S_\mu S_\mu^* S_\mu f S_\mu^*
= f X_\mu.
$$
Now suppose that $X_\mu \ne 0$.
For $\epsilon >0$, there exist $k,l\in \Zp$ with $k \le l$ 
and $X_{k,l} \in \F_{k,l}$
such that 
$|\mu| \le k$ 
and
$\| X_\mu - X_{k,l} \| < \epsilon.$
We may assume that 
$\| X_\mu \| = \| X_{k,l} \| = 1.$
We then have for $f \in \DLam$,
$$
\| f X_{k,l} - X_{k,l} \phi_{\frak L}^{|\mu|}(f)\| < 2 \| f \| \epsilon.
$$
Now ${\frak L}$ satisfies condition (I), so that there exists a projection
$Q_k^l$  in $ \DLam$ defined by $Q_k^l = \phi^k_{{\frak L}}(q_k^l)$ satisfying 
the previous lemma. 
By considering $S_\xi S_\xi^* X_{k,l} S_\xi S_\xi^*$ instead of $X_{k,l}$,
we may suppose that $X_{k,l} $ is of the form $S_\xi E_i^l S_\eta^*$ for some $\xi,\eta \in B_k(\Lambda)$.
It then follows that 
\begin{equation*}
Q_k^l X_{k,l} 
=  \sum_{\nu \in B_{k}(\Lambda)} S_\nu q_k^l S_\nu^* S_\xi E_i^l S_\eta^* 
=  S_\xi q_k^l S_\xi^* S_\xi E_i^l S_\eta^* 
=  S_\xi E_i^l q_k^l  S_\eta^* 
\end{equation*}
and
\begin{equation*}
X_{k,l} Q_k^l  
=   S_\xi E_i^l S_\eta^* \sum_{\nu \in B_{k}(\Lambda)} S_\nu q_k^l S_\nu^* 
=   S_\xi E_i^l S_\eta^* S_\eta q_k^l S_\eta^* 
=   S_\xi E_i^l q_k^l S_\eta^* 
\end{equation*}
so  that $Q_k^l $ commutes with $X_{k,l}$.
Hence we have 
\begin{equation}
  \|  X_{k,l} Q_k^l - X_{k,l} \phi_{\frak L}^{|\mu|}(Q_k^l)\| 
= \|  Q_k^l X_{k,l} - X_{k,l} \phi_{\frak L}^{|\mu|}(Q_k^l)\| 
\le  2 \| Q_k^l \| \epsilon = 2 \epsilon. \label{eq:xklqkl}
\end{equation}
As  
$Q_k^l \phi_{\frak L}^{|\mu|}(Q_k^l)
=\phi_{\frak L}^{k}(q_k^l \phi_{\frak L}^{|\mu|}(q_k^l))
=0,
$
we have
$$
\|  X_{k,l} Q_k^l - X_{k,l} \phi_{\frak L}^{|\mu|}(Q_k^l)\| 
= \max\{ \|X_{k,l} Q_k^l\|, \| X_{k,l} \phi_{\frak L}^{|\mu|}(Q_k^l)\| \}.
$$
Since
the correspondence 
$X \in \F_{k,l} \longrightarrow Q_k^l X Q_k^l \in Q_k^l \F_{k,l} Q_k^l$
extends to an isomorphism from $\F_{k,l} $ to $ Q_k^l \F_{k,l} Q_k^l$
so that 
$\| X_{k,l} Q_k^l \| = \| X_{k,l}\| =1$.
Hence we have
$$
\|  X_{k,l} Q_k^l - X_{k,l} \phi_{\frak L}^{|\mu|}(Q_k^l)\|  \ge 1
$$
a contradiction to \eqref{eq:xklqkl}.
We thus have $X_\mu = 0$ and similarly $X_{-\mu} =0.$
This means that $X = E(X) \in \FL.$

(ii)
For $\mu \in B_k(\Lambda)$, we put $P_\mu = S_\mu S_\mu^*$
and define the map $\mathcal{E}_k^l: \F_{k,l} \longrightarrow \D_{k,l}$
by setting
$\mathcal{E}_k^l(X) = \sum_{\mu \in B_k(\Lambda)} P_\mu X P_\mu$ 
for $X \in \F_{k,l}$.
Since the restriction of
$\mathcal{E}_k^{l+1}$ to $\F_{k,l}$ coincides with
$\mathcal{E}_k^l$, the sequence 
$\{ \mathcal{E}_k^l \}_{k\le l}$
gives rise to an expectation 
${\mathcal{E}}_k^{\frak L} : \F_{k,{\frak L}} \longrightarrow 
\D_{k,{\frak L}}$ for $k \in \N$.
 Similarly the above sequence 
$\{ {\mathcal{E}}_k^{\frak L} \}_{k \in \N}$  
 of expectations yields an expectation
${\mathcal{E}}^{\frak L} : \FL \longrightarrow \DL$
such that 
the restriction of ${\mathcal{E}}^{\frak L}$ 
to $\F_{k,{\frak L}} $ coincides with
${\mathcal{E}}_k^{\frak L}$
for $k \in \N$.
 
For $X \in \DL^\prime \cap \FL$, we know that
${\mathcal{E}}_k^{\frak L}(X) = X$
for $k \in \N$, so that 
${\mathcal{E}}^{\frak L}(X) = X$.
Since ${\mathcal{E}}^{\frak L}(X) \in \DL$, we have 
${\mathcal{E}}^{\frak L}(X) \in \DL$.
\end{proof}
We thus have the following proposition. 
\begin{proposition}[{cf. \cite[Lemma 3.1, Lemma 3.2]{MaJOT2000}}]\label{prop:relativecom}
Let ${\frak L}$ be a left-resolving $\lambda$-graph system
satisfying condition (I). 
Then we have
$$\DLam^\prime \cap \OL = \DL.$$
\end{proposition}
\begin{proof}
 The inclusion relation 
$\DLam^\prime \cap \OL \supset \DL$
is obvious.
For $X \in \DLam^\prime \cap \OL$
by the assertions (i) and (ii) in Lemma \ref{lem:3.6}, we know that 
$X$ belongs to 
$\FL$ and $\DL$
so that $\DLam^\prime \cap \OL \subset \DL$.
\end{proof}

\medskip

\noindent
{\bf 2. The $C^*$-algebras associated with normal subshifts.}

For a normal subshift $\Lambda$,
denote by $\LLmin$ its minimal presentation.

\begin{definition}
The $C^*$-algebra $\OLmin$ associated with the normal subshift
$\Lambda$ is defined by the $C^*$-algebra
 $\mathcal{O}_{\LLmin}$ associated with the minimal $\lambda$-graph system
 $\LLmin$.
\end{definition} 
Let  $(A^{\min}_{l,l+1}, I^{\min}_{l,l+1})_{l\in \Zp}$
be the transition matrix system for the minimal $\lambda$-graph system
$\LLmin$ that is defined before Proposition \ref{prop:lambdaC}. 
Then we have 
\begin{proposition}\label{prop:universalunique}
The $C^*$-algebra $\OLmin$ is the universal unique concrete $C^*$-algebra
generated by partial isometries $S_\alpha$ indexed by symbols 
$\alpha \in \Sigma$ and projections
$E_i^l$ indexed by vertices $v_i^l \in V_l^{\min}$ 
subject to the following operator relations called $(\LLmin)$:
\begin{gather*}
1 =\sum_{\alpha \in \Sigma}S_\alpha S_\alpha^* = \sum_{i=1}^{m(l)} E_i^l, 
\qquad S_\alpha S_\alpha^* E_i^l = E_i^l S_\alpha S_\alpha^*,\\
E_i^l = \sum_{j=1}^{m(l+1)} I^{\min}_{l,l+1}(i,j) E_j^{l+1}, \qquad
S_\alpha^* E_i^l S_\alpha = \sum_{j=1}^{m(l+1)} A^{\min}_{l,l+1}(i,\alpha, j) E_j^{l+1}
\end{gather*}
for $\alpha \in \Sigma,\ i=1, 2,\dots,m(l)$.
\end{proposition}
\begin{proof}
By Proposition \ref{prop:normal(I)}, the $\lambda$-graph system 
$\LLmin$ satisfies condition (I) so that we know that
the $C^*$-algebra $\OLmin$ is the universal unique concrete $C^*$-algebra
generated by partial isometries $S_\alpha$ indexed by symbols 
$\alpha \in \Sigma$ and projections
$E_i^l$ indexed by vertices $v_i^l \in V_l$ 
subject to the operator relations $(\LLmin)$.
\end{proof}
We thus have the following theorem,
that was already seen in \cite{KMAAA2013} and \cite{MaJAMS2013}.
\begin{theorem}
Let $\Lambda$ be a normal subshift.
\begin{enumerate}
\renewcommand{\theenumi}{\roman{enumi}}
\renewcommand{\labelenumi}{\textup{(\theenumi)}}
\item
If $\Lambda$ is $\lambda$-transitive, then the $C^*$-algebra 
$\OLmin$ is simple.
\item If $\Lambda$ is $\lambda$-transitive and satifies $\lambda$-condition (I), 
then the $C^*$-algebra $\OLmin$ is simple and purely infinite.
\end{enumerate}
\end{theorem}
\begin{proof}
(i) The assertion was already seen in \cite{KMAAA2013} and \cite{MaJAMS2013}.

(ii) By Lemma \ref{lem:lambdacond(I)}, the $\lambda$-graph system 
$\LLmin$ satisfies $\lambda$-condition (I).
By \cite{MathScand2005}, the $C^*$-algebra $\OLmin$ is simple and purely infinite.
\end{proof}
The following lemma is useful in our further discussions.
\begin{lemma}[{\cite[Proposition 3.3]{MaJAMS2013}}]
Let $\Lambda$ be a normal subshift.
For a vertex $v_i^l \in V_l^\min$ in $\LLmin$,
there exists $\mu \in S_l(\Lambda)$ such that 
$E_i^l \ge S_\mu S_\mu^*$ in  $\OLmin$.
That is, if $v_i^l$ launches $\mu$, 
the inequality
$E_i^l \ge S_\mu S_\mu^*$ holds. 
\end{lemma}
The above algebraic property of the 
$C^*$-algebra $\OLmin$ characterizes the $C^*$-algebra $\OL$ to be $\OLmin$.

We note that the minimal $\lambda$-graph system 
$\LLmin$ is predecessor-separated, so that 
the projections $E_i^l$ are written in terms of the partial isometries
$S_\alpha, \alpha \in \Sigma$ as in \eqref{eq:eil}.
Hence $C^*$-algebra $\OLmin$ is generated by only the finite family of the partial isometries $S_\alpha, \alpha \in \Sigma$.

We will see  that irreducible sofic shifts $\Lambda$ 
such that $\Lambda$ is not finite as a set satisfy the condition
(ii) in the above theorem.   
We will study more detail in Section 4.

Recall that the $C^*$-algebras
$\D_{\LLmin}$ and $\DLam$ are both commutative $C^*$-subalgebras of $\OLmin$
defined by
\begin{align*}
\D_{\LLmin} =& C^*(S_\mu E_i^l S_\mu^* : \mu \in B_*(\Lambda), i=1,2,\dots, m(l), l\in \Zp), \\
\DLam =& C^*(S_\mu S_\mu^* : \mu \in B_*(\Lambda)).
\end{align*}
The former is isomorphic to $C(X_{\LLmin})$, and the latter is isomorphic to $C(X_\Lambda)$.
The natural factor map $\pi_{\frak L}: X_{\LLmin} \longrightarrow X_\Lambda$ induces the inclusion
$$
D_\Lambda (=C(X_\Lambda)) \subset \D_\LLmin (=C(X_\LLmin)).
$$ 
Since the minimal $\lambda$-graph system $\LLmin$
of a normal subshift $\Lambda$ satisfies condition (I) by 
Proposition \ref{prop:normal(I)},
we have the following proposition.
\begin{proposition}\label{prop:relative}
Let $\Lambda$ be a normal subshift and 
$\LLmin$ be its minimal presentation.
Then we have
$$
\DLam^\prime \cap \OLmin = \D_{\LLmin}.
$$
\end{proposition}

%As in the preceding papers \cite{KMAAA2013} and \cite{MaJAMS2013},
%the $C^*$-algebra $\OLmin$ has a natural action of the circle group
%$\T$  gauge action written $\rho^\Lambda$.
%The fixed point algebra $\FLmin$ of $\OLmin$ is an AF-algebra with its diagonal
%algebra denoted by $\DLmin$.
%The commutative $C^*$-algebra $C(X_\Lambda)$
%of continuous functions on the right one-sided subshift $X_\Lambda$ 
%is denoted by $\DLam$.
%We know that the relative commutatant $\DLam^\prime \cap \OLmin$ coincides with
%$\DLmin$ (Proposition ????????).

%\newpage

%%%%%%%%%%%%%%%%%%%%%%%%%%%%%%%%%%%%%%%%%%%%%%%%%%%%%%%%%%
%%%%%%%%%%%%%%%%%%%%%%%%%%%%%%%%%%%%%%%%%%%%%%%%%%%%%%%%%%%%%%%
\section{Irreducible sofic shifts}
%%%%%%%%%%%%%%%%%%%%%%%%%%%%%%%%%%%%%%%%%%%%%%%%%%%%%%%%%%%%%

Let $\Lambda$ be  an irreducible sofic shift over alphabet $\Sigma$.
An irreducible sofic shift is defined by using an irreducible finite directed labeled graph.
It is realized as a factor of an irreducible shift of finite type.
The class of irreducible sofic shifts includes the class of irreducible shifts of finite type
 (see \cite{Fischer}, \cite{Kitchens}, \cite{Kr84}, \cite{Kr87}, \cite{LM}, \cite{Weiss}, etc.).
As in \cite{KMAAA2013}, \cite{MaIsrael2013} and \cite{MaJAMS2013}, 
irreducible sofic shifts are $\lambda$-synchronizing. 
Let
$G_\Lambda^F =(V_\Lambda^F, E_\Lambda^F, \lambda_\Lambda^F)$ 
be its irreducible left-resolving predecessor-separated finite labeled graph over $\Sigma$
that presents $\Lambda$,
where $(V_\Lambda^F, E_\Lambda^F)$ 
is a finite directed graph with vertex set $V_\Lambda^F$
and edge set $E_\Lambda^F$, 
and $\lambda_\Lambda^F: E_\Lambda^F \longrightarrow \Sigma$
is a labeling map.
It is well-known that 
such finite labeled graph always exists for $\Lambda$. 
It is minimal and unique up to graph isomorphism (\cite{Fischer}, \cite{LM}). 
The labeled graph is called the minimal left-resolving presentation of an irreducible 
sofic shift, or the left Fischer cover. 
Let 
$V_\Lambda^F =\{ v_1,\dots,v_N \}$ and 
$E_\Lambda^F =\{ e_1,\dots,e_M \}.$  
We will first define a labeled Bratteli diagram
$(V, E,\lambda)$ over $\Sigma$ as follows.
Let $V_0 =\{ v_0\}$ a singleton, and
$V_l = \{ v_1,\dots,v_N \}$ for $l \in \N$.
Let
$E_{0,1} = \{ f^0_1, \dots, f^0_M \}$
such that $s(f_i^0) = v_0, t(f^0_i) = t(e_i), \lambda(f^0_i) = \lambda^F(e_i)$    
for $i=1,2,\dots, M,$
and
$E_{l,l+1} = \{ f^l_1, \dots, f^l_M \}$ for $l\in \N$
such that $s(f^l_i) = s(e_i), t(f^l_i) = t(e_i), \lambda(f^l_i) = \lambda^F(e_i)$    
for $i=1,2,\dots, M.$

For $ v_i\in V_1$, put let $\Gamma_1^-(v_i)$ be its predecessor set for the vertex $v_i$,
that is defined by
$$
\Gamma_1^-(v_i) 
=\{ \lambda(f_n^0) \in \Sigma \mid t(f_n^0) = v_i\}, \qquad i=1,2,\dots, N.
$$
If $\Gamma^-_1(v_i) = \Gamma^-_1(v_j)$,
then the two vertices $v_i$ and $v_j$ are identified with each other in $V_1$,
and we have a new vertex set written $V^F_1$.
The sources $\{ s(f^0_1), \dots, s(f^0_M) \}$
of edges $\{ f^0_1, \dots, f^0_M \}$ are identified
following the identification in $V_1$, 
so that we obtain a new edge set 
written $E^F_{0,1}$.
Similarly, for $ v_i, v_j \in V_2,$
 if $\Gamma^-_2(v_i) =\Gamma^-_2(v_j)$,
then the two vertices
$v_i$ and $v_j$ are identified in $V_2$,
and   
the sources $\{ s(f^1_1), \dots, s(f^1_M) \}$
of edges $\{ f^1_1, \dots, f^1_M \}$ are identified
following the identification in $V_2$, so that we obtain a new edge set 
written $E^F_{1,2}$.
Like this way, we continue this procedure to get new vertex sets
$V^F_l, l=0,1,2,\dots $ and edge sets
$E^F_{l,l+1}, l=0,1,2,\dots. $ 
 Since $\Lambda$ is sofic and the original labeled graph
$G_\Lambda^F =(V_\Lambda^F, E_\Lambda^F, \lambda_\Lambda^F)$ 
is predecessor-separated,
there exists $K \in \N$ such that 
$\Gamma^-_k(v_i) \ne \Gamma^-_k(v_j)$ in $B_k(\Lambda)$
 for all $k \ge K$ and $i,j=1,2,\dots,N$ with $i\ne j,$
so that we have
$$
V_l^F = V_l (=V^F_\Lambda), \qquad 
E^F_{l,l+1} = E_{l,l+1} (=E_\Lambda^F)  \quad \text{ for all } l \ge K.
$$
We thus have a labeled Bratteli diagram
$(V^F_l, E^F_{l,l+1}, \lambda^F_{l,l+1})_{l\in \Zp}$ 
over $\Sigma$.
We let
$V^F_l = \{v_1^l, \dots,v_{m(l)}^l\}$.
Since $\Gamma^-_{l+1}(v_i) = \Gamma^-_{l+1}(v_j)$
implies 
$\Gamma^-_{l}(v_i) = \Gamma^-_{l}(v_j)$,
we have a natural surjective map
$V_{l+1}^F \longrightarrow V_l^F$ written $\iota^F_{l+1,l}$
for $l \le K$.
For $l\ge K$, the identity map 
$V^F_{l+1} \longrightarrow V^F_l$ written $\iota^F_{l+1,l}$
is defined.
We thus have a $\lambda$-graph system
$$
{\frak L}_\Lambda^F =(V^F, E^F, \lambda^F, \iota^F)
$$
that presents the original sofic shift $\Lambda$.
As the original labeled graph
$G_\Lambda^F =(V_\Lambda^F, E_\Lambda^F, \lambda_\Lambda^F)$
is minimal, left-resolving and hence predecessor-separated,
our $\lambda$-graph system
$
{\frak L}_\Lambda^F =(V^F, E^F, \lambda^F, \iota^F)
$
is left-resolving and predecessor-separated and presents $\Lambda$.
And also, every vertex $v_i$ of the directed graph 
$G_\Lambda^F $ has a word $\mu$ such that any directed labeled 
path labeled $\mu$ in $G_\Lambda^F $ must leave the vertex $v_i$
(cf. \cite[Proposition 3.3.17]{LM}), so that every vertex of
the $\lambda$-graph system ${\frak L}_\Lambda^F$
launches some word (see \cite[Section 3]{MaIsrael2013}).
Therefore the $\lambda$-graph system ${\frak L}_\Lambda^F$
is  $\lambda$-synchronizing. 
As $\Lambda$ is irreducible,
${\frak L}_\Lambda^F$ is $\iota$-irreducible
by Lemma \ref{lem:2.10} (i).
Hence ${\frak L}_\Lambda^F$ is nothing but the minimal $\lambda$-graph system 
${\frak L}^{\min}$ of $\Lambda$.
Therefore we have
\begin{proposition}
For an irreducible sofic shift $\Lambda$, 
let 
${\frak L}^{\min} =(V^{\min}, E^{\min}, \lambda^{\min}, \iota^{\min})$
be the minimal $\lambda$-graph system for $\Lambda$.
Let $G_\Lambda^F =(V^F, E^F,\lambda^F)$
be its Fischer cover graph for $\Lambda$.
Then there exists $L \in \N$ such that 
\begin{equation*}
V^{\min}_l = V_\Lambda^F, \quad 
E^{\min}_{l,l+1} = E_\Lambda^F, \quad
\lambda^{\min} =\lambda_\Lambda^F, \quad
\iota^{\min}|_{V_l^{\min}} = \id  
\end{equation*}
for all $l \ge L$.
\end{proposition}
Namely, 
the minimal $\lambda$-graph system 
${\frak L}^{\min}$
for an irreducible sofic shift $\Lambda$ 
is identified with its left Fischer cover.

Let $\Lambda$ be an irreducible sofic shift such that $\Lambda$ is not finite,
so that $\Lambda$ is a normal subshift.
Let $G_\Lambda^F =(V_\Lambda^F, E_\Lambda^F, \lambda_\Lambda^F)$
be its left Fischer cover graph with vertex set
$V_\Lambda^F =\{ v_1,\dots,v_N\}$.
Consider the following matrix :
\begin{equation} \label{eq:Aialphaj}
A(i,\alpha,j) 
= 
\begin{cases}
1 & \text{ if there exist } e \in E ; \lambda(e) = \alpha, s(e) = v_i, t(e) = v_j, \\
0 & \text{ otherwise. }
\end{cases}
\end{equation}
Let $S_\alpha, \alpha \in \Sigma$ and
$E_i, i=1,2,\dots,N$ be partial isometries and projections
respectively 
satisfying the following operator relations: 
\begin{equation}
1 =\sum_{\alpha \in \Sigma}S_\alpha S_\alpha^* = \sum_{i=1}^{N} E_i, 
\qquad S_\alpha S_\alpha^* E_i = E_i S_\alpha S_\alpha^*, 
\qquad
S_\alpha^* E_i S_\alpha = \sum_{j=1}^{N} A(i,\alpha, j) E_j \label{eq:SalphaEi}
\end{equation}
for $\alpha \in \Sigma,\ i=1, 2,\dots,N$.
Let us denote by ${\mathcal{O}}_{G^F_\Lambda}$
the universal $C^*$-algebra generated by 
$S_\alpha, \alpha \in \Sigma$ and
$E_i, i=1, 2, \dots,N$ satisfying the above relations.
We put 
\begin{equation*}
\widehat{\Sigma} 
=\{(\alpha,i) \in \Sigma \times \{1,2,\dots,N\} \mid
\text{ there exists } e \in E ; \lambda(e)= \alpha, t(e) = v_i \}.
\end{equation*}
For $(\alpha,i), (\beta,j) \in \widehat{\Sigma}$,
by using the matrix $A$ given by \eqref{eq:Aialphaj}, 
we define a matrix
\begin{equation*}
\widehat{A}((\alpha,i), (\beta,j)) =
\sum_{k=1}^N  A(k,\alpha,i) A(i,\beta,j).
\end{equation*}
Since the labeled graph $G_\Lambda^F$ is left-resolving,
the $(\alpha,i), (\beta,j)$-entry 
$\widehat{A}((\alpha,i), (\beta,j))
$
of the matrix 
$\widehat{A}$ is one or zero.
Let us denote by 
${\mathcal{O}}_{\widehat{A}}$
the Cuntz-Krieger algebra for the matrix $\widehat{A}$. 
We then have the following proposition.
\begin{proposition}\label{prop:4.2}
Let $\Lambda$ be an irreducible sofic shift such that $\Lambda$ is infinite.
Let $G_\Lambda^F =(V_\Lambda^F, E_\Lambda^F, \lambda_\Lambda^F)$
be its left Fischer cover graph.
Let $\LLmin$ be its minimal presentation of $\lambda$-graph system for the normal subshift
$\Lambda$.
Then the $C^*$-algebra
$\OLmin$ is a simple purely infinite $C^*$-algebra that 
%$$ \OLmin \cong {\mathcal{O}}_{G^F_\Lambda} \cong {\mathcal{O}}_{\widehat{A}}.
%$$
%Hence the $C^*$-algebra $\OLmin$ 
is isomorphic to the Cuntz-Krieger algebra 
${\mathcal{O}}_{\widehat{A}}$
for the matrix  $\widehat{A}$. 
\end{proposition}
\begin{proof}
By  the universality and the uniqueness of the $C^*$-algebra
$\OLmin$ for the canonical generating partial isometries 
$S_\alpha, \alpha \in \Sigma$
and projections $E_i^l, i=1,2, \dots, m(l), l\in \Zp$
subject to the relations $(\LLmin)$ as in Proposition \ref{prop:universalunique},
the $C^*$-algebra $\OLmin$ is canonically
isomorphic to the above $C^*$-algebra  
$ {\mathcal{O}}_{G^F_\Lambda}$.

We will henceforth show that 
$ {\mathcal{O}}_{G^F_\Lambda}$ is isomorphic to the Cuntz--Krieger algebra 
${\mathcal{O}}_{\widehat{A}}$.
Let  $S_\alpha, \alpha \in \Sigma$ and
$E_i, i=1,2, \dots,N$ be partial isometries and projections
respectively satisfying the operator relations
\eqref{eq:SalphaEi}.
For $(\alpha,i) \in \widehat{\Sigma}$,
put
$S_{(\alpha,i)} = S_\alpha E_i$.
We then have
\begin{equation*}
\sum_{(\alpha,i) \in \widehat{\Sigma}}S_{(\alpha,i)}S_{(\alpha,i)}^*
= \sum_{\alpha\in \Sigma}S_\alpha E_iS_\alpha^* 
=  \sum_{\alpha\in \Sigma}S_\alpha S_\alpha^* = 1.
\end{equation*}
As 
$
S_\alpha^*S_\alpha 
=\sum_{k=1}^N S_\alpha^* E_k S_\alpha
=\sum_{k=1}^N \sum_{j=1}^N A(k,\alpha,j) E_j,
$
we have
\begin{equation} \label{eq:Sai1}
S_{(\alpha,i)}^* S_{(\alpha,i)}
= E_i ( \sum_{k=1}^N \sum_{j=1}^N A(k,\alpha,j) E_j)  E_i 
=  \sum_{k=1}^N A(k,\alpha,i) E_i.
\end{equation}
Since
$S_\beta^* E_i S_\beta = \sum_{j=1}^N A(i,\beta,j) E_j$,
we have
\begin{equation}\label{eq:Sai2}
E_i = \sum_{\beta \in \Sigma}\sum_{j=1}^N A(i,\beta,j) S_\beta E_j S_\beta^*
    = \sum_{(\beta,j)\in \widehat{\Sigma}} A(i,\beta,j) S_{(\beta,j)} S_{(\beta,j)}^*.
\end{equation}
By \eqref{eq:Sai1} and \eqref{eq:Sai2},
we thus obtain
\begin{align*} 
S_{(\alpha,i)}^* S_{(\alpha,i)}
= &  \sum_{k=1}^N A(k,\alpha,i) (\sum_{(\beta,j)\in \widehat{\Sigma}} A(i,\beta,j) S_{(\beta,j)} S_{(\beta,j)}^*) \\
= &  \sum_{(\beta,j)\in \widehat{\Sigma}} \sum_{k=1}^N A(k,\alpha,i) A(i,\beta,j)) S_{(\beta,j)} S_{(\beta,j)}^* \\
= &  \sum_{(\beta,j)\in \widehat{\Sigma}} \widehat{A}((\alpha,i), (\beta,j)) S_{(\beta,j)} S_{(\beta,j)}^*. 
\end{align*}
Hence the $C^*$-algebra $C^*(S_{(\alpha,i)}; (\alpha,i) \in \widehat{\Sigma})$
generated by $S_{(\alpha,i)}, (\alpha,i) \in \widehat{\Sigma}$
is isomorphic to the Cuntz-Krieger algebra
${\mathcal{O}}_{\widehat{A}}$
for the matrix $\widehat{A}$.
By \eqref{eq:Sai2}, we have
\begin{equation*}
E_i= \sum_{(\beta,j)\in \widehat{\Sigma}} A(i,\beta,j) S_{(\beta,j)} S_{(\beta,j)}^*,
\qquad
S_\alpha = \sum_{i=1}^N S_\alpha E_i = \sum_{i=1}^N S_{(\alpha,i)}
\end{equation*}
so that $S_\alpha, E_i$ are generated by 
$S_{(\alpha,i)}, (\alpha,i) \in \widehat{\Sigma}$.
We thus have 
$C^*(S_\alpha, E_i; \alpha \in \Sigma, i=1,2,\dots,N) = C^*(S_{(\alpha,i)}; (\alpha,i) \in \widehat{\Sigma})$
and  hence 
${\mathcal{O}}_{G^F_\Lambda} ={\mathcal{O}}_{\widehat{A}}$.
\end{proof}
%By Theorem \ref{thm:main1.4} and Theorem \ref{thm:main1.5},
%we have the following theorem.
%\begin{theorem}\label{thm:soficconj}
%Let $\Lambda_1$ and $\Lambda_2$ be two irreducible subshifts such that 
%$\Lambda_i, i=1,2$ are infinite.
%\begin{enumerate}
%\renewcommand{\theenumi}{\roman{enumi}}
%\renewcommand{\labelenumi}{\textup{(\theenumi)}}
%\item Their one-sided subshifts 
%$(X_{\Lambda_1},\sigma_{\Lambda_1})$  and $(X_{\Lambda_2},\sigma_{\Lambda_2})$ 
%are eventually  conjugate if and only if there exists an isomorphism
%$\Phi:\OLamonemin\longrightarrow\OLamtwomin$ of simple purely infinite $C^*$-algebras such that 
%$\Phi({\mathcal{D}}_{\Lambda_1}) ={\mathcal{D}}_{\Lambda_2}$ and
%$\Phi\circ\rho^{\Lambda_1}_t = \rho^{\Lambda_2}_t\circ\Phi, t \in \T$.
%\item
%Their two-sided subshifts $({\Lambda_1},\sigma_{\Lambda_1})$ and 
%$({\Lambda_2},\sigma_{\Lambda_2})$ 
%are topologically   conjugate if and only if there exists an isomorphism
%$\widetilde{\Phi}:\OLamonemin\otimes\K\longrightarrow\OLamtwomin\otimes\K$ of %simple $C^*$-algebras such that 
%$\widetilde{\Phi}({\mathcal{D}}_{\Lambda_1}\otimes\C) ={\mathcal{D}}_{\Lambda_2}\otimes\C$ and $\widetilde{\Phi}\circ(\rho^{\Lambda_1}_t\otimes\id) 
%= (\rho^{\Lambda_2}_t\otimes\id) \circ \widetilde{\Phi}, t \in \T$.
%\end{enumerate}
%\end{theorem}

%\newpage

%%%%%%%%%%%%%%%%%%%%%%%%%%%%%%%%%%%%%%%%%%%%%%%
%%%%%%%%%%%%%%%%%%%%%%%%%%%%%%%%%%%%%%%%%%%%%%%
\section{Other examples of normal subshifts}
%%%%%%%%%%%%%%%%%%%%%%%%%%%%%%%%%%%%%%%%%%%%%%%
In this section, other examples of normal subshifts and their $C^*$-algebras 
will be presented.
 
\medskip

\noindent
{\bf 1. Dyck shifts.}

For a positive integer $N>1$, the Dyck shift $D_N$ of order $N$
was introduced by W. Krieger \cite{KriegerMST1974},
related to Dyck language in formal language theory in computer science (cf. \cite{HU}).
Consider an alphabet 
$\Sigma = \Sigma^+ \sqcup \Sigma^-$
where 
$\Sigma^- =\{ \alpha_1,\dots, \alpha_N\},
 \Sigma^+ =\{ \beta_1,\dots, \beta_N\}.
$
Following \cite{KriegerMST1974}, the Dyck inverse monoid
for $\Sigma$ is the inverse monoid defined by the product relations:
$ \alpha_i \beta_j = {\bf 1}$ if $i=j$, otherwise
$ \alpha_i \beta_j = 0$, for $i,j=1,\dots,N$.
The symbol ${\bf 1}$ plays a r\^{o}le of empty word
such that $\alpha_i {\bf 1} = {\bf 1} \alpha_i = \alpha_i,
\beta_j {\bf 1} = {\bf 1} \beta_j = \beta_j.$
By the product structure,
a word $\omega_1\cdots\omega_n$ of $\Sigma$ is defined to be admissible
if the reduced word of the product $\omega_1\cdots\omega_n$  
in the monoid is not $0$. 
The Dyck shift written $D_N$ is defined to be the subshift over
alphabet $\Sigma$ whose admissible words are the admissible words in this sense. 
It is well-known that the subshift $D_N$ is not sofic.
As in \cite{KMAAA2013},  the Dyck shift $D_N$ is
$\lambda$-synchronizing and hence normal.
Its minimal $\lambda$-graph system 
${\frak L}_{D_N}^{\min} =(V^\min, E^\min,\lambda^\min,\iota^\min)$
was already studied in \cite{KMDocMath2003}, in which 
the minimal $\lambda$-graph system ${\frak L}_{D_N}^{\min}$ 
was called the Cantor horizon $\lambda$-graph system 
written ${\frak L}^{Ch(D_N)}$.  
Let us briefly review its construction.

Let $\Lambda_N$ be the two-sided full $N$-shift over 
$\{1,2, \dots N\}$.
%Consider the sets $B_l(\Lambda_N), B_l(D_N)$ of admissible words of $\Lambda_N, D_N$ 
%of length $l$ for $\Lambda_N, D_N$, respectively.   
Let
\begin{equation}
V_l^{\min} := \{\beta_{\mu_1}\cdots \beta_{\mu_l} \in {(\Sigma^+)}^l
 \mid \mu_1\cdots\mu_l \in \{1,2, \dots,N\}^l \} \label{eq:Vlmin}
\end{equation}
and the mapping
$\iota^{\min}: V_{l+1}^{\min} \longrightarrow V_l^{\min}$
is defined by 
\begin{equation*}
\iota(\beta_{\mu_1}\cdots \beta_{\mu_l}\beta_{\mu_{l+1}})
= \beta_{\mu_1}\cdots \beta_{\mu_l}
\quad 
\text{ for }
\beta_{\mu_1}\cdots \beta_{\mu_l}\beta_{\mu_{l+1}}
 \in V_{l+1}^{\min}.
\end{equation*}
Define a labeled edge labeled $\alpha_j$ from 
the vertex
$\beta_{\mu_1}\cdots \beta_{\mu_l}
 \in  V_{l}^{\min}
$ 
to the vertex
$\beta_{\mu_0}\beta_{\mu_1}\cdots \beta_{\mu_l}
 \in  V_{l+1}^{\min}
$
precisely if $\mu_0 =j$.
Define  a labeled edge labeled $\beta_j$ from 
the vertex
$\beta_j \beta_{\mu_1}\cdots \beta_{\mu_{l-1}}
 \in  V_{l}^{\min}
$ 
to the vertex
$\beta_{\mu_1}\cdots \beta_{\mu_l}\beta_{\mu_{l+1}}
 \in  V_{l+1}^{\min}.
$
Such edges are denoted by $E_{l,l+1}^{\min}$.
We then have a $\lambda$-graph system presenting the Dyck shift $D_N$.
It is the minimal left-resolving presentation and hence it is the minimal 
$\lambda$-graph system ${\frak L}_{D_N}^{\min}$ (cf. \cite{MaJAMS2013}).
Since the subshift $D_N$ is $\lambda$-irreducible satisfying $\lambda$-condition (I),
 we know the following proposition.
\begin{proposition}[\cite{KMDocMath2003}, {\cite{MaJOT2007}, \cite{MaJAMS2013}}]
The $C^*$-algebra $\mathcal{O}_{D_N^{\min}}$ associated with
the  minimal $\lambda$-graph system ${\frak L}_{D_N}^{\min}$
for the Dyck shift $D_N$ is simple and purely infinite.
\end{proposition}
The K-groups of the algebra $\mathcal{O}_{D_N^{\min}}$
was computed in the following way:
\begin{equation*}
K_0(\mathcal{O}_{D_N^{\min}}) \cong \Z/ N\Z \oplus C({\frak C},\Z),
\qquad
K_1(\mathcal{O}_{D_N^{\min}}) \cong 0
\end{equation*}
where ${\frak C}$ denotes the Cantor set 
(\cite{KMDocMath2003}, \cite{MaJAMS2013}).

\medskip
  
\noindent
{\bf 2. Markov--Dyck shifts.}

The class of Markov--Dyck shifts contains the class of Dyck shifts.
It is a natural generalization of Dyck shifts as 
the class of topological Markov shifts contains the class of full shifts. 
Let $A =[A(i,j)]_{i,j=1}^N$
be an $N \times N$ square matrix with entries in $\{0,1\}$.
We assume that the matrix is irreducible satisfying condition (I)
 in the sense of Cuntz--Krieger \cite{CK}. 
The Markov--Dyck shift $D_A$ for the matrix $A$
 is defined by using the canonical generating partial isometries of the Cuntz--Krieger algebra $\OA$ in the following way.
Let $s_1,\dots, s_N$ be the canonical generating partial isometries of the Cuntz--Krieger algebra $\OA$ that satisfies the relations:
\begin{equation*}
1 =\sum_{j=1}^N s_j s_j^*,\qquad 
s_i^* s_i =\sum_{j=1}^N A(i,j) s_j s_j^*, \quad i=1,2, \dots, N.
\end{equation*}
Similarly to the Dyck shift, we consider 
the alphabet  
$\Sigma = \Sigma^+ \sqcup \Sigma^-$
where 
$\Sigma^- =\{ \alpha_1,\dots, \alpha_N\},$
$ \Sigma^+ =\{ \beta_1,\dots, \beta_N\}.
$
Let
$\hat{\alpha}_i = s_i^*, \hat{\beta}_i = s_i,
i=1,2, \dots,N$.
We say that a word 
$\gamma_1\cdots\gamma_n $ of 
$\Sigma$ for 
$\gamma_1, \dots, \gamma_n \in \Sigma$
is forbidden if 
 $\hat{\gamma}_1\cdots\hat{\gamma}_n =0$
in the algebra $\OA$.
The Markov-Dyck shift $D_A$ for the matrix $A$
is defined by the subshift over alphabet $\Sigma$ by the 
 forbidden words.
These kinds of subshifts first
appeared in \cite{Kr2006BLM} by using certain semigroups.
More general setting was studied in \cite{HIK}.
The above definition by using generators of $C^*$-algebras 
was seen in \cite{MaMathScand2011} (cf. \cite{MaDCDS2015}).
If all entries of $A$ is one, then 
the product structure 
of $\hat{\alpha}_i, \hat{\beta}_i, i=1, 2, \dots,N$
go to that of the Dyck inverse monoid,
so that 
the Markov-Dyck shift $D_A$ coincides with the Dyck shift $D_N$.

The Markov--Dyck shift $D_A$ is not sofic for every irreducible matrix 
$A$ with entries in $\{0,1\}$ satisfying condition (I).
It is  always $\lambda$-synchronizing and hence normal.
Hence we have its minimal $\lambda$-graph system
${\frak L}_{D_A}^{\min}$ for $D_A$.
The $\lambda$-graph system was studied in \cite{KMDocMath2003}
in which it was called the Cantor horizon $\lambda$-graph system 
and written ${\frak L}^{Ch(D_A)}$.
Let $\Lambda_A$ denotes the shift space 
\begin{equation*}
\Lambda_A = \{ (x_n)_{n\in \Z} \in \{1,\dots,N\}^\Z
\mid A(x_n, x_{n+1}) =1 \text{ for all } n \in \Z\}
\end{equation*}
of the two-sided topological Markov shift defined by the matrix $A$.
We denote by $B_l(\Lambda_A)$ 
the set of admissible words of $\Lambda_A$ with its length $l$.
The vertex set $V^{\min}_l$ at level $l$ of the minimal $\lambda$-graph system  
${\frak L}_{D_A}^{\min}$ is defined by
\begin{equation*}
V_l^{\min} := \{\beta_{\mu_1}\cdots \beta_{\mu_l} \in {(\Sigma^+)}^l
 \mid \mu_1\cdots\mu_l \in B_l(\Lambda_A) \}.
\end{equation*}
The mapping
$\iota^{\min}: V_{l+1}^{\min} \longrightarrow V_l^{\min}$
is similarly defined to the minimal $\lambda$-graph system 
${\frak L}_{D_N}^{\min}$
of the Dyck shift by deleting its rightmost symbol of words 
in $V_{l+1}^{\min}.$
A labeled edge labeled $\alpha_j$ from 
$\beta_{\mu_1}\cdots \beta_{\mu_l}\in V_l^{\min}$
to 
$\beta_{\mu_0}\beta_{\mu_1}\cdots \beta_{\mu_l}\in V_{l+1}^{\min}$
is defined 
precisely if $\mu_0 =j$.
A labeled edge labeled $\beta_j$ from 
the vertex
$\beta_j \beta_{\mu_1}\cdots \beta_{\mu_{l-1}}
 \in V_l^{\min}
$ 
to the vertex
$\beta_{\mu_1}\cdots \beta_{\mu_l}\beta_{\mu_{l+1}}
 \in V_{l+1}^{\min}
$
is defined.
Such edges are denoted by $E_{l,l+1}^{\min}$.
We then have a $\lambda$-graph system presenting the Markov--Dyck shift $D_A$.
It is the minimal left-resolving presentation and hence it is the minimal 
$\lambda$-graph system ${\frak L}_{D_A}^{\min}$ (cf. \cite{MaJAMS2013}).
Since the matrix $A$ is irreducible and satisfies condition (I),
the subshift $D_A$ is $\lambda$-irreducible satisfying $\lambda$-condition (I),
 so that we know the following proposition.
\begin{proposition}[{\cite{KMDocMath2003}, \cite{MaJAMS2013}}]
The $C^*$-algebra $\mathcal{O}_{D_A^\min}$ associated with
the  minimal $\lambda$-graph system ${\frak L}_{D_A^\min}$
for the Makov--Dyck shift $D_A$ is simple and purely infinite.
\end{proposition}
K-group formulas for the $C^*$-algebras 
$\mathcal{O}_{D_A^\min}$ were studied in \cite{MaMathScand2011}.

\medskip

%%%%%%%%%%%%%%%%%%%%%%%%%%%%%%%%%%%%%%%%%%%%%%%
%

\noindent
{\bf 3. Motzkin shifts.}

Motzkin language  
appears in automata theory  as well as Dyck language (\cite{HU}).
The Motzkin shifts are non sofic subshifts associated with the Motzkin language
(cf. \cite{MaMZ2004}).
For a positive integer $N>1$,
similarly to the Dyck shift, we consider 
the alphabet  
$\Sigma =\Sigma^+ \sqcup \Sigma^-$
where 
$\Sigma^- =\{ \alpha_1,\dots, \alpha_N\},
 \Sigma^+ =\{ \beta_1,\dots, \beta_N\}
$
and the Dyck inverse monoid
for $\Sigma^+ \sqcup \Sigma^-$ as in previous paragraphs.
The Dyck inverse monoid is defined by the product relations:
$ \alpha_i \beta_j = {\bf 1}$ if $i=j$, otherwise
$ \alpha_i \beta_j = 0$, for $i,j=1,\dots,N$.
Let us consider a new alphabet set
$\Sigma_{\bf 1}$ 
defined by
$\Sigma_{\bf 1} =\Sigma^+ \cup \Sigma^-\cup \{{\bf 1}\}$. 
The Motzkin shift $M_N$ of order $N$ is defined to be a subshift
over $\Sigma_{\bf 1}$ such that a word 
$\gamma_1\cdots\gamma_n$ of $\Sigma_{\bf 1}$
is forbidden  precisely if
$\gamma_1\cdots\gamma_n =0$. 
As seen in \cite{MaMZ2004},
the subshift $M_N$ is $\lambda$-synchronizing and hence normal.
Its minimal $\lambda$-graph system ${\frak L}_{M_N}^{\min}$
was described as the Cantor horizon $\lambda$-graph system written
${\frak L}^{Ch(M_N)}$ in \cite{MaMZ2004}. 
Let $V_l^{\min}$ be the vertex set defined by \eqref{eq:Vlmin}.
The mapping $\iota: V_{l+1}^{\min}\longrightarrow V_{l}^{\min}$
is similarly defined as in the case of Dyck shifts.
Labeled edges labeled symbols in $\Sigma$ from $V_l^{\min}$ to $V_{l+1}^{\min}$
are defined in a similar way to Dyck shifts.
In addition to the labeled edges above, 
an additional labeled edge labeled ${\bf 1}$ from
$\beta_{\mu_1}\cdots \beta_{\mu_l}\in V_l^\min$
to
$\beta_{\mu_1}\cdots \beta_{\mu_l}\beta_{\mu_{l+1}}\in V_{l+1}^\min$
is defined for every pair 
$\beta_{\mu_1}\cdots \beta_{\mu_l}\in V_l^\min$
and
$\beta_{\mu_1}\cdots \beta_{\mu_l}\beta_{\mu_{l+1}}\in V_{l+1}^\min$.
We then have a $\lambda$-graph system that is 
the minimal $\lambda$-graph system
${\frak L}_{M_N}^{\min}$
for the Motzkin shift $M_N$.
Since the $\lambda$-graph system
${\frak L}_{M_N}^{\min}$ contains the minimal $\lambda$-graph system
${\frak L}_{D_N}^{\min}$ of the Dyck shift $D_N$ as a subsystem,
${\frak L}_{M_N}^{\min}$ is $\lambda$-irreducible and satisfies $\lambda$-condition (I).
Therefore we have
\begin{proposition}[{\cite{MaMZ2004}}]
The $C^*$-algebra $\mathcal{O}_{M_N^\min}$ associated with
the  minimal $\lambda$-graph system ${\frak L}_{M_N}^{\min}$
for the Motzkin shift $M_N$ is simple and purely infinite.
\end{proposition}
The K-groups of the algebra $\mathcal{O}_{M_N^\min}$
was computed in \cite{MaMZ2004}
for the case of $N=2$.
As in the paper \cite{MaMZ2004},
the strategy to compute $K_i(\mathcal{O}_{M_N^\min}), i=1,2$
works well for general  $\mathcal{O}_{M_N^\min}, N=2,3,\dots$,
so that we have:
\begin{equation*}
K_0(\mathcal{O}_{M_N^\min}) \cong C({\frak C},\Z),
\qquad
K_1(\mathcal{O}_{M_N^\min}) \cong 0
\end{equation*}
where ${\frak C}$ denotes the Cantor set 
(\cite{MaMZ2004}).

 \medskip
 
 %%%%%%%%%%%%%%%%%%%%%%%%%%%%%%%%%%%%%%%%%%%

\noindent
{\bf 4. $\beta$-shifts.}

The  $\beta$-shift for real number $\beta>1$
was first introduced in \cite{Parry}, \cite{Renyi}.
It is an interpolation between full shifts, 
simultaneously   one of natural generalization of full shifts.
For a real number $\beta >1$,
take a natural number $N$ such that 
$N -1 < \beta \le N$.
Let $f_\beta:[0,1]\longrightarrow [0,1]$
be the mapping $f_\beta(x) = \beta x - [\beta x]$
for $x \in [0,1]$, where $[t]$ is the integer part of $t \in \R$.
Let $\Sigma =\{0,1,\dots, N-1 \}$.
The $\beta$-expansion of $x \in [0,1]$ is a sequence
$d_i(x,\beta), i\in \N$ of $\Sigma$ defined by
$$
d_i(x,\beta) = [\beta f_\beta^{i-1}(x)], \qquad i \in \N,
$$
so that we know that 
$x = \sum_{i=1}^\infty \frac{d_i(x,\beta)}{\beta^i}.$
We endow $\Sigma^\N$ with the lexicographical order.
Put
$\zeta_\beta =\sup_{x \in [0,1)}(d_i(x,\beta))_{i\in \N}$.
Define the one-sided subshift $X_{\Lambda_\beta}$
by
setting
\begin{equation*}
X_{\Lambda_\beta} = \{\omega \in \Sigma^\N \mid \sigma^i(\omega) \le \zeta_\beta, i\in \Zp\},
\end{equation*}
where $\sigma^i(\omega) =(\omega_{n+i})_{n\in \N}$ 
for $\omega =(\omega_n)_{n\in \N}$.
Its two-sided extension $\Lambda_\beta$ is defined by
\begin{equation*}
\Lambda_\beta = \{(\omega_n)_{n\in \Z} \in \Sigma^\Z 
\mid (\omega_{n+k})_{n\in \N}\in X_{\Lambda_\beta}, k\in \Z\}.
\end{equation*}
Put the maximum sequence 
$\zeta_\beta =(\xi_1,\xi_2,\dots )$
and the real number 
\begin{equation*}
b_{\xi_1\dots\xi_k} = \beta^k -\xi_1 \beta^{k-1}-\xi_2 \beta^{k-2} -\cdots -\xi_{k-1} \beta -\xi_k
\end{equation*}
As seen in \cite[Proposition 3.8]{KMW}, 
\begin{enumerate}
\renewcommand{\theenumi}{\roman{enumi}}
\renewcommand{\labelenumi}{\textup{(\theenumi)}}
\item
$\Lambda_\beta$ is a full shift if and only if $b_{\xi_1} =1$.
\item
$\Lambda_\beta$ is a shift of finite type  if and only if $b_{\xi_1\cdots\xi_k} =1$ for some $k \ge 1$.
\item
$\Lambda_\beta$ is a sofic subshift if and only if $b_{\xi_1\cdots\xi_l} =b_{\xi_1\cdots\xi_m}$ for some $l\ne m$.
\end{enumerate}
Hence $\Lambda_\beta$ is not sofic
unless $\beta$ is an algebraic integer. 
As in \cite{KMAAA2013}, the $\beta$-shift $\Lambda_\beta$ is
$\lambda$-synchronizing for every $\beta$, so that it is normal.
In \cite{KMW}, the $C^*$-algebra $\mathcal{O}_\beta$ of the $\beta$-shift $\Lambda_\beta$ 
was studied (cf. \cite{IJM1997}).
The $C^*$-algebra $\mathcal{O}_\beta$ is indeed the $C^*$-algebra 
$\mathcal{O}_{{\frak L}_{\Lambda_\beta}^{\min}}$
associated with the minimal $\lambda$-graph system
${\frak L}_{\Lambda_\beta}^{\min}$ of the subshift $\Lambda_\beta$.

We will briefly review  the construction of ${\frak L}_{\Lambda_\beta}^{\min}$
seen in \cite{KMW}.
For $l \in \N$, order the real numbers 
$\{b_{\xi_1},b_{\xi_1\xi_2},\dots ,b_{\xi_1\xi_2\cdots\xi_l} \}$
by its usual order in $\R$.
They give rise to
disjoint intervals partitioned by 
$\{b_{\xi_1},b_{\xi_1\xi_2},\dots ,b_{\xi_1\xi_2\cdots\xi_l} \}$
in $[0,1]$.
Let $m(l)$ be the number of the partitions in  $(0,1]$.
If $\Lambda_\beta$ is  sofic,
there exists $l_0$ such that $m(l) = L$ for all $l >l_0$.
If $\Lambda_\beta$ is  not sofic, $m(l) = l+1$.
Let 
$v_1^l, \dots, v_{m(l)}^l$ 
be the ordered set of 
the disjoint partitions of $(0,1]$.
The order is defined along the usual order in $\R$.
We denote by $V_l^{\min}$ 
the set $\{ v_1^l, \dots, v_{m(l)}^l \}$.
Suppose that 
$v_i^l $ corresponds to the interval 
$(b_{\xi_1\cdots\xi_q},  
 b_{\xi_1\cdots\xi_p}]$
with
$b_{\xi_1\cdots\xi_q} < b_{\xi_1\cdots\xi_p}$. 
For $\xi_{p+1} \in \Sigma$,
we define the labeled edge labeled $\xi_{p+1}$ from $v_i^l$ to
the vertices $v_j^{l+1} \in V_{l+1}^{\min}$ 
 corresponding to the partitions
 contained in the interval
$( b_{\xi_1\cdots\xi_q \xi_{p+1}}, b_{\xi_1\cdots\xi_p\xi_{p+1}}]$.
For $0\le \alpha <\xi_{p+1}$,
we define 
the labeled edge labeled $\alpha$ from $v_i^l$ to
the vertices $v_j^{l+1} \in V_{l+1}^{\min}$ 
 corresponding to the partitions
 contained in the interval
$( b_{\xi_1\cdots\xi_q \alpha}, 1]$.
Such edges are written $E_{l,l+1}^{\min}$.
We define the map $\iota^{\min}:V_{l+1}^{\min}\longrightarrow V_l^{\min}$
by setting $\iota(v_j^{l+1}) = v_i^l$ 
if the part in $(0,1]$ corresponding to $v_j^{l+1}$
is contained in the part in $(0,1]$ corresponding to $v_i^{l+1}$.
The resulting labeled Bratteli diagram becomes a $\lambda$-graph system.
It is not difficult to see that the $\lambda$-graph system is 
$\lambda$-synchronizing and hence minimal (cf. \cite{KMAAA2013}).  
The $C^*$-algebra $\mathcal{O}_{\beta}$ studied in \cite{KMW}
is generated by a finite family $S_0, S_1,\dots, S_{N-1}$ of partial isometries
corresponding to the letters of $\Sigma$. 
For an admissible word $\mu\in B_*(\Lambda_\beta)$,
put $a_\mu = S_\mu^* S_\mu$.
It was proved in \cite{KMW} that there exists a unique KMS-state
written $\varphi$ for gauge action on $\mathcal{O}_{\beta}$ 
(cf. \cite{KMW}).
As in \cite{KMW}, we know  
\begin{equation*}
 \varphi(a_{\xi_1\xi_2 \cdots\xi_k})= b_{\xi_1\xi_2 \cdots\xi_k},
 \qquad k \in \N.
%\beta^k - \xi_1 \beta^{k-1} - \xi_2 \beta^{k-2} - \cdots -\xi_{k-1}\beta - \xi_k
%( = \sum_{i=1}^{\infty}\frac{\xi_{k+i}}{\beta^{i}}).
\end{equation*}
By \cite[Corollary 3.2]{KMW}, we see 
\begin{equation} \label{eq:relationobeta}
  S_\alpha^{*}a_{\xi_1 \cdots \xi_n}S_\alpha = 
\begin{cases}
0        &     \quad   \alpha > \xi_{n+1} \\
a_{\xi_1 \cdots \xi_{n+1}}  &  \quad  \alpha= \xi_{n+1}\\
1        &     \quad  \alpha< \xi_{n+1}.
\end{cases}
\end{equation}
Since the projections
in the commutative $C^*$-algebra $\A_\beta$ 
generated by the projections of the form
$a_\mu ,\mu \in B_*(\Lambda_\beta)$
is generated by the projection of the form
$E_i^l: = b_{\xi_1\cdots\xi_p} -b_{\xi_1\cdots\xi_q}$, 
the relation \eqref{eq:relationobeta}
tells us that the $C^*$-algebra 
$\mathcal{O}_{\Lambda_\beta^\min}$
%$\mathcal{O}_{{\frak L}_{\Lambda_\beta}^{\min}}$
associated with the minimal $\lambda$-graph system 
${\frak L}_{\Lambda_\beta}^{\min}$ is canonically isomorphic to 
the $C^*$-algebra $\mathcal{O}_\beta$ studied in \cite{KMW}.
We therefore have 
\begin{proposition}[{\cite[Theorem 3.6 and Theorem 4.12]{KMW}}] 
The $C^*$-algebra $\mathcal{O}_{\Lambda_\beta^\min}$ 
for the $\beta$-shift $\Lambda_\beta$
is simple and purely infinite for each $1< \beta \in \R$
such that 
\begin{align*}
K_0(\mathcal{O}_{\Lambda_\beta^\min}) & =
{\begin{cases}
 \quad \Z / (\eta_1 + \cdots + \eta_m - 1) \Z & \text{ if }
      d(1,\beta) = \eta_1 \eta_2 \cdots \eta_m 000 \cdots   \\
 \quad \Z / (\xi_{1} + \cdots + \xi_{k}) \Z & \text{ if }
      d(1,\beta) = \nu_1 \cdots \nu_l \xi_{1} \cdots \xi_{k}
                                      \xi_{1} \cdots \xi_{k}
                                      \xi_{1} \cdots \xi_{k}
                                        \cdots  \\
 \quad \Z & \text{ otherwise. }
\end{cases}} \\
K_1(\mathcal{O}_{\Lambda_\beta^\min}) & = \{ 0 \}  \qquad \text{ for any } \beta >1.
\end{align*}
\end{proposition}
\begin{remark}
It was shown that the KMS-state for the gauge action on 
$\mathcal{O}_\beta$
is unique at the inverse temperature $\log\beta$,
 which is the topological entropy for the $\beta$-shift $\Lambda_\beta$
(\cite{KMW}).
 Hence  two subshifts $\Lambda_\beta, \Lambda_{\beta'}$ are topologically conjugate if and only if 
 $\beta = \beta'$.
\end{remark}

%\newpage
%%%%%%%%%%%%%%%%%%%%%%%%%%%%%%%%%%%%%%%%%%%%
%%%%%%%%%%%%%%%%%%%%%%%%%%%%%%%%%%%%%%%%%%%%%
\section{Continuous orbit equivalence}
%%%%%%%%%%%%%%%%%%%%%%%%%%%%%%%%%%%%%%%%%%%%%%%
In this section, we will discuss continuous orbit equivalence in normal subshifts.
Let ${\frak L}_1,{\frak L}_2$ be left-resolving $\lambda$-graph systems
and
$(\Lambda_1, \sigma_{\Lambda_1}),(\Lambda_2, \sigma_{\Lambda_2})$
their presenting two-sided subshifts, respectively.
In \cite{MaDynam2020}, 
the notion of $({\frak L}_1,{\frak L}_2)$-continuous orbit equivalence between
their one-sided subshifts
$(X_{\Lambda_1}, \sigma_{\Lambda_1}),(X_{\Lambda_2}, \sigma_{\Lambda_2})$
was introduced in the following way.

\begin{definition}[{\cite[Definition 4.1]{MaDynam2020}, \cite[Section 6]{MaYMJ2010}}] 
\label{def:coelambda}
 One-sided subshifts 
$(X_{\Lambda_1},\sigma_{\Lambda_1})$
and $(X_{\Lambda_2},\sigma_{\Lambda_2})$
are said to be
$({\frak L}_1,{\frak L}_2)$-{\it continuously orbit equivalent}\/ 
if there exist two homeomorphisms
$h_{\frak L}: X_{{\frak L}_1}\longrightarrow X_{{\frak L}_2}$
and
$h_{\Lambda}: X_{\Lambda_1}\longrightarrow X_{\Lambda_2}$
and continuous functions
$k_i, l_i : X_{{\frak L}_i} \longrightarrow \Zp, i=1,2$ such that  
$\pi_{{\frak L}_2} \circ h_{\frak L} =h_{\Lambda}\circ\pi_{{\frak L}_1}$
and
\begin{align}
\sigma_{{\frak L}_2}^{k_1(x)}(h_{\frak L}(\sigma_{{\frak L}_1}(x)))
& = 
\sigma_{{\frak L}_2}^{l_1(x)}(h_{\frak L}(x)), \qquad x \in X_{{\frak L}_1}, \label{eq:coex} \\ 
\sigma_{{\frak L}_1}^{k_2(y)}(h_{\frak L}^{-1}(\sigma_{{\frak L}_2}(y)))
& = 
\sigma_{{\frak L}_1}^{l_2(y)}(h_{\frak L}^{-1}(y)), \qquad y \in X_{{\frak L}_2}. \label{eq:coey}
\end{align}
\end{definition}
We first show the following lemma.
\begin{lemma}\label{lem:coe6.2}
Let ${\frak L}_1,{\frak L}_2$ be left-resolving $\lambda$-graph systems
satisfying condition (I) and
$(\Lambda_1, \sigma_{\Lambda_1}),(\Lambda_2, \sigma_{\Lambda_2})$
their presenting two-sided subshifts, respectively.
Suppose that 
one-sided subshifts 
$(X_{\Lambda_1},\sigma_{\Lambda_1})$
and $(X_{\Lambda_2},\sigma_{\Lambda_2})$
are 
$({\frak L}_1,{\frak L}_2)$-continuously orbit equivalent. 
If $\Lambda_1$ is a normal subshift and ${\frak L}_1$ is its minimal
presentation of $\Lambda_1$,
then $\Lambda_2$ is also normal and 
${\frak L}_2$ is its minimal
presentation.
\end{lemma}
\begin{proof}
Assume that the one-sided subshifts 
$(X_{\Lambda_1},\sigma_{\Lambda_1})$
and $(X_{\Lambda_2},\sigma_{\Lambda_2})$
are 
$({\frak L}_1,{\frak L}_2)$-continuously orbit equivalent
and 
${\frak L}_1$ is the minimal presentation of the normal subshift
$\Lambda_1$.
By \cite[Theorem 1.2]{MaDynam2020},
there exists an isomorphism
$\Phi: \mathcal{O}_{{\frak L}_1}
\longrightarrow 
\mathcal{O}_{{\frak L}_2}
$
of $C^*$-algebras
such that 
$\Phi(\D_{\Lambda_1}) = \D_{\Lambda_2}$.
Now ${\frak L}_1 = \LLamonemin$, so that 
we may write 
$\mathcal{O}_{{\frak L}_1}
=
\OLamonemin$.
Let
$S_\alpha^1, E_i^{1l}$
and
$S_\alpha^2, E_i^{2l}$
be the canonical generators of the $C^*$-algebras
$
\OLamonemin$
and
$\mathcal{O}_{{\frak L}_2},
$
respectively.
By Proposition \ref{prop:relativecom},
the condition
$\Phi(\D_{\Lambda_1}) = \D_{\Lambda_2}$
implies 
$\Phi(\D_{{\frak L}_1}) = \D_{{\frak L}_2}$.
Hence 
for a vertex $v_i^{2l}$ in ${\frak L}_2$ 
and the corresponding projection 
$E_i^{2l}  \in \mathcal{O}_{{\frak L}_2}$,
we have 
$\Phi^{-1}(E_i^{2l}) \in \D_{{\frak L}_1}.$
We may find
a word $\nu \in B_*(\Lambda_1)$ and a vertex $v_j^{1l}$ in ${\frak L}_1$
such that 
\begin{equation*}%\label{eq:PhiEilno1}
\Phi^{-1}(E_i^{2l}) \ge S_\nu^1 E_j^{1l}S_\nu^{1*},
\qquad   S_\nu^{1*} S_\nu^1 \ge E_j^{1l}.
\end{equation*}
Now $\Lambda_1$ is normal, there exists 
a word $\eta \in B_*(\Lambda_1)$ such that 
%\begin{equation*}%\label{eq:PhiEilno2}
$ E_j^{1l} \ge S_\eta^1 S_\eta^{1*}
$
%\end{equation*}
by \cite[Proposition 3.3]{MaJAMS2013},
so that 
\begin{equation}\label{eq:PhiEilno3}
 S_\nu^1 E_j^{1l}S_\nu^{1*} \ge
 S_\nu^{1} S_\eta^1 S_\eta^{1*} S_\nu^{1*} \ne 0.
\end{equation}
Hence we have
\begin{equation}\label{eq:PhiEilno4}
  E_i^{2l} \ge 
\Phi( S_{\nu\eta}^1 S_{\nu\eta}^{1*}).
\end{equation}
Since
$\Phi(\D_{\Lambda_1}) = \D_{\Lambda_2}$,
one may find 
$\mu \in B_*(\Lambda_2)$ such that 
\begin{equation}\label{eq:PhiEilno5}
 \Phi( S_{\nu\eta}^1 S_{\nu\eta}^{1*})
\ge S_\mu^2 S_\mu^{2*}.
\end{equation}
By \eqref{eq:PhiEilno4}, \eqref{eq:PhiEilno5},
we have
\begin{equation*}%\label{eq:PhiEilno5}
 E_i^{2l}
\ge S_\mu^2 S_\mu^{2*}.
\end{equation*}
This implies that the vertex 
$v_i^{2l}$ in ${\frak L}_2$
launches $\mu$
by \cite[Proposition 3.3]{MaJAMS2013}
so that
the $\lambda$-graph system
${\frak L}_2$ is $\lambda$-synchronizing.
Therefore we conclude that 
the subshift $\Lambda_2$ is normal and 
${\frak L}_2$ is its minimal presentation. 
 \end{proof}
Now the following definition seems to be reasonable.  
\begin{definition}\label{def:coefornormal}
Let $(\Lambda_1,\sigma_1)$
and $(\Lambda_2,\sigma_2)$
be normal subshifts.
Their one-sided subshifts
$(X_{\Lambda_1},\sigma_1)$
and $(X_{\Lambda_2},\sigma_2)$
are said to be 
{\it continuously orbit equivalent}\/ 
if they are $(\LLamonemin,\LLamtwomin)$-continuously orbit equivalent.
\end{definition} 
Therefore we know the following proposition.
 \begin{proposition}[{\cite[Theorem 1.2]{MaDynam2020}}]
\label{prop:onesidednormalcoe}
Let $(\Lambda_1, \sigma_{\Lambda_1})$ 
and 
$(\Lambda_2, \sigma_{\Lambda_2})$
be normal subshifts.
Then the following two assertions are equivalent:
\begin{enumerate}
\renewcommand{\theenumi}{\roman{enumi}}
\renewcommand{\labelenumi}{\textup{(\theenumi)}}
\item  
Their one-sided subshifts 
$(X_{\Lambda_1},\sigma_{\Lambda_1})$
and
$(X_{\Lambda_2},\sigma_{\Lambda_2})$
are continuously orbit equivalent.
\item
There exists an isomorphism
$\Phi: \OLamonemin
\longrightarrow \OLamtwomin
$ of $C^*$-algebras
such that 
$\Phi({\mathcal{D}}_{\Lambda_1})={\mathcal{D}}_{\Lambda_2}.
$ 
\end{enumerate}
\end{proposition}

We note the following proposition.
 \begin{proposition}
\label{prop:coesftsofic}
Let $(\Lambda_1, \sigma_{\Lambda_1})$ 
and 
$(\Lambda_2, \sigma_{\Lambda_2})$
be normal subshifts
such that 
their one-sided subshifts 
$(X_{\Lambda_1},\sigma_{\Lambda_1})$
and
$(X_{\Lambda_2},\sigma_{\Lambda_2})$
are continuously orbit equivalent.
\begin{enumerate}
\renewcommand{\theenumi}{\roman{enumi}}
\renewcommand{\labelenumi}{\textup{(\theenumi)}}
\item  
$(X_{\Lambda_1},\sigma_{\Lambda_1})$
is a shift of finite type if and only if 
$(X_{\Lambda_2},\sigma_{\Lambda_2})$
is a shift of finite type.
\item  
$(X_{\Lambda_1},\sigma_{\Lambda_1})$
is a sofic shift if and only if 
$(X_{\Lambda_2},\sigma_{\Lambda_2})$
is a sofic shift.
\end{enumerate}
\end{proposition}
\begin{proof}
The minimal presentations 
$\LLamonemin, \LLamtwomin$
of
$\Lambda_1, \Lambda_2$
are written 
${\frak L}_1, {\frak L}_2$,
respectively.

(i)
It is easy to see that 
a normal subshift $\Lambda$ is a shift of finite type
if and only if
$\D_\Lambda = \D_{\LLmin}$.  
Now 
there exists an isomorphism
$\Phi: \OLamonemin
\longrightarrow \OLamtwomin
$ of $C^*$-algebras
such that 
$\Phi({\mathcal{D}}_{\Lambda_1})={\mathcal{D}}_{\Lambda_2}.
$ 
Since 
$\D_{\LLmin} =\D_\Lambda^\prime\cap\OLmin$
for a normal subshift $\Lambda$,
we know that 
$\D_{{\frak L}_1} = {\mathcal{D}}_{\Lambda_1}$
if and only if
$\D_{{\frak L}_2} = {\mathcal{D}}_{\Lambda_2}.$
Hence $(X_{\Lambda_1},\sigma_{\Lambda_1})$
is a shift of finite type 
if and only if 
$(X_{\Lambda_2},\sigma_{\Lambda_2})$
is a shift of finite type.

(ii)
Suppose that 
$(X_{\Lambda_1},\sigma_{\Lambda_1})$
is sofic.
As in Section 4, the dynamical system 
$(X_{{\frak L}_1}, \sigma_{{\frak L}_1})$
is a shift of finite type.
We know that the class of shifts of finite type 
is preserved under continuous orbit equivalence 
by the above discussion (i).
By definition,  the shift of finite type
$(X_{{\frak L}_1}, \sigma_{{\frak L}_1})$
is continuously orbit equivalent 
to 
$(X_{{\frak L}_2}, \sigma_{{\frak L}_2})$
as shifts of finite type (cf. \cite{MaPacific2010}).
Hence $(X_{{\frak L}_2}, \sigma_{{\frak L}_2})$ is a shift of finite type.
As there exists a factor map
$\pi_2:X_{{\frak L}_2} \longrightarrow X_{\Lambda_2}$
such that 
$\pi_2\circ \sigma_{{\frak L}_2} = \sigma_{\Lambda_2}\circ\pi_2$,
we see that 
$(X_{\Lambda_2},\sigma_{\Lambda_2})$
is a sofic shift by \cite{Weiss}.
\end{proof} 
 \begin{proposition}
\label{prop:coesofic}
Let $(\Lambda_i, \sigma_{\Lambda_i}), i=1,2$
be sofic shifts and
$G_{\Lambda_i}^F, i=1,2$
be its left Fischer cover graphs.
Let us denote by
$\widehat{A}_i, i=1,2$ the transition matrices
of the graphs $G_{\Lambda_i}^F, i=1,2$.
Let $\pi_i:X_{\widehat{A}_i}\longrightarrow   
X_{\Lambda_i}, i=1,2
$
be the natural factor maps from the shifts of finite type
$X_{\widehat{A}_i}$ to the sofic shifts 
$   
X_{\Lambda_i}, i=1,2.$
Then the following three assertions are equivalent.
\begin{enumerate}
\renewcommand{\theenumi}{\roman{enumi}}
\renewcommand{\labelenumi}{\textup{(\theenumi)}}
\item  
Their one-sided sofic shifts  
$(X_{\Lambda_1},\sigma_{\Lambda_1})$
and
$(X_{\Lambda_2},\sigma_{\Lambda_2})$
are continuously orbit equivalent.
\item
The shifts of finite type
$(X_{\widehat{A}_1}, \sigma_{\widehat{A}_1})$ 
and
$(X_{\widehat{A}_2}, \sigma_{\widehat{A}_1})$ 
are continuously orbit equivalent via 
a homeomorphism 
$h_{\widehat{A}}: X_{\widehat{A}_1} \longrightarrow 
X_{\widehat{A}_2}
$
such that there exists a homeomorphism
$h_\Lambda: X_{\Lambda_1}\longrightarrow 
X_{\Lambda_2}
$
satisfying
$\pi_2\circ h_{\widehat{A}} =h_\Lambda\circ\pi_1$.
\item
There exists an isomorphism
$\Phi: \mathcal{O}_{\widehat{A}_1}
\longrightarrow \mathcal{O}_{\widehat{A}_2}
$ of Cuntz--Krieger algebras
such that 
$\Phi(C(X_{\Lambda_1}))=C(X_{\Lambda_2}),
$ 
where $C(X_{\Lambda_i})$
is embedded into $C(X_{\Lambda_i})\subset\mathcal{O}_{\widehat{A}_i}$
through the factor maps
$\pi_i:X_{\widehat{A}_i}\longrightarrow   
X_{\Lambda_i}, i=1,2.
$
\end{enumerate}
\end{proposition}
\begin{proof}
Since the topological dynamical systems
$(X_{{\frak L}_{\Lambda_i}}, 
\sigma_{{\frak L}_{\Lambda_i}})$
are the shifts of finite type
$(X_{\widehat{A}_i}, \sigma_{\widehat{A}_i}), i=1,2,$ 
the assertions are direct from the previous discussions.
\end{proof}
 
We will give an example (cf. \cite[Example 6.15]{BC2019}).  
\begin{example}
\end{example}
Let $\Lambda_0$ and $\Lambda_1$ be 
the even shift over the alphabet $\{0,1\}$ and
the odd shift over the alphabet $\{0,1\}$,
respectively.
Their forbidden words 
${\frak F}_*(\Lambda_0)$
are 
${\frak F}_*(\Lambda_1)$
are defined by
$$
{\frak F}_*(\Lambda_0) = \{ 10^{2n+1}1 \mid n \in \Zp\}, \qquad
{\frak F}_*(\Lambda_1) = \{ 10^{2n}1 \mid n \in \Zp\}
$$
where
$1 0^k 1 = 1\overbrace{0\cdots 0}^{k} 1$
for
$k=2n+1, 2n.$
It is well-known that the subshifts 
$\Lambda_0, \Lambda_1$ are both sofic shifts.
Their left Fischer covers
$G_{\Lambda_0}^F,G_{\Lambda_1}^F$
are shown in Figure \ref{fig:Fischerevenodd}, respectively.
 
%%%%%%%%%%%%%%%%%%%%%%%%%%%%%%%%%%%%%%%%%
\begin{figure}[htbp]
\begin{center}
\input{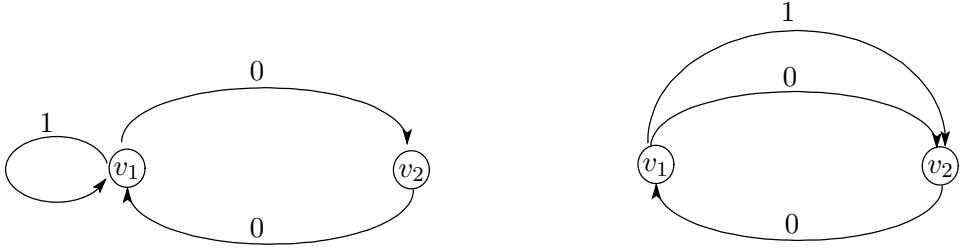}
\end{center}
\caption{Left Fischer covers of $\Lambda_0$ and $\Lambda_1$}
\label{fig:Fischerevenodd}
\end{figure}
%%%%%%%%%%%%%%%%%%%%%%%%

\noindent
We write $\alpha =0, \beta =1$ for the alphabet $\{0,1\}$.
To describe the transition matrices
for the Fischer cover graphs 
$G_{\Lambda_0}^F,G_{\Lambda_1}^F$,
consider the new alphabet sets 
$\widehat{\Sigma}_0, \widehat{\Sigma}_1$
by setting
\begin{equation*}
\widehat{\Sigma}_0 := \{ (\alpha, v_1), \, (\alpha, v_2), \, (\beta, v_1)\},
\qquad 
\widehat{\Sigma}_1 := \{ (\alpha, v_1), \, (\alpha, v_2), \, (\beta, v_2)\},
\end{equation*}
and put
\begin{gather*}
u_1:=(\alpha, v_1), \quad u_2:=(\alpha, v_2), \quad u_3:=(\beta, v_1) \quad \text{ in } \widehat{\Sigma}_0, \\
w_1:=(\alpha, v_1), \quad w_2:=(\alpha, v_2), \quad w_3:=(\beta, v_2) \quad \text{ in } \widehat{\Sigma}_1.
\end{gather*}
We then have the associated transition graphs for 
$G_{\Lambda_0}^F$ and $G_{\Lambda_1}^F$, respectively.
They are shwon in Figure \ref{fig:transitionFischer}.
%%%%%%%%%%%%%%%%%%%%%%%%%%%%%%%%%%%%%%%%%
\begin{figure}[htbp]
\begin{center}
%WinTpicVersion3.08
\unitlength 0.1in
\begin{picture}( 41.9800, 13.7500)( 14.8800,-23.2500)
% STR 2 0 3 0
% 3 3036 1842 3036 1942 0 0
% 
\put(30.3600,-19.4200){\makebox(0,0)[lb]{}}%
% STR 2 0 3 0
% 3 2420 860 2420 960 0 0
% 
\put(24.2000,-9.6000){\makebox(0,0)[lb]{}}%
% STR 2 0 3 0
% 3 2700 2130 2700 2230 2 0
% $u_3$
\put(27.0000,-22.3000){\makebox(0,0)[lb]{$u_3$}}%
% ELLIPSE 2 0 3 0
% 4 3090 2170 2897 2325 2853 2234 2909 2141
% 
\special{pn 8}%
\special{ar 3090 2170 194 156  3.3379255 6.2831853}%
\special{ar 3090 2170 194 156  0.0000000 2.8160494}%
% SARROW 2 0 3 1
% 2 2910 2224 2907 2219
% 
\special{pn 8}%
\special{pa 2910 2224}%
\special{pa 2908 2220}%
\special{fp}%
\special{sh 1}%
\special{pa 2908 2220}%
\special{pa 2924 2286}%
\special{pa 2934 2266}%
\special{pa 2958 2266}%
\special{pa 2908 2220}%
\special{fp}%
% VECTOR 2 0 3 0
% 2 1740 2010 2210 1380
% 
\special{pn 8}%
\special{pa 1740 2010}%
\special{pa 2210 1380}%
\special{fp}%
\special{sh 1}%
\special{pa 2210 1380}%
\special{pa 2154 1422}%
\special{pa 2178 1424}%
\special{pa 2186 1446}%
\special{pa 2210 1380}%
\special{fp}%
% VECTOR 2 0 3 0
% 2 2567 2175 1783 2170
% 
\special{pn 8}%
\special{pa 2568 2176}%
\special{pa 1784 2170}%
\special{fp}%
\special{sh 1}%
\special{pa 1784 2170}%
\special{pa 1850 2190}%
\special{pa 1836 2170}%
\special{pa 1850 2150}%
\special{pa 1784 2170}%
\special{fp}%
% VECTOR 2 0 3 0
% 2 2310 1370 2718 2040
% 
\special{pn 8}%
\special{pa 2310 1370}%
\special{pa 2718 2040}%
\special{fp}%
\special{sh 1}%
\special{pa 2718 2040}%
\special{pa 2700 1974}%
\special{pa 2690 1994}%
\special{pa 2666 1994}%
\special{pa 2718 2040}%
\special{fp}%
% VECTOR 2 0 3 0
% 2 2120 1370 1655 2002
% 
\special{pn 8}%
\special{pa 2120 1370}%
\special{pa 1656 2002}%
\special{fp}%
\special{sh 1}%
\special{pa 1656 2002}%
\special{pa 1712 1960}%
\special{pa 1688 1960}%
\special{pa 1678 1936}%
\special{pa 1656 2002}%
\special{fp}%
% STR 2 0 3 0
% 3 1550 2140 1550 2240 2 0
% $u_2$
\put(15.5000,-22.4000){\makebox(0,0)[lb]{$u_2$}}%
% STR 2 0 3 0
% 3 2164 1200 2164 1300 2 0
% $u_1$
\put(21.6400,-13.0000){\makebox(0,0)[lb]{$u_1$}}%
% STR 2 0 3 0
% 3 5686 1832 5686 1932 0 0
% 
\put(56.8600,-19.3200){\makebox(0,0)[lb]{}}%
% ELLIPSE 2 0 3 0
% 4 5438 2180 5570 2290 5570 2172 5570 2172
% 
\special{pn 8}%
\special{ar 5438 2180 132 110  0.0000000 6.2831853}%
% STR 2 0 3 0
% 3 5070 850 5070 950 0 0
% 
\put(50.7000,-9.5000){\makebox(0,0)[lb]{}}%
% STR 2 0 3 0
% 3 5350 2130 5350 2230 2 0
% $w_3$
\put(53.5000,-22.3000){\makebox(0,0)[lb]{$w_3$}}%
% VECTOR 2 0 3 0
% 2 4390 2000 4860 1370
% 
\special{pn 8}%
\special{pa 4390 2000}%
\special{pa 4860 1370}%
\special{fp}%
\special{sh 1}%
\special{pa 4860 1370}%
\special{pa 4804 1412}%
\special{pa 4828 1414}%
\special{pa 4836 1436}%
\special{pa 4860 1370}%
\special{fp}%
% VECTOR 2 0 3 0
% 2 4770 1360 4305 1992
% 
\special{pn 8}%
\special{pa 4770 1360}%
\special{pa 4306 1992}%
\special{fp}%
\special{sh 1}%
\special{pa 4306 1992}%
\special{pa 4362 1950}%
\special{pa 4338 1950}%
\special{pa 4328 1926}%
\special{pa 4306 1992}%
\special{fp}%
% STR 2 0 3 0
% 3 4210 2120 4210 2220 2 0
% $w_2$
\put(42.1000,-22.2000){\makebox(0,0)[lb]{$w_2$}}%
% STR 2 0 3 0
% 3 4820 1200 4820 1300 2 0
% $w_1$
\put(48.2000,-13.0000){\makebox(0,0)[lb]{$w_1$}}%
% VECTOR 2 0 3 0
% 2 5410 2010 4982 1351
% 
\special{pn 8}%
\special{pa 5410 2010}%
\special{pa 4982 1352}%
\special{fp}%
\special{sh 1}%
\special{pa 4982 1352}%
\special{pa 5002 1418}%
\special{pa 5012 1396}%
\special{pa 5036 1396}%
\special{pa 4982 1352}%
\special{fp}%
% VECTOR 2 0 3 0
% 2 4920 1401 5352 2056
% 
\special{pn 8}%
\special{pa 4920 1402}%
\special{pa 5352 2056}%
\special{fp}%
\special{sh 1}%
\special{pa 5352 2056}%
\special{pa 5332 1990}%
\special{pa 5324 2012}%
\special{pa 5300 2012}%
\special{pa 5352 2056}%
\special{fp}%
% ELLIPSE 2 0 3 0
% 4 4310 2150 4442 2260 4442 2142 4442 2142
% 
\special{pn 8}%
\special{ar 4310 2150 132 110  0.0000000 6.2831853}%
% ELLIPSE 2 0 3 0
% 4 4910 1240 5042 1350 5042 1232 5042 1232
% 
\special{pn 8}%
\special{ar 4910 1240 132 110  0.0000000 6.2831853}%
% ELLIPSE 2 0 3 0
% 4 2240 1230 2372 1340 2372 1222 2372 1222
% 
\special{pn 8}%
\special{ar 2240 1230 132 110  0.0000000 6.2831853}%
% ELLIPSE 2 0 3 0
% 4 1620 2170 1752 2280 1752 2162 1752 2162
% 
\special{pn 8}%
\special{ar 1620 2170 132 110  0.0000000 6.2831853}%
% ELLIPSE 2 0 3 0
% 4 2760 2170 2892 2280 2892 2162 2892 2162
% 
\special{pn 8}%
\special{ar 2760 2170 132 110  0.0000000 6.2831853}%
\end{picture}%
\end{center}
\caption{Transition graphs for the left Fischer covers of $\Lambda_0$ and $\Lambda_1$}
\label{fig:transitionFischer}
\end{figure}
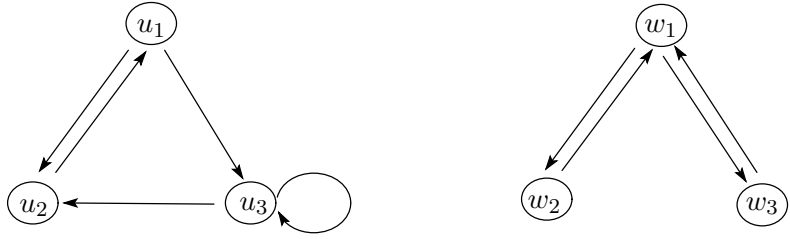
%%%%%%%%%%%%%%%%%%%%%%%%
Their transition matrices are denoted by 
$\widehat{A}_0$ and $\widehat{A}_1$,
respectively. 
They are written 
\begin{equation*}
\widehat{A}_0
=
\begin{bmatrix}
0 & 1& 1 \\
1 & 0& 0 \\
0 & 1& 1 
\end{bmatrix},
\qquad
\widehat{A}_1
=
\begin{bmatrix}
0 & 1& 1 \\
1 & 0& 0 \\
1 & 0& 0 
\end{bmatrix}.
\end{equation*}
Let
$s_1, s_2, s_3$
and $t_1, t_2, t_3$ 
be the generating partial isometries 
of the Cuntz--Krieger algebras
$\mathcal{O}_{\widehat{A}_0}$
and
$\mathcal{O}_{\widehat{A}_1},$
respectively.
They satisfy the following operator relations:
\begin{gather*}
s_1 s_1^* + s_2 s_2^* + s_3 s_3^* =1, \qquad
s_1^* s_1= s_2 s_2^* + s_3 s_3^*, \qquad
s_2^* s_2 = s_1 s_1^*, \qquad
s_3^* s_3= s_2 s_2^* + s_3 s_3^*, \\
t_1 t_1^* + t_2 t_2^* + t_3 t_3^* =1, \qquad
t_1^* t_1= t_2 t_2^* + t_3 t_3^*, \qquad
t_2^* t_2 = t_1 t_1^*, \qquad
t_3^* t_3= t_1 t_1^*.
\end{gather*}
\begin{proposition}
There exists an isomorphism
$\Phi: \mathcal{O}_{\widehat{A}_0}
\longrightarrow \mathcal{O}_{\widehat{A}_1}
$ of Cuntz--Krieger algebras
such that 
$$
\Phi(\D_{\widehat{A}_0})=\D_{\widehat{A}_1},
\qquad
\Phi(C(X_{\Lambda_0}))=C(X_{\Lambda_1}),
$$ 
where $\D_{\widehat{A}_i} = C(X_{\widehat{A}_i}), i=0,1.$
\end{proposition}
\begin{proof}
Put
$s'_1 = t_1, s'_2 = t_2, s'_3 = t_3 t_1$.
They are partial isometries in 
$\mathcal{O}_{\widehat{A}_1}$ satisfying
\begin{equation*}
s'_1 s^{\prime *}_1 + s'_2 s^{\prime *}_2 + s'_3 s^{\prime *}_3 =1, \qquad
s^{\prime *}_1 s'_1= s'_2 s^{\prime *}_2 + s'_3 s^{\prime *}_3, \qquad
s^{\prime *}_2 s'_2 = s'_1 s^{\prime *}_1, \qquad
s^{\prime *}_3 s'_3= s'_2 s^{\prime *}_2 + s'_3 s^{\prime *}_3. 
\end{equation*}
Since 
$t_3 = s'_3 s^{\prime *}_1$,
by putting 
$\Phi(s_i)= s'_i, i=1,2,3$,
$\Phi$ extends an isomorphism from 
$\mathcal{O}_{\widehat{A}_0}$
to $\mathcal{O}_{\widehat{A}_1}$.
For an admissible word
$\mu$ of the shift of finite type 
$(\Lambda_{\widehat{A}_0}, \sigma_{\widehat{A}_0})$ 
defined by the matrix $\widehat{A}_0$,
denote by $\tilde{\mu}$ an admisible word
of $(\Lambda_{\widehat{A}_1}, \sigma_{\widehat{A}_1})$ 
defined by substitutiing
$$
1 \longrightarrow 1, \qquad 
2 \longrightarrow 2, \qquad 
3 \longrightarrow 31. 
$$
It is direct to see that the equality  
$\Phi(s_\mu s_\mu^* ) = t_{\tilde{\mu}}t_{\tilde{\mu}}^*$ 
holds.
Hence we have 
$
\Phi(\D_{\widehat{A}_0})=\D_{\widehat{A}_1}.
$

We will next show that 
$\Phi(C(X_{\Lambda_0}))=C(X_{\Lambda_1}).$
Define the partial isometries
by setting
\begin{equation*}
S_\alpha := s_1 + s_2, \quad
S_\beta := s_3\quad \text{ in } \mathcal{O}_{\widehat{A}_0}
\quad
\text{and }
\quad
T_\alpha := t_1 + t_2, \quad
T_\beta := t_3
\quad \text{ in } \mathcal{O}_{\widehat{A}_1}.
\end{equation*}
It is easy to see that the equalities 
\begin{equation*}
\Phi(S_\alpha) = T_\alpha, 
\quad
\Phi(S_\beta)= T_{\beta \alpha}
\quad
\text{and }
\quad
\Phi^{-1}(T_\alpha)= S_\alpha, \quad
\Phi^{-1}(T_\beta)= S_\beta S_\alpha^*
\end{equation*}
hold.
For $\xi \in B_*(\Lambda_0),$
let
$\bar{\xi}$ be the admissible word of $\Lambda_1$
by substituting 
$$\alpha\longrightarrow \alpha, \qquad \beta \longrightarrow \beta\alpha$$
in $\xi$. 
Then we have
$\Phi(S_\xi S_\xi^*) = T_{\bar{\xi}}T_{\bar{\xi}}^*.$
As
$C(X_{\Lambda_0})$ and $C(X_{\Lambda_1})$ 
are generated by projections
$S_\xi S_\xi^*, \xi \in B_*(\Lambda_0)$
and
$S_\eta S_\eta^*, \eta \in B_*(\Lambda_1),$
respectively,
we know that 
$\Phi(C(X_{\Lambda_0}))=C(X_{\Lambda_1}).$
\end{proof}
\begin{corollary}
The even shift $(X_{\Lambda_0},\sigma_{\Lambda_0})$
and the odd shift $(X_{\Lambda_1},\sigma_{\Lambda_1})$
are continuously orbit equivalent to each other.
\end{corollary}

\begin{remark}
Keep the above notation 
for $S_\alpha, S_\beta$ and $T_\alpha, T_\beta$.
Let us denote by 
$C^*(S_\alpha, S_\beta)$ the $C^*$-subalgebra of 
$\mathcal{O}_{\widehat{A}_0}$
generated by the partial isometries
$S_\alpha, S_\beta$.
It is easy to see that the identities
\begin{equation*}
s_1 = S_\alpha^* S_\beta^* S_\beta S_\alpha S_\alpha, 
\qquad 
s_2 = S_\alpha - S_\alpha^* S_\beta^* S_\beta S_\alpha S_\alpha, \qquad 
s_3 = S_\beta
\end{equation*}
hold, so that the $C^*$-subalgebra
$C^*(S_\alpha, S_\beta)$ coincides with $\mathcal{O}_{\widehat{A}_0}$.
Similarly we know the identities
\begin{equation*}
t_1 = T_\beta^*  T_\beta T_\alpha, 
\qquad 
t_2 = T_\alpha - T_\beta^* T_\beta T_\alpha, \qquad 
t_3 = T_\beta
\end{equation*}
so that 
the $C^*$-subalgebra
$C^*(T_\alpha, T_\beta)$ coincides with $\mathcal{O}_{\widehat{A}_1}$.
\end{remark}

%%%%%%%%%%%%%%%%%%%%%%%%%%%%%%%%%%%%%%%%%%%%%%%%%%%%%%%
%%%%%%%%%%%%%%%%%%%%%%%%%%%%%%%%%%%%%%%%%%%%%%%%%%%%%%%%%%%%%%%%
 \section{One-sided topological conjugacy}
%%%%%%%%%%%%%%%%%%%%%%%%%%%%%%%%%%%%%%%%%%%%%%%%%%%%%%%%%%%%%%%%%%%%%%
%%%%%%%%%%%%%%%%%%%%%%%%%%
In what follows, a sliding bock code means a shift commuting continuous map between 
subshifts, that is always defined by a block map (see \cite{LM}).

In this section, we will prove that 
the triplet 
$(\OLmin, \DLam, \rho^\Lambda)$
for a normal subshift $\Lambda$
is invariant under topological conjugacy of one-sided subshifts,
where $\rho^\Lambda$ denotes the gauge action $\rho^{\LLmin}$ on $\OLmin$
defined in \eqref{eq:gauge}.
  For a left-resolving $\lambda$-graph system ${\frak L}$,
let us denote by
$\Lambda$ the presented subshift.
There exists a natural factor map 
$\pi_{\frak L}: X_{\frak L} \longrightarrow X_\Lambda$ such that 
$\pi_{\frak L}\circ \sigma_{\frak L} = \sigma_\Lambda\circ \pi_{\frak L}$
that is defined in Section 3.
\begin{definition}[{\cite{MaYMJ2010},\cite{MaDynam2020}}]
Let ${\frak L}_1$ and ${\frak L}_2$
be left-resolving $\lambda$-graph systems that present subshifts
$\Lambda_1$ and $\Lambda_2$, respectively.
One-sided subshifts 
$(X_{\Lambda_1},\sigma_{\Lambda_1})$
and
$(X_{\Lambda_2},\sigma_{\Lambda_2})$
are said to be $({\frak L}_1,{\frak L}_2)$-{\it conjugate}\/
if there exist topological conjugacies
$h_\L: (X_{\L_1}, \sigma_{\L_1}) \longrightarrow (X_{\L_2}, \sigma_{\L_2}) $
and
$h_{\Lambda}: (X_{\Lambda_1}, \sigma_{\Lambda_1}) 
\longrightarrow (X_{\Lambda_2}, \sigma_{\Lambda_2}) $
such that 
$\pi_{\L_2} \circ h_\L = h_{\Lambda}\circ \pi_{\L_1}$.
\end{definition}
Equivalently,
there exist homeomorphisms
$h_\L: X_{\L_1} \longrightarrow X_{\L_2}$
and
$h_{\Lambda}: X_{\Lambda_1}
\longrightarrow X_{\Lambda_2}
$
such that 
\begin{equation} \label{eq:hfrak}
\begin{cases}
&h_\L(\sigma_{\L_1}(x))  =\sigma_{\L_2}(h_\L(x)) , \qquad x \in X_{\L_1}, \\ 
&h_\L^{-1}(\sigma_{\L_2}(y)) =\sigma_{\L_1}(h_\L^{-1}(y)), \qquad y \in X_{\L_2},
\end{cases}
\end{equation}
and 
\begin{equation} \label{eq:hpi}
\pi_{\L_2} \circ h_\L = h_{\Lambda}\circ \pi_{\L_1}.
\end{equation}
We remark that the equalities \eqref{eq:hfrak} and \eqref{eq:hpi}
automatically imply the equalities
\begin{equation} \label{eq:hlambda}
\begin{cases}
&h_\Lambda(\sigma_{\Lambda_1}(a)) 
 =\sigma_{\Lambda_2}(h_\Lambda(a)) , \qquad a \in X_{\Lambda_1}, \\ 
&h_\Lambda^{-1}(\sigma_{\Lambda_2}(b))
 =\sigma_{\Lambda_1}(h_\Lambda^{-1}(b)), \qquad b \in X_{\Lambda_2}.
\end{cases}
\end{equation}
We note that if
one-sided subshifts 
$(X_{\Lambda_1},\sigma_{\Lambda_1})$
and
$(X_{\Lambda_2},\sigma_{\Lambda_2})$
are $({\frak L}_1,{\frak L}_2)$-conjugate,
they are $({\frak L}_1,{\frak L}_2)$-eventually conjugate
in the sense of \cite{MaDynam2020}
or in the sense of the following section.
\begin{lemma}
Let $\LLamonemin$ and $\LLamtwomin$
be the minimal $\lambda$-graph systems for normal subshifts 
$\Lambda_1$ and $\Lambda_2$, respectively.
Assume that there exists a topological conjugacy
$h:X_{\Lambda_1}\longrightarrow X_{\Lambda_2}.$
Then there exists $L \in \N$ such that 
for any $l \in \N$ and a word $\mu \in S_l(\Lambda_1)$,
there exists a word $\tilde{\mu} \in S_l(\Lambda_2)$ with 
$|\tilde{\mu}| =|\mu | + L$
such that 
\begin{enumerate}
\renewcommand{\theenumi}{\roman{enumi}}
\renewcommand{\labelenumi}{\textup{(\theenumi)}}
\item 
for $\eta \in \Gamma^-_l(\tilde{\mu})$ and $y \in \Gamma^+_\infty(\tilde{\mu})$,
the equality $h^{-1}(\eta \tilde{\mu} y)_{[l+1, l+|\mu|]} = \mu$ holds,
\item there exists $\gamma \in \Gamma^+_{2L}(\mu)$ such that 
for $\xi \in \Gamma^-_l(\mu), x\in \Gamma^+_\infty(\xi\mu\gamma),$
the equality

\noindent
 $h(\xi \mu \gamma x)_{[l+1, l+|\mu|+L]} = \tilde{\mu}$ holds.
\end{enumerate}
\end{lemma}
\begin{proof}
Since $h:X_{\Lambda_1}\longrightarrow X_{\Lambda_2}$
is a topological conjugacy, there exist $L\in \N$ and
 block maps
$$
\varphi:B_{L+1}(\Lambda_1) \longrightarrow \Sigma_2, \qquad
\phi:B_{L+1}(\Lambda_2) \longrightarrow \Sigma_1,
$$
such that 
$h = \varphi^{[0,L]}_\infty : X_{\Lambda_1}\longrightarrow X_{\Lambda_2}$
and
$h^{-1} = \phi^{[0,L]}_\infty : X_{\Lambda_2}\longrightarrow X_{\Lambda_1}$
where 
$\varphi^{[0,L]}_\infty((x_n)_{n\in \N}) =(\varphi(x_n,\dots,x_{L+n}))_{n\in \N}$
and
$\phi^{[0,L]}_\infty$ is similarly defined (see \cite{LM}).
Let $\mu \in S_l(\Lambda_1)$
with $\mu =(\mu_1,\dots,\mu_m)$.
Since 
$h^{-1} : X_{\Lambda_2}\longrightarrow X_{\Lambda_1}$
is a sliding block code,
there exists 
$\tilde{\mu} =(\tilde{\mu}_1, \dots,\tilde{\mu}_{|\mu|+L}) \in B_{|\mu|+L}(\Lambda_2)$
such that 
$$
\mu_n = \phi(\tilde{\mu}_n, \dots,\tilde{\mu}_{n+L}), \qquad n=1,2,\dots, m.
$$
Suppose that  $y, y' \in \Gamma^+_\infty(\tilde{\mu})$ and
$\eta \in \Gamma^-_l(\tilde{\mu} y)$.
Hence we have
$h^{-1}(\eta\tilde{\mu}y) \in X_{\Lambda_1}$
such that 
$h^{-1}(\eta\tilde{\mu}y)_{[l+1,l+|\mu|]} =\mu.$
Take $\xi \in B_l(\Lambda_1)$ such that 
$h^{-1}(\eta \tilde{\mu} y) = \xi h^{-1}(\tilde{\mu} y)$.
Since 
$h^{-1}(\tilde{\mu} y) = \mu z$ for some $z \in \Gamma^+_\infty(\mu)$,
we have
$h^{-1}(\eta \tilde{\mu} y) = \xi  \mu z$,
so that we have $\xi \in \Gamma^-_l(\mu z)$.

Let $h^{-1}(\tilde{\mu} y') = \mu z'$ for some $z' \in \Gamma^+_\infty(\mu).$
As $\mu \in S_l(\Lambda_1)$, 
the condition  $\xi \in \Gamma^-_l(\mu z)$ implies
 $\xi \in \Gamma^-_l(\mu z')$, so that
$\xi \mu z' \in X_{\Lambda_1}$.
As $h(\xi \mu z) = \eta\tilde{\mu}y$,
we see 
$h(\xi \mu )_{[1,l]} = \eta$.
Now $\tilde{\mu}y' = h(\mu z')$ so that we have
$$
h(\xi \mu z')=h(\xi \mu )_{[1,l]}h(\mu z') = \eta \tilde{\mu}y'.
$$ 
Hence we have
$\eta \in \Gamma^-_l(\tilde{\mu}y')$.
This implies that 
$\Gamma^-_l(\tilde{\mu}y)= \Gamma^-_l(\tilde{\mu}y')
$
so that we conclude that 
$\tilde{\mu} \in S_l(\Lambda_2)$.
One may find $\gamma = (\gamma_1,\dots,\gamma_{2L}) \in \Gamma^+_{2L}(\mu)$
such that 
$h(\mu\gamma x)_{[1,|\mu|+L]} =\tilde{\mu}$
for any $x \in \Gamma_\infty^+(\mu \gamma)$.
Hence 
 $h(\xi \mu \gamma x)_{[l+1, l+|\mu|+L]} = \tilde{\mu}$ holds
for $\xi \in \Gamma^-_l(\mu), x\in \Gamma^+_\infty(\xi\mu\gamma).$
\end{proof}
\begin{lemma}
 For $\mu\in S_l(\Lambda_1),$ 
let $\tilde{\mu}\in S_l(\Lambda_2)$ be as above. 
For $\gamma^\prime \in \Gamma^+_{2L}(\mu)$, 
put 
$\tilde{\mu}^\prime:= h(\xi \mu \gamma' x')_{[l+1, |\mu|+L]} \in S_l(\Lambda_2)$
for some $\xi \in \Gamma^-_l(\mu), x'\in \Gamma^+_\infty(\xi\mu\gamma').$
Then
$\tilde{\mu}\underset{l}{\sim}\tilde{\mu}^\prime$ in $S_l(\Lambda_2)$.
 Hence the equivalence class of $\tilde{\mu}$ 
does not depend on the choice 
of $\gamma$ and $x$ as long as $\xi \mu \gamma x \in B_*(\Lambda_1)$. 
\end{lemma}
\begin{proof}
For $\eta \in \Gamma^-_l(\tilde{\mu})$,
take $y \in \Gamma^+_\infty(\tilde{\mu})$ such that 
$\eta \tilde{\mu}y \in X_{\Lambda_2}$.
Hence
$h^{-1}(\eta \tilde{\mu} y) = \xi  \mu z \in X_{\Lambda_1}$.
As
 $\gamma^\prime \in \Gamma^+_{2L}(\mu)$
and hence $\xi \mu \gamma' \in B_*(\Lambda_1)$,
we have
$\xi \mu \gamma' x' \in X_{\Lambda_1}$
for any $x' \in \Gamma_\infty^+(\xi \mu \gamma')$.
We then have
$$
h(\xi\mu\gamma'x')_{[1,|\mu|+L]}
= \eta h(\xi\mu\gamma'x')_{[l+1,|\mu|+L]}
=\eta \tilde{\mu}^\prime
$$
so that 
$\eta \in \Gamma^-_l(\tilde{\mu}^\prime)$
and
hence
$\Gamma^-_l(\tilde{\mu}) \subset \Gamma^-_l(\tilde{\mu}^\prime)$.
Similarly we have
$\Gamma^-_l(\tilde{\mu}^\prime) \subset \Gamma^-_l(\tilde{\mu})$
so that 
$\Gamma^-_l(\tilde{\mu}) =\Gamma^-_l(\tilde{\mu}^\prime).$
\end{proof}
\begin{lemma} \label{lem:4.4}
Suppose $\nu \in B_*(\Lambda_1)$ with $|\nu |\ge L$ and
$\nu \gamma \in S_l(\Lambda_1)$ and 
$\nu \delta \in S_{l+1}(\Lambda_1)$
for some $\gamma, \delta \in \Gamma_*^+(\nu)$
such that 
$\nu \gamma \underset{l}{\sim} \nu \delta.$
Then we have
$\widetilde{\nu\gamma} \in S_l(\Lambda_2)$ and 
$\widetilde{\nu\delta} \in S_{l+1}(\Lambda_2)$
such that 
$\widetilde{\nu\gamma} \underset{l}{\sim} \widetilde{\nu\delta}.$
\end{lemma}
\begin{proof}
By the previous lemma,
we know that 
$\widetilde{\nu\gamma} \in S_l(\Lambda_2)$ and 
$\widetilde{\nu\delta} \in S_{l+1}(\Lambda_2)$.
It suffices to show that 
$\widetilde{\nu\gamma} \underset{l}{\sim} \widetilde{\nu\delta}.$
For $\eta \in \Gamma^-_l(\widetilde{\nu\gamma})$,
we have $\eta\widetilde{\nu\gamma} y \in X_{\Lambda_2}$ for some
$y \in X_{\Lambda_2}$.
Hence we have
$h^{-1}(\eta\widetilde{\nu\gamma} y) = \xi \nu \gamma z $
for some $\xi \in \Gamma^-_l(\nu \gamma), z \in \Gamma^+_l(\nu \gamma)$.
As 
$\nu \gamma \underset{l}{\sim} \nu \delta$
and hence 
$\xi \nu \delta \in B_*(\Lambda_1)$,
we see
$\xi \nu \delta z'  \in X_{\Lambda_1}$ for some $z' \in X_{\Lambda_1}$.
Since
$\eta\widetilde{\nu\gamma} y =h( \xi \nu \gamma z), $
we have
 $$
h(\xi \nu \delta z') 
= \eta h(\nu \delta z')_{[l+1,\infty)}
= \eta\widetilde{\nu\delta}y'
\quad\text{ for some } y' \in \Gamma^+_\infty(\widetilde{\nu\delta}).
$$
Hence we have 
$\eta\in\Gamma^-_l(\widetilde{\nu\delta})$
so that  
$\Gamma^-_l(\widetilde{\nu\gamma}) \subset \Gamma^-_l(\widetilde{\nu\delta}).
$
Similarly we have
$\Gamma^-_l(\widetilde{\nu\delta}) \subset \Gamma^-_l(\widetilde{\nu\gamma})
$
so that 
$\Gamma^-_l(\widetilde{\nu\gamma}) = \Gamma^-_l(\widetilde{\nu\delta}).
$
\end{proof}
\begin{proposition}\label{prop:conjeventconj}
Let $\LLamonemin$ and $\LLamtwomin$
be the minimal $\lambda$-graph systems for normal subshifts 
$\Lambda_1$ and $\Lambda_2$, respectively.
Assume that the one-sided subshifts
$(X_{\Lambda_1},\sigma_{\Lambda_1})$
and
$(X_{\Lambda_2},\sigma_{\Lambda_2})$
are topologically conjugate.
Then they are
 $(\LLamonemin, \LLamtwomin)$-conjugate.
\end{proposition}
\begin{proof}
Let
$h:X_{\Lambda_1}\longrightarrow X_{\Lambda_2}$
be a topological conjugacy.
Keep the notation as in the previous lemmas.
For $(\alpha_i, u_i)_{i\in \N} \in X_{\LLamonemin}$
where
$u_i =(u_i^l)_{l\in \Zp} \in \Omega_{\LLamonemin}, i \in \N$.
Put
$(\tilde{\alpha}_i)_{i\in \N}:= h((\alpha_i)_{i\in \N}) \in X_{\Lambda_2}$.
Fix $i\in \N$ and $l\in \N$.
Take $\gamma \in B_*(\Lambda_1)$ such that $u_{i+L}^{l+L}$ launches $\gamma$,
and
$\delta\in B_*(\Lambda_1)$ such that $u_{i+L}^{l+L+1}$ launches $\delta$.
We see the following picture:
\begin{equation*}
\begin{CD}
u_i^l @>{\alpha_{i+1}}>>u_{i+1}^{l+1} @>{\alpha_{i+2}}>>
u_{i+2}^{l+2} @>>> \cdots @>{\alpha_{i+L}}>> u_{i+L}^{l+L} @>{\gamma}>> \\
@A{\iota}AA  @A{\iota}AA  @A{\iota}AA @. @A{\iota}AA \\
u_i^{l+1} @>{\alpha_{i+1}}>>u_{i+1}^{l+2} @>{\alpha_{i+2}}>>
u_{i+2}^{l+3} @>>> \cdots @>{\alpha_{i+L}}>> u_{i+L}^{l+L+1} @>{\delta}>> 
 \end{CD}
\end{equation*}
Put $\nu = (\alpha_{i+1},\alpha_{i+2}, \dots, \alpha_{i+L}) \in B_L(\Lambda_1)$
with $|\nu | = L$.
Hence we see that 
$u_i^l = [\nu \gamma]_l$ the $l$-past equivalence class of 
$\nu \gamma \in S_l(\Lambda_1)$, and 
$u_i^{l+1} = [\nu \delta]_{l+1}$ the $l+1$-past equivalence class of 
$\nu \delta \in S_{l+1}(\Lambda_1)$.
By the preceding lemma, we know that   
$\widetilde{\nu\gamma} \in S_l(\Lambda_2)$, 
$\widetilde{\nu\delta} \in S_{l+1}(\Lambda_2)$
and 
$\widetilde{\nu\gamma} \underset{l}{\sim} \widetilde{\nu\delta}.$
Define 
$\tilde{u}_i^l := [\widetilde{\nu\gamma}]_l $
the $l$-past equivalence class of 
$\widetilde{\nu\gamma} \in S_l(\Lambda_2)$, and
$\tilde{u}_i^{l+1} := [\widetilde{\nu\delta}]_{l+1} $
the $l+1$-past equivalence class of 
$\widetilde{\nu\delta} \in S_{l+1}(\Lambda_2)$.
It does not depend on the choice of 
$\gamma$ and $\delta$ by Lemma \ref{lem:4.4}.
 Hence we have vertices
$\tilde{u}_i^l \in V_l^{\Lambda_2^\min}$
and
$\tilde{u}_i^{l+1} \in V_{l+1}^{\Lambda_2^\min}$.
Since
$\widetilde{\nu\gamma} \underset{l}{\sim} \widetilde{\nu\delta},$
we have
$\iota(\tilde{u}_i^{l+1}) =\tilde{u}_i^l$
so that we have an $\iota$-orbit
$$
\tilde{u}_i = (\tilde{u}_i^l)_{l\in \Zp} \in \Omega_\LLamtwomin
\quad \text{ for each } i \in \N.
$$
By its construction, we have for some $x \in X_{\Lambda_1}$
\begin{align*}
h((\alpha_1,\dots,\alpha_i)\nu\gamma x)
%&= h((\alpha_1,\dots,\alpha_i)(\alpha_{i+1},\dots,\alpha_{i+L})\gamma x) \\
&= h((\alpha_1,\dots,\alpha_i)(\alpha_{i+1},\dots,\alpha_{i+L})\gamma x)_{[1,i]}
     h(\nu\gamma x) \\
&= (\tilde{\alpha}_1,\dots,\tilde{\alpha}_i) h(\nu\gamma x)_{[1, |\nu \gamma|]}
     h(x)_{[1,\infty)} \\
&= (\tilde{\alpha}_1,\dots,\tilde{\alpha}_i) \widetilde{\nu\gamma}h(x).
 \end{align*}
Hence we have
$ (\tilde{\alpha}_1,\dots,\tilde{\alpha}_i) \in \Gamma^-_i(\tilde{u}_i^l)$.
It is easy to see that $(\tilde{u}_{i-1}, \alpha_i, \tilde{u}_i) \in E_\LLamtwomin$
so that we have  a sequence
$(\tilde{\alpha}_i, \tilde{u}_i)_{i\in\N} \in X_\LLamtwomin$.
Consequently we get a map
$$
\varphi: (\alpha_i, u_i)_{i\in \N} \in X_{\LLamonemin} 
\longrightarrow
(\tilde{\alpha}_i, \tilde{u}_i)_{i\in\N} \in X_\LLamtwomin
$$
that is continuous by its construction.
Since $h((\alpha_i)_{i\in \N}) = (\tilde{\alpha}_i)_{i\in \N},$
it satisfies $h \circ \pi_{{\frak L}_1} =\pi_{{\frak L}_2}\circ \varphi$.
Similarly we get a map
$$
\phi: (\beta_i, w_i)_{i\in \N} \in X_{\LLamtwomin}
 \longrightarrow
(\tilde{\beta}_i, \tilde{w}_i)_{i\in\N} \in X_\LLamonemin
$$
satisfying
$\varphi \circ \phi = \id_{X_\LLamtwomin}
$ 
and
$\phi \circ \varphi = \id_{X_\LLamonemin}.
$ 
By putting $h_{\frak L} =\varphi$, we have a desired topological conjugacy
$
h_{\frak L}:  X_{\LLamonemin} \longrightarrow X_\LLamtwomin.
$
\end{proof}
\begin{theorem}\label{thm:one-sidedconj4}
Let $\Lambda_1$ and $\Lambda_2$ be normal subshifts.
If their one-sided subshifts 
$(X_{\Lambda_1},\sigma_{\Lambda_1})$ 
and 
$(X_{\Lambda_2},\sigma_{\Lambda_2})$ 
are topologically conjugate, then there exists an isomorphism
$\Phi:\OLamonemin\longrightarrow\OLamtwomin$ of $C^*$-algebras such that 
$\Phi({\mathcal{D}}_{\Lambda_1}) ={\mathcal{D}}_{\Lambda_2}$
and
$\Phi\circ\rho^{\Lambda_1}_t = \rho^{\Lambda_2}_t\circ\Phi, t \in \T$.
\end{theorem}
\begin{proof}
By Proposition \ref{prop:conjeventconj},
$(X_{\Lambda_1},\sigma_{\Lambda_1})$ 
and 
$(X_{\Lambda_2},\sigma_{\Lambda_2})$ 
are
 $(\LLamonemin, \LLamtwomin)$-conjugate,
so that they are 
 $(\LLamonemin, \LLamtwomin)$-eventually conjugate
in the sense of \cite{MaDynam2020} or in the sense of the following section.
By \cite[Theorem 1.3]{MaDynam2020},
we have a desired isomorphism
$\Phi:\OLamonemin\longrightarrow\OLamtwomin$ of $C^*$-algebras.
\end{proof}

\begin{remark}
Brix--Carlsen in \cite{BC2017} gave an example of a pair
$(X_A, \sigma_A)$ and $(X_B, \sigma_B)$
 of irreducible shifts of finite type
such that the converse of Theorem \ref{thm:one-sidedconj4} does not hold.
They found two irreducible matrices $A, B$ with entries in $\{0,1\}$
such that there exists an isomorphism
$\Phi: \OA \longrightarrow \mathcal{O}_B$
of the Cuntz--Krieger algebras such that 
$\Phi(\mathcal{D}_A) = \mathcal{D}_B$
and
$\Phi\circ \rho^A_t = \rho^B_t\circ \Phi$,
but the one-sided topological Markov shifts
$(X_A,\sigma_A)$ and $(X_B,\sigma_B)$ are not topologically conjugate.  
\end{remark}
%%%%%%%%%%%%%%%%%%%%%%%%%%%%%%%%%%%%%%%%%%%%
%%%%%%%%%%%%%%%%%%%%%%%%%%%%%%%%%%%%%%%%%%%%%%%%%%%%%%%%%%%%%%%%%%%%%%
%%%%%%%%%%%%%%%%%%%%%%%%%%%%%%%%%%%%%%%%%%%%%%%%%%%
%%%%%%%%%%%%%%%%%%%%%%%%%%%%%%%%%%%%%%%%%%%%%%%%%%%%%%

%\newpage
  
%%%%%%%%%%%%%%%%%%%%%%%%%%%%%%%%%%%%%%%%%%%%%%%%%%%%%%%%%%%
\section{ One-sided  eventual conjugacy}
%%%%%%%%%%%%%%%%%%%%%%%%%%%%%%%%%%%%%%%%%%%%%%%%%%%
In this section, we will prove that 
a slightly weaker equivalence relation than one-sided topological conjugavy
in one-sided normal subshifts $X_{\Lambda_1}, X_{\Lambda_2}$, called eventual conjugacy,
is equivalent to the condition that 
there exists an isomorphism
$\Phi:\OLamonemin\longrightarrow\OLamtwomin$ of $C^*$-algebras satisfying 
$\Phi({\mathcal{D}}_{\Lambda_1}) ={\mathcal{D}}_{\Lambda_2}$
and
$\Phi\circ\rho^{\Lambda_1}_t = \rho^{\Lambda_2}_t\circ\Phi, t \in \T$.

Let $\Lambda_1$ and $\Lambda_2$ be  subshifts.
Suppose that their one-sided subshifts 
$(X_{\Lambda_1},\sigma_{\Lambda_1})$ 
and 
$(X_{\Lambda_2},\sigma_{\Lambda_2})$ 
eventually conjugate.
This means that there exist
 a homeomorphism
$h: X_{\Lambda_1}\longrightarrow X_{\Lambda_2}$
and an integer  $K \in \Zp$ such that 
\begin{equation} \label{eq:eventconj}
\begin{cases}
&\sigma_{\Lambda_2}^K(h(\sigma_{\Lambda_1}(x))) 
 =\sigma_{\Lambda_2}^{K+1}(h(x)) , \qquad x \in X_{\Lambda_1}, \\ 
&\sigma_{\Lambda_1}^K(h^{-1}(\sigma_{\Lambda_2}(y))
 =\sigma_{\Lambda_1}^{K+1}(h^{-1}(y)), \qquad y \in X_{\Lambda_2}.
\end{cases}
\end{equation}
Let $h_{[1,K]}:X_{\Lambda_1}\longrightarrow B_K({\Lambda_2})$
and
  $h_1:X_{\Lambda_1}\longrightarrow X_{\Lambda_2}$
be continuous maps defined by setting
$$
h_{[1,K]}(x) := h(x)_{[1,K]}, \qquad
h_1(x) := \sigma_{\Lambda_2}^K(h(x)), \qquad
x \in X_{\Lambda_1}.
$$
We then have
\begin{equation*}
h(x) = h_{[1,K]}(x)h_1(x), \qquad
x \in X_{\Lambda_1}.
\end{equation*}
Since 
$h_{[1,K]}:X_{\Lambda_1}\longrightarrow B_K({\Lambda_2})$
is continuous, 
for $\xi_i \in \{\xi_1,\dots,\xi_m\} = B_K(\Lambda_2)$,
$h_{[1,K]}^{-1}(\xi_i)$
is a finite union of cylinder sets, so that
there exists $M_1 \in \Zp$ 
and a block map
$\varphi_1: B_{M_1}({\Lambda_1}) \longrightarrow B_K({\Lambda_2})$
such that
$$
h_{[1,K]}(x) = \varphi_1(x_1,\dots,x_{M_1}) \qquad
\text{ for } x = (x_i)_{i\in \N} \in X_{\Lambda_1}.
$$
Hence we have
\begin{equation*}
h(x) =\varphi_1(x_{[1, M_1]})h_1(x), \qquad
x \in X_{\Lambda_1}.
\end{equation*}
By \eqref{eq:eventconj},
we have the equality
\begin{equation}
h_1(\sigma_{\Lambda_1}(x))) 
=\sigma_{\Lambda_2}(h_1(x)) , \qquad x \in X_{\Lambda_1},
\end{equation}
so that 
  $h_1:X_{\Lambda_1}\longrightarrow X_{\Lambda_2}$
is a sliding block code (cf. \cite{LM}).

Similarly there exists $M_2 \in \Zp$ and a block map 
$\varphi_2: B_{M_2}({\Lambda_2}) \longrightarrow B_K({\Lambda_1})$
and a continuous map
$h_2: X_{\Lambda_2} \longrightarrow X_{\Lambda_1}$
such that 
$$
h^{-1}(y) = \varphi_2(y_{[1,M_2]}) h_2(y), \qquad
h_2(\sigma_{\Lambda_2}(y))) 
=\sigma_{\Lambda_1}(h_2(y))\quad \text{ for } y = (y_i)_{i\in \N} \in X_{\Lambda_2}.
$$
We may assume that $M_1 = M_2$ written $M$ such that 
$M \ge K$.
It then follows that 
\begin{align*}
x = & h^{-1}(h(x)) \\
  = & \varphi_2(h(x)_{[1,M]}) h_2(h(x))\\
  = & \varphi_2(\varphi_1(x_{[1,M]}) h_1(x)_{[1,M-K]}) h_2(\varphi_1(x_{[1,M]})h_1(x)).
\end{align*}
This implies
\begin{align}
x_{[1,K]} & = \varphi_2(\varphi_1(x_{[1,M]}) h_1(x)_{[1,M-K]}),\label{eq:x1K} \\
x_{[K+1,\infty)}& = h_2(\varphi_1(x_{[1,M]})h_1(x)), \label{eq:xK1}
\end{align}
and hence
\begin{equation}
x_{[2K+1,\infty)} 
=\sigma_{\Lambda_1}^K(x_{[K+1,\infty)})
=h_2(\sigma_{\Lambda_2}^K(\varphi_1(x_{[1,M]})h_1(x)))
= h_2(h_1(x)), \qquad x \in X_{\Lambda_1}.
\end{equation}
Similarly we have 
\begin{align*}
y_{[1,K]} & = \varphi_1(\varphi_2(y_{[1,M]}) h_2(y)_{[1,M-K]}), \\
y_{[K+1,\infty)}& = h_1(\varphi_2(y_{[1,M]})h_2(y)), \\
\end{align*}
and 
\begin{equation}
y_{[2K+1,\infty)} = h_1(h_2(y)), \qquad y \in X_{\Lambda_2}.
\end{equation}

%%%%%%%%%%%%%%%%%%%%%%%%%%%%%%%%%%%%%%%%%%%%%%%%%%%%%%%%%%%%%%%
%%%%%%%%%%%%%%%%%%%%%%%%%%%%%%%%%%%%%%%%%%%%%%%%%%%%%%%%%%%%%%%%%

%The map
%$\sigma_{{\Lambda_2^{\varphi_1}}\times {\Lambda_2}}:
%X_{{\Lambda_2^{\varphi_1}}\times {\Lambda_2}} \longrightarrow
%X_{{\Lambda_2^{\varphi_1}}\times {\Lambda_2}} 
%$ is then defined such as 
%\begin{align*}
%& ((\varphi_1(x_{[1,M]}), h_1(x)_1),
%(\varphi_1(x_{[2,M+1]}), h_1(x)_2),
%(\varphi_1(x_{[3,M+2]}), h_1(x)_3),\dots ) \\
%&\longrightarrow \\
%& ((\varphi_1(x_{[2,M+1]}), h_1(x)_2),
%(\varphi_1(x_{[3,M+2]}), h_1(x)_3),
%(\varphi_1(x_{[4,M+3]}), h_1(x)_4),\dots )
%\end{align*}
%%%%%%%%%%%%%%%%%%%%%%%%%%%%%%

For $\xi=(\xi(1),\dots,\xi(K)) \in B_K(\Lambda_2)$
and $y =(y_n)_{n\in \N} \in X_{\Lambda_2}$ with
$y \in \Gamma_\infty^+(\xi)$, we write
$(\xi, y) := (\xi(1),\dots,\xi(K), y_1,y_2,\dots ) \in X_{
\Lambda_2}$. 
Now suppose that 
 $(\xi,y) \in X_{\Lambda_2}$
such that  
$\xi = \varphi_1(x_{[1,M]}), y = h_1(x)$ for some $x=(x_n)_{n\in \N} \in X_{\Lambda_1}$.
Define
$$
\tau(\xi, y) := (\varphi_1(x_{[2,M+1]}), h(\sigma_{\Lambda_1}(x))).
$$
Under the identification
$(\xi, y) = (\varphi_1(x_{[1,M]}), h_1(x)) = h(x), 
$ 
we have
\begin{equation*}
\tau(h(x)) = h(\sigma_{\Lambda_1}(x)) \quad \text{ for } x \in X_{\Lambda_1}.
\end{equation*}
Hence we have
a continuous surjection
$\tau: X_{\Lambda_2}\longrightarrow X_{\Lambda_2}$
such that 
$$
\tau = h \circ \sigma_{\Lambda_1}\circ h^{-1}.
$$
By the relations
\eqref{eq:x1K}, \eqref{eq:xK1},
we know that 
$$
x_{[2,M+1]} 
= x_{[2,K]} x_{[K+1,M+1]}
= \varphi_2(\xi y_{[1, M-K]})_{[2,K]} h_2(\xi y)_{[1, M-K+1]}
$$
so that 
\begin{equation*}
\varphi_1(x_{[2,M+1]} ) = \varphi_1(\varphi_2(\xi y_{[1, M-K]})_{[2,K]} h_2(\xi y)_{[1, M-K+1]})
\end{equation*}
and
\begin{equation} \label{eq:tauxiy}
\tau(\xi,y)
 = (\varphi_1(\varphi_2(\xi y_{[1, M-K]})_{[2,K]} h_2(\xi y)_{[1, M-K+1]}),
 \sigma_{\Lambda_2}(y)) \in B_K({\Lambda_2}) \times X_{\Lambda_2},  
\end{equation}
for
$(\xi,y) \in X_{\Lambda_2}.$
As 
$h_1: X_{\Lambda_1}\longrightarrow X_{\Lambda_2}$
and
$h_2: X_{\Lambda_2}\longrightarrow X_{\Lambda_1}$
are both sliding block codes,
one may take integers $N_1, N_2 \in \N$
and block maps 
$\phi_1: B_{N_1}(\Lambda_1) \longrightarrow \Sigma_2$,
$\phi_2: B_{N_2}(\Lambda_2) \longrightarrow \Sigma_1$
such that 
$$
h_1(x) = \phi_1(x_{[i,N_1+i]})_{i\in \N}
\quad \text{for } x \in X_{\Lambda_1}
\quad
\text{ and }
\quad
h_2(y) = \phi_2(y_{[i,N_2+i]})_{i\in \N}
\quad \text{ for } y \in X_{\Lambda_2}.
$$
We may assume that $N_1, N_2 \ge K$.
Hence we have
$
h_2(\xi y)_{[1,M-K+1]} = \phi_2(\xi y_{[1,M-2K+1+N_2]}).
$ 
We put
$L = M - 2K + 1 + N_2$ and
\begin{equation} \label{eq:tauvarphi1}
\tau^{\varphi_1}(\xi,y)
 = \varphi_1(\varphi_2(\xi y_{[1, M-K]})_{[2,K]} \phi_2(\xi y_{[1,L]}))
\end{equation}
so that 
\begin{equation} \label{eq:tauvpxiy}
\tau(\xi,y)
=(\tau^{\varphi_1}(\xi,y),
 \sigma_{\Lambda_2}(y)) \in B_K({\Lambda_2}) \times X_{\Lambda_2}, \qquad (\xi,y) \in X_{\Lambda_2}. 
\end{equation}
As $N_2 \ge K$, we note that $L \ge M-K$.
Hence the word
$\tau^{\varphi_1}(\xi,y) \in B_K(\Lambda_2)$
is determined by  only $\xi \in B_K(\Lambda_2)$
and $y_{[1,L]}$, so that we may write
$$
\tau^{\varphi_1}(\xi,y) =\tau^{\varphi_1}(\xi,y_{[1,L]}).
$$
%We define a sequence of words in $B_K(\Lambda_2)$ by setting
%\begin{align*}
%\xi_1 & = \xi, \\
%\xi_2 & = \tau^{\varphi_1}(\xi_1,y_{[1,L]}), \\
%\xi_3 & = \tau^{\varphi_1}(\xi_2,y_{[2,L+1]}), \\
%\cdots & \cdots \\
%\xi_{n+1} & = \tau^{\varphi_1}(\xi_n,y_{[n,L+n-1]}).
%\end{align*}
Let $X_{\Lambda_2^{[L]}}$ be the right one-sided subshift of 
 the $L$th higher block shift $\Lambda_2^{[L]}$ of $\Lambda_2$ ( see \cite{LM}).
Define sliding block codes: 
\begin{align*}
g_1: & x \in X_{\Lambda_1} \longrightarrow (\varphi_1(x_{[n,M +n-1]}))_{n \in \N} \in B_K(\Lambda_2)^\N, \\ 
h_1^L: & x \in X_{\Lambda_1} \longrightarrow (h_1(x)_{[n,L + n-1]})_{n \in \N} \in X_{\Lambda_2^{[L]}}
\end{align*}
and put 
\begin{align*}
g_1(x)_n =& \varphi_1(x_{[n,M+n-1]}) \in B_K(\Lambda_2), \\
h_1^L(x)_n =& h_1(x)_{[n,L+n-1]} \in B_L(\Lambda_2)
\end{align*}
so that 
$g_1(x) = (g_1(x)_n)_{n\in \N}, h_1^L(x) = (h_1^L(x)_n)_{n\in \N}.
$
Since $h_1\circ \sigma_{\Lambda_1} = \sigma_{\Lambda_2} \circ h_1,$
we have 
${h}_1^L\circ \sigma_{\Lambda_1} 
= \sigma_{{\Lambda_2}^{[L]}} \circ {h}_1^L.$
Define
$\hat{h}^L: X_{\Lambda_1} \longrightarrow (B_K(\Lambda_2)\times B_L(\Lambda_2))^\N$
by seting
\begin{equation*}
\hat{h}^L(x) =(g_1(x), h_1^L(x)) =(\varphi_1(x_{[n,M +n-1]}), h_1(x)_{[n,L + n-1]})_{n \in \N}. 
\end{equation*}
\begin{lemma}
Define 
\begin{equation*}
X_{\Lambda_2^{\varphi_1}\times\Lambda_2^{[L]}}
=\{ (g_1(x), h_1^L(x)) \in (B_K(\Lambda_2)\times B_L(\Lambda_2))^\N
\mid x \in X_{\Lambda_1} \},
\end{equation*}
and the map
$\sigma_{\Lambda_2^{\varphi_1}\times\Lambda_2^{[L]}}:
X_{\Lambda_2^{\varphi_1}\times\Lambda_2^{[L]}}
\longrightarrow
X_{\Lambda_2^{\varphi_1}\times\Lambda_2^{[L]}}
$ by setting
\begin{equation*}
\sigma_{\Lambda_2^{\varphi_1}\times\Lambda_2^{[L]}}
((g_1(x)_n, h_1^L(x)_n)_{n\in \N}) = (g_1(x)_{n+1}, h_1^L(x)_{n+1})_{n\in \N},
\end{equation*}
Then 
$(X_{\Lambda_2^{\varphi_1}\times\Lambda_2^{[L]}},
\sigma_{\Lambda_2^{\varphi_1}\times\Lambda_2^{[L]}}
)$
is a subshift over $B_K(\Lambda_2)\times B_L(\Lambda_2)$
that is conjugate to $(X_{\Lambda_1},\sigma_{\Lambda_1})$ via
\begin{equation*}
\hat{h}^L: x \in X_{\Lambda_1} \longrightarrow 
(g_1(x)_n, h_1^L(x)_n)_{n\in \N} \in X_{\Lambda_2^{\varphi_1}\times\Lambda_2^{[L]}}.
\end{equation*}
\end{lemma}
\begin{proof}
Since
$\hat{h}^L: X_{\Lambda_1} \longrightarrow X_{\Lambda_2^{\varphi_1}\times\Lambda_2^{[L]}}$
is a sliding block code, 
the pair
$(X_{\Lambda_2^{\varphi_1}\times\Lambda_2^{[L]}},
\sigma_{\Lambda_2^{\varphi_1}\times\Lambda_2^{[L]}}
)$
gives rise to a subshift over $B_K(\Lambda_2)\times B_L(\Lambda_2)$.
As
$\hat{h}_1^L\circ \sigma_{\Lambda_1} 
= \sigma_{\Lambda_2^\varphi \times{\Lambda_2}^{[L]}} \circ \hat{h}_1^L,$
it remains to show that $\hat{h}^L$ is injective.
Suppose tha $\hat{h}^L(x) = \hat{h}^L(z)$ for some 
$x = (x_n)_{n\in \N}, z = (z_n)_{n\in \N}$.
Hence we have
$$
\varphi_1(x_{[1,M_1]}) = \varphi_1(z_{[1,M_1]}),
\qquad
h_1^L(x) = h_1^L(z)
\quad\text{ and hence }
\quad
h_1(x) = h_1(z)
$$
so that $h(x) = h(z)$ proving $x=z$.
\end{proof}

Define 
$\Sigma'_2 = \{ (\xi, y_{[1,L]}) \in  B_K(\Lambda_2)\times B_L(\Lambda_2)
\mid \xi \in \Gamma^-_K(y_{[1,L]}) \}$
and a subshift $(\Lambda'_2,\sigma_{\Lambda'_2})$
over $\Sigma_2'$
by its right one-sided subshift
\begin{align*}
X_{\Lambda'_2} = &\{
(\xi_n, y_{[n,L+n-1]})_{n\in \N} \in (B_K(\Lambda_2)\times B_L(\Lambda_2))^{\N}
\mid \\
&\xi_{n+1} = \tau^{\varphi_1}(\xi_n, y_{[n, L + n-1]}), n \in \N, 
(\xi_1, (y_n)_{n\in \N}) \in X_{\Lambda_2} \} 
\end{align*}
and $\sigma_{\Lambda'_2} : X_{\Lambda'_2} \longrightarrow X_{\Lambda'_2}$
by
\begin{equation*}
\sigma_{\Lambda'_2}(\xi_n, y_{[n,L+n-1]})_{n\in \N}) 
=(\xi_{n+1}, y_{[n+1,L+n]})_{n\in \N}.
\end{equation*}
We then have
\begin{lemma}\label{lem:LamptwoLamone}
$(X_{\Lambda_2^{\varphi_1}\times\Lambda_2^{[L]}},
\sigma_{\Lambda_2^{\varphi_1}\times\Lambda_2^{[L]}}
)
= (X_{\Lambda'_2}, \sigma_{\Lambda'_2})
$
so that 
$ (X_{\Lambda'_2}, \sigma_{\Lambda'_2})$ is topologically conjugate to
$(X_{\Lambda_1}, \sigma_{\Lambda_1})$.
Hence the subshift 
$ (\Lambda'_2, \sigma_{\Lambda'_2})$ is normal if 
$(\Lambda_1,\sigma_1)$ is normal.
\end{lemma}
\begin{proof}
For $(\xi_n, y_{[n,L+n-1]})_{n\in \N} \in X_{\Lambda'_2}$,
we see $(\xi_1, y_1, y_2,\dots ) \in X_{\Lambda_2}$.
Put
$x = h^{-1}(\xi_1, y_1, y_2, \dots ) \in X_{\Lambda_1}$
so that  we know  
$\xi_n = g_1(x)_n$ and $y_{[n,L+n-1]} = h_1^L(x)_n$
for all $n \in \N$.
Hence we may identify
$(g_1(x), h_1^L(x))$ with
$(\xi_n, y_{[n,L+n-1]})_{n\in \N}$.
The identification
between 
$(g_1(x), h_1^L(x))$ and
$(\xi_n, y_{[n,L+n-1]})_{n\in \N}$
yields the identification between 
the subshifts
$(X_{\Lambda_2^{\varphi_1}\times\Lambda_2^{[L]}},
\sigma_{\Lambda_2^{\varphi_1}\times\Lambda_2^{[L]}})
$
and
$(X_{\Lambda'_2}, \sigma_{\Lambda'_2})$.
This implies that
$ (X_{\Lambda'_2}, \sigma_{\Lambda'_2})$ is topologically conjugate to
$(X_{\Lambda_1}, \sigma_{\Lambda_1})$.
\end{proof}

\medskip

In what follows, we assume that  the subshifts
$(\Lambda_1, \sigma_{\Lambda_1})$ and 
$(\Lambda_2, \sigma_{\Lambda_2})$
are both normal.
 Since the subshift
$ (X_{\Lambda'_2}, \sigma_{\Lambda'_2})$ is topologically conjugate to
$(X_{\Lambda_1}, \sigma_{\Lambda_1})$
as one-sided subshifts,
Theorem \ref{thm:one-sidedconj4}
ensures us that 
there exists an isomorphism
$\Phi_1:\OLamonemin\longrightarrow
{\mathcal{O}}_{{\Lambda'_2}^{\operatorname{min}}}
$ of $C^*$-algebras such that 
$\Phi_1({\mathcal{D}}_{\Lambda_1}) ={\mathcal{D}}_{\Lambda'_2}$
and
$\Phi_1\circ\rho^{\Lambda_1}_t = \rho^{\Lambda'_2}_t\circ\Phi_1, t \in \T$.
%We note that $(\Lambda_1, \sigma_{\Lambda_1})$
%is normal so is 
%$ (\Lambda'_2, \sigma_{\Lambda'_2})$.
We will henceforth prove that 
there exists an isomorphism
$\Phi_2:
{\mathcal{O}}_{{\Lambda'_2}^{\operatorname{min}}}
\longrightarrow
\OLamtwomin
$ 
of $C^*$-algebras such that 
$\Phi_2({\mathcal{D}}_{\Lambda'_2}) ={\mathcal{D}}_{\Lambda_2}$
and
$\Phi_2\circ\rho^{\Lambda'_2}_t = \rho^{\Lambda_2}_t\circ\Phi_2, t \in \T$.

Let $(V', E', \lambda',\iota')$ be the minimal $\lambda$-graph system
$\LLamptwomin$ of $\Lambda'_2$. 
The vertex set $V'_l$ is denote by
$\{ v_1^{\prime l}, \dots, v_{m'(l)}^{\prime l} \}.$ 
Since $\LLamptwomin$ is predecessor separated,
the projections of the form
$E_i^{\prime l}$ in the $C^*$-algebra
${\mathcal{O}}_{{\Lambda'_2}^{\operatorname{min}}}$
corresponding to the vertex 
$v_i^{\prime l}  \in V_l^\prime$
of
$\LLamptwomin$
is written in terms of the generating partial isometries
$S'_{(\xi, y_{[1,L]})}, (\xi, y_{[1,L]})\in \Sigma'_2$
by the formula \eqref{eq:eil}.

%Hence the $C^*$-algebra
%${\mathcal{O}}_{{\Lambda'_2}^{\operatorname{min}}}$
%is generated by a family of partial isometries
%$S_{(\xi, y_{[1,L]})}, (\xi, y_{[1,L]})\in \Sigma'_2$
%indexed by elements of the alphabet $\Sigma'_2$.
Let $S_\alpha,\alpha \in \Sigma_2$ be the generating partial isometries 
of the $C^*$-algebra $\OLamtwomin$.
For $(\xi, y) \in B_K(\Lambda_2) \times X_{\Lambda_2}$
with $\xi \in \Gamma_K^-(y)$, 
let us define a sequence $(\xi_n)_{n\in \N}$ of words
of $B_K(\Lambda_2)$ by 
\begin{equation} \label{eq:xiy}
\xi_1 := \xi, \qquad 
\xi_{n+1} = \tau^{\varphi_1}(\xi_n, y_{[n, L+n-1]}), \quad n \in \N.
\end{equation}
For a word $w= (w_1,\dots, w_k) \in B_k(\Lambda_2)$, 
we write 
the partial isometry $S_{w_1}\cdots S_{w_k} \in \OLamtwomin$
as
$S_w$ in $\OLamtwomin$. 
For $(\xi, y_{[1,L]}) \in \Sigma'_2$,
we define a partial isometry
$\widehat{S}_{(\xi, y_{[1,L]})}$ in 
$\OLamtwomin$ by setting
\begin{equation*}
\widehat{S}_{(\xi, y_{[1,L]})} := S_{\xi_1 y_{[1,L]}} S_{\xi_2 y_{[2,L]}}^* \in \OLamtwomin 
\quad \text{ where }
\xi_1 = \xi, \, \xi_2 = \tau^{\varphi_1}(\xi, y_{[1,L]}).
\end{equation*}
We also write 
for $\mu =(\mu_1,\dots, \mu_m) \in B_m(\Lambda_2')$
$$
\widehat{S}_\mu := \widehat{S}_{\mu_1}\cdots \widehat{S}_{\mu_m} \in \OLamtwomin.
$$
We write
$\LLamtwomin =
(V^{\Lambda_2^\min}, E^{\Lambda_2^\min}, \lambda^{\Lambda_2^\min}, \iota^{\Lambda_2^\min}). $
The transition matrix system of $\LLamtwomin $
is denoted by 
$(A_{l,l+1}^\min, I_{l,l+1}^\min)_{l\in \Zp}.$ 
For $w =(w_1,\dots, w_l) \in B_l(\Lambda_2)$,
w define for $v_j^l \in V_l^{\Lambda_2^\min}$ by 
\begin{equation*}
 A_{0,l}^\min(0, w, j) 
=
{\begin{cases}
1 & \text{ there exists } \gamma \in E_{0,l}^{\Lambda_2^\min} ;
      \lambda(\gamma) =w, t(\gamma) = v_j^l, \\
 0 & \text{ otherwise,}
\end{cases}
}
\end{equation*}
where the top vertex $V_0^{\Lambda_2^\min} =\{ v_0\}$ a singleton.
We note the following lemma. 
\begin{lemma}\label{lem:key1}
For $(\xi, y) \in B_K(\Lambda_2) \times X_{\Lambda_2}$
with $\xi \in \Gamma_K^-(y)$, 
let $(\xi_n)_{n\in \N}$  be the sequence 
of $B_K(\Lambda_2)$ defined by \eqref{eq:xiy}.
We then have 
$S_{\xi_1 y_{[1,L+1]}}^*S_{\xi_1 y_{[1,L+1]}}
\le
 S_{\xi_2 y_{[2,L+1]}}^*
S_{\xi_2 y_{[2,L+1]}}$
and hence
$S_{\xi_1 y_{[1,L+1]}} S_{\xi_2 y_{[2,L+1]}}^*
S_{\xi_2 y_{[2,L+1]}} = S_{\xi_1 y_{[1,L+1]}}.
$
More generally we have  
\begin{gather}
S_{\xi_{n} y_{[n,L+n]}}^*S_{\xi_{n} y_{[n,L+n]}}\le
 S_{\xi_{n+1} y_{[n+1,L+n]}}^*
S_{\xi_{n+1} y_{[n+1,L+n]}} \quad \text{ and } \label{eq:InSxinyn} \\
S_{\xi_{n} y_{[n,L+n]}} S_{\xi_{n+1} y_{[n+1,L+n]}}^*
S_{\xi_{n+1} y_{[n+1,L+n]}} 
=S_{\xi_{n} y_{[n,L+n]}}, \qquad n \in \N. \label{eq:Sxinyn}
\end{gather}
\end{lemma}
\begin{proof}
For $z =(z_n)_{n\in \N} \in X_{\Lambda_2}$ with 
$z \in \Gamma^+_\infty(\xi_1 y_{[1,L+1]})$,
we put
$x = h^{-1}(\xi_1 y_{[1,L+1]}z) \in X_{\Lambda_1}.$
Let
$y'_n = y_n$ for $ n=1,2,\dots, L+1$
and $y'_{L + n+1} = z_n$ for $n \in \N$,
and hence 
$(y'_n)_{n\in \N} =y_{[1,L+1]}z\in X_{\Lambda_2}$.
Put
$\xi'_1 =\xi_1$ and 
$\xi'_{n+1} =\tau^\varphi(\xi'_n, y'_{[n, L+n-1]}), n \in \N.$
Hence 
\begin{equation*}
\sigma_{\Lambda_2^{\varphi_1}\times\Lambda_2^{[L]}}(\hat{h}^L(x))
%& =((\varphi_1(x_{[2,M+1]}), h_1(x)_2),(\varphi_1(x_{[3,M+2]}), h_1(x)_3), \dots ) \\
 =
((\xi'_2, y'_{[2,L+1]}),(\xi'_3, y'_{[3,L+2]}), \dots ).
\end{equation*}
Since $\xi'_2 y'_{[2,\infty)} \in X_{\Lambda_2}$
and $y'_{[L+2,\infty)} = z$,
we have 
$\xi'_2 y'_{[2,L+1]} z \in X_{\Lambda_2}.$
As
$\xi'_2 = \tau^\varphi(\xi'_1, y'_{[1, L]}) =\tau^\varphi(\xi_1, y_{[1, L]}) =\xi_2$
and
$y'_{[1,L+1]} = y_{[1,L+1]}$,
we know 
$\xi_2 y_{[2,L+1]} z \in X_{\Lambda_2}.$
%We know $\xi_2 = \tau^\varphi(\xi_1,y_{[1,L]})$ and
%As and $z \in \Gamma^+_\infty(\xi_1 y_{[1,L+1]})$,
%we have $\xi_2 y_{[2,L+1]} z \in X_{\Lambda_2}$
%We then have
%\begin{equation}
%\hat{h}^L(x)=
%((\varphi_1(x_{[1,M]}), h_1(x)_{[1,L]}), (\varphi_1(x_{[2,M+1]}), h_1(x)_{[2,L+1]}),\dots ). 
%\end{equation}
%As
%\begin{gather*}
%\xi_1 = \varphi_1(x_{[1,M]}), \xi_2 = \varphi_1(x_{[2,M+1]}), \dots \\
%y_{[1,L+1]} = h_1(x)_{[1,L+1]}, y_{[2,L+1]} = h_1(x)_{[2,L+1]}, \dots \\
%z = h_1(x)_{[L+2, \infty)},
%\end{gather*}
%and
Hence  we see
\begin{equation*}
 \Gamma^+_\infty(\xi_1 y_{[1,L+1]}) 
\subset  \Gamma^+_\infty(\xi_2 y_{[2,L+1]}) \qquad \text{ in } X_{\Lambda_2}
\end{equation*}
and hence 
\begin{equation}
 \Gamma^+_*(\xi_1 y_{[1,L+1]}) \subset  \Gamma^+_*(\xi_2 y_{[2,L+1]})  \qquad \text{ in } B_*({\Lambda_2}).  \label{eq:Gammaxiy2}
\end{equation}
Consider the $\lambda$-graph system
$\LLamtwomin$.
Suppose that
$A_{0, K+L+1}^{\min}(0, \xi_1 y_{[1,L+1]}, j) =1$
for some $j=1,2,\dots,m(K+L+1)$.
Since $\LLamtwomin$ is the $\lambda$-synchronizing $\lambda$-graph system,
the vertex $v_j^{K+L+1}$ in $V_{K+L+1}^{\Lambda_2^\min}$
 is written as
$v_j^{K+L+1} =[w]_{K+L+1}$ for some   
$w \in S_{K+L+1}(\Lambda_2)$.
Hence $w \in \Gamma_*^+(\xi_1 y_{[1, L+1]})$.
By \eqref{eq:Gammaxiy2},
$ w\in \Gamma^+_*(\xi_2 y_{[2,L+1]})$
so that 
$ \xi_2 y_{[2,L+1]} w \in S_1(\Lambda_2)$
and there exists a labeled edge labeled  
$\xi_2 y_{[2,L+1]}$ from 
the top vertex $v_0$ to 
$[w]_{K+L} \in V_{K+L}^{\Lambda_2^\min}$
in $\LLamtwomin$.
Since
$[w]_{K+L}= \iota^{\min}([w]_{K+L+1})$,
by putting
 $v_{j'}^{K+L} =[w]_{K+L},$
we have
$$
A_{0, K+L}^{\min}(0, \xi_2 y_{[2,L+1]}, j') =1, \qquad
I_{K+L, K+L+1}(j', j) =1
$$
so that 
$$
E_j^{K+L+1} \le E_{j'}^{K+L} \qquad \text{ in } \OLamtwomin.
$$
As
\begin{align*}
S_{\xi_1 y_{[1,L+1]}}^* S_{\xi_1 y_{[1,L+1]}}
= & \sum_{j=1}^{m(K+L+1)} A_{0, K+L+1}^{\min}(0, \xi_1 y_{[1,L+1]}, j) E_j^{K+L+1}, \\
S_{\xi_2 y_{[2,L+1]}}^* S_{\xi_2 y_{[2,L+1]}}
= & \sum_{j'=1}^{m(K+L)} A_{0, K+L}^{\min}(0, \xi_2 y_{[2,L+1]}, j') E_{j'}^{K+L},
\end{align*}
we have
\begin{equation*}
S_{\xi_1 y_{[1,L+1]}}^* S_{\xi_1 y_{[1,L+1]}}
\le 
S_{\xi_2 y_{[2,L+1]}}^* S_{\xi_2 y_{[2,L+1]}}
\end{equation*}
so that
\begin{equation*}
S_{\xi_1 y_{[1,L+1]}} S_{\xi_2 y_{[2,L+1]}}^*
S_{\xi_2 y_{[2,L+1]}} = S_{\xi_1 y_{[1,L+1]}}.
\end{equation*}
Similarly we have
the inequality \eqref{eq:InSxinyn}
and  the equality \eqref{eq:Sxinyn}.
\end{proof}
By using the above lemma, we see that the following lemma holds.
\begin{lemma}\label{lem:5.6}
For $(\xi, y) \in B_K(\Lambda_2) \times X_{\Lambda_2}$
with $\xi \in \Gamma_K^-(y)$, 
let $(\xi_n)_{n\in \N}$  be the sequence 
of $B_K(\Lambda_2)$ defined by \eqref{eq:xiy}.
We then have
\begin{equation*}
\widehat{S}_{(\xi_1, y_{[1,L]})} \widehat{S}_{(\xi_2, y_{[2,L+1]})} 
\cdots \widehat{S}_{(\xi_n, y_{[n,L+n-1]})} 
=S_{\xi_1 y_{[1,L+n-1]}}
S_{\xi_{n+1} y_{[n+1,L+n-1]}}^*, \qquad n\in \N.
\end{equation*}
\end{lemma}
\begin{proof}
The following equalities hold:
\begin{align*}
& \widehat{S}_{(\xi_1, y_{[1,L]})} \widehat{S}_{(\xi_2, y_{[2,L+1]})} 
   \cdots \widehat{S}_{(\xi_n, y_{[n,L+n-1]})} \\
=&  S_{\xi_1 y_{[1,L]}} S_{\xi_2 y_{[2,L]}}^*
    S_{\xi_2 y_{[2,L+1]}} S_{\xi_3 y_{[3,L+1]}}^* 
\cdots \widehat{S}_{(\xi_n, y_{[n,L+n-1]})}      \\
=& S_{\xi_1 y_{[1,L]}} S_{\xi_2 y_{[2,L]}}^*
    S_{\xi_2 y_{[2,L]}}S_{y_{L+1}}S_{y_{L+1}}^*
    S_{y_{L+1}}
    S_{\xi_3 y_{[3,L+1]}}^* 
\cdots \widehat{S}_{(\xi_n, y_{[n,L+n-1]})}      \\
=&  S_{\xi_1 y_{[1,L]}} S_{y_{L+1}}S_{y_{L+1}}^*
S_{\xi_2 y_{[2,L]}}^* 
    S_{\xi_2 y_{[2,L]}}    S_{y_{L+1}}
S_{\xi_3 y_{[3,L+1]}}^* 
\cdots \widehat{S}_{(\xi_n, y_{[n,L+n-1]})}      \\
=&  S_{\xi_1 y_{[1,L+1]}} 
S_{\xi_2 y_{[2,L+1]}}^* 
    S_{\xi_2 y_{[2,L+1]}}
    S_{\xi_3 y_{[3,L+1]}}^* 
\cdots \widehat{S}_{(\xi_n, y_{[n,L+n-1]})}.
\end{align*}
By Lemma \ref{lem:key1}, 
we see that 
$$
S_{\xi_1 y_{[1,L+1]}} 
S_{\xi_2 y_{[2,L+1]}}^* 
    S_{\xi_2 y_{[2,L+1]}}
= S_{\xi_1 y_{[1,L+1]}}.
$$
We thus have
\begin{equation*}
\widehat{S}_{(\xi_1, y_{[1,L]})} \widehat{S}_{(\xi_2, y_{[2,L+1]})} 
   \cdots \widehat{S}_{(\xi_n, y_{[n,L+n-1]})} 
=  S_{\xi_1 y_{[1,L+1]}} 
    S_{\xi_3 y_{[3,L+1]}}^* 
\cdots \widehat{S}_{(\xi_n, y_{[n,L+n-1]})},      
\end{equation*}
so that  inductively we have the desired equality.
\end{proof}
The following lemma is direct by using Lemma \ref{lem:5.6}.
\begin{lemma} \label{lem:5.7}
For $(\xi, y) \in B_K(\Lambda_2) \times X_{\Lambda_2}$
with $\xi \in \Gamma_K^-(y)$, 
let $(\xi_n)_{n\in \N}$  be the sequence 
of $B_K(\Lambda_2)$ defined by \eqref{eq:xiy}.
\begin{enumerate}
\renewcommand{\theenumi}{\roman{enumi}}
\renewcommand{\labelenumi}{\textup{(\theenumi)}}
\item
\begin{align*}
& (\widehat{S}_{(\xi_1, y_{[1,L]})} \widehat{S}_{(\xi_2, y_{[2,L+1]})} 
\cdots \widehat{S}_{(\xi_n, y_{[n,L+n-1]})})^*
(\widehat{S}_{(\xi_1, y_{[1,L]})} \widehat{S}_{(\xi_2, y_{[2,L+1]})} 
\cdots \widehat{S}_{(\xi_n, y_{[n,L+n-1]})}) \\
 =&
S_{\xi_{n+1} y_{[n+1,L+n-1]}}
S_{\xi_1 y_{[1,L+n-1]}}^*
S_{\xi_1 y_{[1,L+n-1]}}
S_{\xi_{n+1} y_{[n+1,L+n-1]}}^*, \qquad n\in \N.
\end{align*}
\item
\begin{align*}
& (\widehat{S}_{(\xi_1, y_{[1,L]})} \widehat{S}_{(\xi_2, y_{[2,L+1]})} 
\cdots \widehat{S}_{(\xi_n, y_{[n,L+n-1]})})
(\widehat{S}_{(\xi_1, y_{[1,L]})} \widehat{S}_{(\xi_2, y_{[2,L+1]})} 
\cdots \widehat{S}_{(\xi_n, y_{[n,L+n-1]})})^* \\
 =&
S_{\xi_1 y_{[1,L+n-1]}}
S_{\xi_1 y_{[1,L+n-1]}}^*, \qquad n\in \N.
\end{align*}
\item
\begin{equation*}
 \sum_{(\xi_1, y_{[1,L]}) \in \Sigma'_2}
\widehat{S}_{(\xi_1, y_{[1,L]})}
\widehat{S}_{(\xi_1, y_{[1,L]})}^* \\
 =  1.
\end{equation*}
\end{enumerate}
\end{lemma}
%\begin{proof}
%(i) Direct computations

%(ii)  Direct computations.

%(iii)  Direct computations.
%\end{proof}
The following lemma is direct from the above lemma.
\begin{lemma}
For $\nu, \mu \in B_*(\Lambda'_2)$,
we have
\begin{equation*}
 \widehat{S}_\nu^* \widehat{S}_\nu \widehat{S}_\mu \widehat{S}_\mu^*
= \widehat{S}_\mu \widehat{S}_\mu^* \widehat{S}_\nu^* \widehat{S}_\nu.
\end{equation*}
\end{lemma}\label{lem:commute}
\begin{proof}
Let $\nu =((\xi_1,y_{[1,L]}), \dots, (\xi_n,y_{[n, L+n-1]})),
\mu =((\eta_1,w_{[1,L]}), \dots, (\eta_m,w_{[m, L+m-1]}))$.
Since
\begin{equation*}
 \widehat{S}_\nu^* \widehat{S}_\nu 
= 
S_{\xi_{n+1} y_{[n+1,L+n-1]}}
S_{\xi_1 y_{[1,L+n-1]}}^*
S_{\xi_1 y_{[1,L+n-1]}}
S_{\xi_{n+1} y_{[n+1,L+n-1]}}^* 
\end{equation*}
and
\begin{equation*}
\widehat{S}_\mu \widehat{S}_\mu^*
=
S_{\eta_1 w_{[1,L+m-1]}}
S_{\eta_1 w_{[1,L+m-1]}}^*,
\end{equation*} 
both $\widehat{S}_\nu^* \widehat{S}_\nu$
and
$\widehat{S}_\mu \widehat{S}_\mu^*$
are contained in
${\mathcal{D}}_{{\frak L}_{\Lambda_2}^{\operatorname{min}}}$ 
that is a  commutative $C^*$-algebra.
\end{proof}

Let
$\OhatLamtwomin$
be the $C^*$-subalgebra of $\OLamtwomin$ generated by 
the partial isometries
$$
\widehat{S}_{(\xi_1, y_{[1,L]})},
\qquad   (\xi_1, y_{[1,L]})\in \Sigma'_2.
$$
The $C^*$-subalgebra
$\DhatLLamtwomin$ of $\OhatLamtwomin$
is defined by the $C^*$-algebra  generated by elements of the form:
$$
\widehat{S}_\mu \widehat{S}_\nu^* \widehat{S}_\nu \widehat{S}_\mu^*,
\qquad  \mu, \nu \in B_*(\Lambda'_2)
$$
and
the $C^*$-subalgebra
$\DhatLamtwo$ of $\OhatLamtwomin$
is defined by the $C^*$-algebra  generated by elements of the form:
$$
\widehat{S}_\mu \widehat{S}_\mu^*,
\qquad \mu \in B_*(\Lambda'_2).
$$
\begin{lemma}
Let $\nu, \mu \in B_*(\Lambda'_2)$ be 
$\nu =((\xi_1,y_{[1,L]}), \dots, (\xi_n,y_{[n, L+n-1]}))$
and

\noindent
$\mu =((\eta_1,w_{[1,L]}), \dots, (\eta_m,w_{[m, L+m-1]}))$.
We  then  have
\begin{equation}\label{eq:hatsmusnu}
 \widehat{S}_\mu \widehat{S}_\nu^* \widehat{S}_\nu \widehat{S}_\mu^* 
=
 {\begin{cases}
& S_{\xi_1 y_{[1,L+n-1]}} S_{\eta_1 w_{[1,L+m-1]}}^* 
S_{\eta_1 w_{[1,L+m-1]}} S_{\xi_1 y_{[1,L+n-1]}}^* \\
& \text{if } \xi_{n+1} = \eta_{m+1},  y_{[n+1,L+n-1]} =w_{[m+1,L+m-1]}, \\
& 0 \quad \text{ otherwise, }
\end{cases}}
 \end{equation}
where 
$\xi_{n+1} = \tau^{\varphi_1}(\xi_n, y_{[n,L+n-1]}),
 \eta_{m+1} = \tau^{\varphi_1}(\eta_m, w_{[m,L+m-1]}).
$
\end{lemma}
\begin{proof}
We have
\begin{align*}
 \widehat{S}_\mu \widehat{S}_\nu^* \widehat{S}_\nu \widehat{S}_\mu^* 
=&
S_{\xi_1 y_{[1,L+n-1]}} S_{\xi_{n+1} y_{[n+1,L+n-1]}}^*\cdot 
S_{\eta_{m+1} w_{[m+1,L+m-1]}} S_{\eta_{1} w_{[1,L+m-1]}}^*  \\
&\cdot 
S_{\eta_{1} w_{[1,L+m-1]}} S_{\eta_{m+1} w_{[m+1,L+m-1]}} ^*\cdot
S_{\xi_{n+1} y_{[n+1,L+n-1]}} S_{\xi_1 y_{[1,L+n-1]}} ^*. 
\end{align*}
Similarly to \eqref{eq:Gammaxiy2}, we know  
$\Gamma^+_*(\xi_1 y_{[1,L+n-1]}) \subset
\Gamma^+_*(\xi_{n+1} y_{[n+1,L+n-1]}),
$ 
so that 
the equality
$$
S_{\xi_1 y_{[1,L+n-1]}} S_{\xi_{n+1} y_{[n+1,L+n-1]}}^*\cdot 
S_{\eta_{m+1} w_{[m+1,L+m-1]}}
=S_{\xi_1 y_{[1,L+n-1]}}
$$ holds
 if and only if 
$\xi_{n+1} y_{[n+1,L+n-1]} =\eta_{m+1} w_{[m+1,L+m-1]}$,
otherwise
$$
S_{\xi_1 y_{[1,L+n-1]}} S_{\xi_{n+1} y_{[n+1,L+n-1]}}^*\cdot 
S_{\eta_{m+1} w_{[m+1,L+m-1]}}
=0.
$$
Hence we have the equality \eqref{eq:hatsmusnu}.
%\begin{align*}
%& \widehat{S}_\mu \widehat{S}_\nu^* \widehat{S}_\nu \widehat{S}_\mu^* \\
%=
%& {\begin{cases}
%S_{\xi_1 y_{[1,L+n-1]}} S_{\eta_1 w_{[1,L+m-1]}}^* 
%S_{\eta_1 w_{[1,L+m-1]}} S_{\xi_1 y_{[1,L+n-1]}}^*
%& \text{if } \xi_{n+1} = \eta_{m+1},  y_{[n+1,L+n-1]} =w_{[m+1,L+m-1]}, \\
%0 & \text{ otherwise. }
%\end{cases}}
% \end{align*}
\end{proof}
%%%%%%%%%%%%%%%%%%%%%%%%%%%%%%%%%%%%%%%%%%
\begin{lemma}\label{lem:iff}
For 
$\nu, \mu \in B_*(\Lambda'_2)$,
we have
\begin{enumerate}
\renewcommand{\theenumi}{\roman{enumi}}
\renewcommand{\labelenumi}{\textup{(\theenumi)}}
\item
$S_\mu^{\prime *} S_\mu^{\prime} \ge S'_\nu S_\nu^{\prime*}$ in $\OLamptwomin$
if and only if 
$\widehat{S}_\mu^* \widehat{S}_\mu \ge \widehat{S}_\nu \widehat{S}_\nu^*$ 
in $\OhatLamtwomin.$
\item
$1- S_\mu^{\prime*} S'_\mu \ge S'_\nu S_\nu^{\prime*}$ in $\OLamptwomin$
if and only if 
$1- \widehat{S}_\mu^* \widehat{S}_\mu \ge \widehat{S}_\nu \widehat{S}_\nu^*$ in $\OhatLamtwomin.$
\end{enumerate}
\end{lemma}
\begin{proof}
Let $\nu, \mu \in B_*(\Lambda'_2)$ be 
$\nu =((\xi_1,y_{[1,L]}), \dots, (\xi_n,y_{[n, L+n-1]}))$
and

\noindent
$\mu =((\eta_1,w_{[1,L]}), \dots, (\eta_m,w_{[m, L+m-1]}))$.

(i) Assume that 
$S_\mu^{\prime *} S_\mu^{\prime}\ge S'_\nu S_\nu^{\prime*}$.
Since 
$$
\mu \nu =
((\eta_1,w_{[1,L]}),\dots, (\eta_m,w_{[m,L+m-1]}),(\xi_1,y_{[1,L]}),\dots,
(\xi_n,y_{[n, L+n-1]}))
$$ 
is admissible in $\Lambda'_2$,
we see that 
$\xi_1 = \eta_{m+1}, w_{[m+1, L+m-1]} =y_{[1, L-1]}.$ 
Hence we have
\begin{align*}
& \widehat{S}_\mu^* \widehat{S}_\mu \cdot \widehat{S}_\nu \widehat{S}_\nu^*  \\
= & 
S_{\eta_{m+1} w_{[m+1,L+m-1]}} S_{\eta_1 w_{[1,L+m-1]}}^*S_{\eta_1 w_{[1,L+m-1]}} S_{\eta_{m+1} w_{[m+1,L+m-1]}}^*
\cdot 
S_{\xi_1 y_{[1,L+n-1]}} S_{\xi_1 y_{[1,L+n-1]}}^* \\
= & 
S_{\xi_1 y_{[1,L-1]}} S_{\eta_1 w_{[1,L+m-1]}}^*S_{\eta_1 w_{[1,L+m-1]}} S_{\xi_1 y_{[1,L-1]}}^*
\cdot 
S_{\xi_1 y_{[1,L+n-1]}} S_{\xi_1 y_{[1,L+n-1]}}^* \\
= & 
S_{\xi_1 y_{[1,L+n-1]}} S_{\xi_1 y_{[1,L+n-1]}}^* S_{\xi_1 y_{[1,L-1]}} S_{\eta_1 w_{[1,L+m-1]}}^* \\
& \cdot S_{\eta_1 w_{[1,L+m-1]}} S_{\xi_1 y_{[1,L-1]}}^*
S_{\xi_1 y_{[1,L+n-1]}} S_{\xi_1 y_{[1,L+n-1]}}^*. \\
= & 
S_{\xi_1 y_{[1,L+n-1]}} S_{y_{[L,L+n-1]}}^*S_{\xi_1 y_{[1,L-1]}}^* S_{\xi_1 y_{[1,L-1]}} S_{\eta_1 w_{[1,L+m-1]}}^* \\
& \cdot S_{\eta_1 w_{[1,L+m-1]}} S_{\xi_1 y_{[1,L-1]}}^*
S_{\xi_1 y_{[1,L-1]}}S_{y_{[L,L+n-1]}} S_{\xi_1 y_{[1,L+n-1]}}^*.
\end{align*}
Now the equality
\begin{align*}
S_{\eta_1 w_{[1,L+m-1]}} S_{\xi_1 y_{[1,L-1]}}^* S_{\xi_1 y_{[1,L-1]}} 
=
&S_{\eta_1 w_{[1,L+m-1]}} S_{\eta_{m+1} w_{[m+1,L+m-1]}}^* 
                                   S_{\eta_{m+1} w_{[m+1,L+m-1]}} \\
=
&S_{\eta_1 w_{[1,L+m-1]}}
\end{align*}
holds
because the last equality may be shown in a similar way to \eqref{eq:Sxinyn}.
Hence we have 
\begin{align*}
& \widehat{S}_\mu^* \widehat{S}_\mu \cdot \widehat{S}_\nu \widehat{S}_\nu^*  \\
= & 
S_{\xi_1 y_{[1,L+n-1]}} S_{y_{[L,L+n-1]}}^* S_{\eta_1 w_{[1,L+m-1]}}^*  
 \cdot 
S_{\eta_1 w_{[1,L+m-1]}} S_{y_{[L,L+n-1]}} S_{\xi_1 y_{[1,L+n-1]}}^* \\
= & 
S_{\xi_1 y_{[1,L+n-1]}} S_{\eta_1 w_{[1,L+m-1]} y_{[L,L+n-1]}}^*
S_{\eta_1 w_{[1,L+m-1]} y_{[L,L+n-1]}} S_{\xi_1 y_{[1,L+n-1]}}^*. \\
\end{align*}
Since for 
\begin{equation}
\gamma = ((\zeta_1,z_{[1,L]}),(\zeta_2,z_{[2,L+1]}), \dots, (\zeta_k,z_{[k,L+k-1]}))
\in \Gamma_*^+(\nu), \label{eq:gamma}
\end{equation}
with
\begin{equation*}
\nu \gamma =
((\xi_1,y_{[1,L]}),\dots,(\xi_n,y_{[n, L+n-1]}),(\zeta_1,z_{[1,L]}),(\zeta_2,z_{[2,L+1]}), \dots, (\zeta_k,z_{[k,L+k-1]}))
\in  B_*(\Lambda'_2),
\end{equation*}
the condition
$S_\mu^{\prime*} S'_\mu \ge S'_\nu S_\nu^{\prime*}$ 
implies
\begin{align}
\mu\nu \gamma 
=& ((\eta_1,w_{[1,L]}),\dots, (\eta_m,w_{[m,L+m-1]}), (\xi_1,y_{[1,L]}),\dots,(\xi_n,y_{[n, L+n-1]}),\\
& (\zeta_1,z_{[1,L]}),(\zeta_2,z_{[2,L+1]}), \dots, (\zeta_k,z_{[k,L+k-1]}))
\in  B_*(\Lambda'_2). \label{eq:munugamma}
\end{align}
Hence 
in addition to $\xi_1 = \eta_{m+1}, w_{[m+1, L+m-1]} =y_{[1, L-1]}$,
we have the inequality 
\begin{equation}
S_{\xi_{1} y_{[1,L+n-1]}} S_{\eta_1w_{[1, L+m-1]} y_{[L,L+n-1]}}^* 
S_{\eta_1 w_{[1,L+m-1]}y_{[L,L+n-1]}} 
S_{\xi_1 y_{[1,L+n-1]}}^* 
\ge 
S_{\xi_1 y_{[1,L+n-1]}} S_{\xi_1 y_{[1,L+n-1]}}^*,  \label{eq:5.16}
\end{equation}
proving  
$\widehat{S}_\mu^* \widehat{S}_\mu \cdot \widehat{S}_\nu \widehat{S}_\nu^*
\ge\widehat{S}_\nu \widehat{S}_\nu^*
$ and hence
\begin{equation*}
\widehat{S}_\mu^* \widehat{S}_\mu \cdot \widehat{S}_\nu \widehat{S}_\nu^*
=\widehat{S}_\nu \widehat{S}_\nu^*.
\end{equation*}

Conversely suppose that 
the inequality 
$\widehat{S}_\mu^* \widehat{S}_\mu \ge \widehat{S}_\nu \widehat{S}_\nu^*$ 
in $\OhatLamtwomin$
holds.
The inequality is equivalent to the equality
$\widehat{S}_\mu^* \widehat{S}_\mu \cdot \widehat{S}_\nu \widehat{S}_\nu^*
=\widehat{S}_\nu \widehat{S}_\nu^*
$
that is also equivalent to the inequality \eqref{eq:5.16}
because of the preceding equality
\begin{equation}
 \widehat{S}_\mu^* \widehat{S}_\mu \cdot \widehat{S}_\nu \widehat{S}_\nu^*  
= 
S_{\xi_1 y_{[1,L+n-1]}} S_{\eta_1 w_{[1,L+m-1]} y_{[L,L+n-1]}}^*
S_{\eta_1 w_{[1,L+m-1]} y_{[L,L+n-1]}} S_{\xi_1 y_{[1,L+n-1]}}^*.  \label{eq:SmuhatSnuhat}
\end{equation}
For $\gamma \in \Gamma_*^+(\nu)$ as in \eqref{eq:gamma},
the inequality \eqref{eq:5.16} together with \eqref{eq:munugamma} implies 
$\mu \nu \gamma \in B_*(\Lambda'_2)$ and hence
$\nu \gamma \in \Gamma_*^+(\mu)$.
In the identity 
\begin{equation}\label{eq:maruone}
S'_\nu S_\nu^{\prime*} =
\sum_{k=1}^{m'(|\mu| + |\nu|)}
S'_\nu E_k^{\prime |\mu| + |\nu|} S_\nu^{\prime*}
\quad \text{ in } \OLamptwomin,
\end{equation}
take $k =1,2,\dots,m'(|\mu| + |\nu|)$ such that 
$S'_\nu E_k^{\prime |\mu| + |\nu|} S_\nu^{\prime*}\ne 0$.
As $S_\nu^{\prime*} S'_\nu \ge E_k^{\prime |\mu| +|\nu|}$,
we see that 
$$
A'_{|\mu|, |\mu| +|\nu|} (j, \nu, k) =1
\quad
\text{ for some }
j=1,2,\dots,m'(|\mu|)
$$
where
$A'_{|\mu|, |\mu| +|\nu|} (j, \nu, k)$
is a matrix component defined by the transition matrix system
$(A'_{l,l+1}, I'_{l,l+1})_{l\in \Zp}$
of ${\frak L}_{\Lambda'_2}^{\operatorname{min}}.$ 
Since
\begin{equation*}
S'_\nu S_\nu^{\prime*} E_j^{\prime|\mu|}S'_\nu S_\nu^{\prime*} =
\sum_{k=1}^{m'(|\mu| + |\nu|)}
A'_{|\mu|, |\mu| +|\nu|} (j, \nu, k)
S'_\nu E_k^{\prime |\mu| + |\nu|} S_\nu^{\prime*},
\end{equation*}
we have
\begin{equation} \label{eq:marutwo}
E_j^{\prime|\mu|} \ge 
S'_\nu E_k^{\prime |\mu| + |\nu|} S_\nu^{\prime*},
\end{equation}
Take $\delta \in E'_{|\mu|, |\mu|+|\nu|}$
such that 
$\lambda(\delta) = \nu, s(\delta) = v_j^{\prime|\mu|} \in V'_{|\mu|},
t(\delta) = v_k^{\prime|\mu|+|\nu|} \in V'_{|\mu|+|\nu|}
$
in the $\lambda$-graph system
${\frak L}_{\Lambda'_2}^{\operatorname{min}}.$
There exists a word $\gamma \in B_*(\Lambda'_2)$
such that 
$v_k^{\prime|\mu|+|\nu|}$ launches $\gamma$.
Take $\delta' \in E'_{|\mu|+|\nu|,|\mu|+|\nu|+|\gamma|}$
such that 
$\lambda(\delta') = \gamma,
 s(\delta') = v_k^{\prime|\mu|+|\nu|} \in V'_{|\mu|+|\nu|}
$
in 
${\frak L}_{\Lambda'_2}^{\operatorname{min}}.$
The labeled path $\delta\delta'$ is the unique path 
labeled $\nu \gamma$
in $E'_{|\mu|+|\nu|,|\mu|+|\nu|+|\gamma|}$.
Since
$\gamma \in \Gamma_*^+(\nu)$
implies 
$\nu \gamma \in \Gamma_*^+(\mu)$,
we know  that 
$\mu \in \Gamma_{|\mu|}^-(v_j^{\prime|\mu|}).$
Hence there exists a labeled path
$\delta^{\prime\prime} \in E'_{0,|\mu|}$
labeled $\mu$ such that 
$t(\delta^{\prime\prime}) =v_j^{\prime|\mu|}.
$
This implies that
\begin{equation}\label{eq:maruthree}
S_\mu^{\prime *} S_\mu^{\prime} \ge E_j^{\prime|\mu|}
\end{equation}
so that by \eqref{eq:marutwo} we see
$S_\mu^{\prime *} S_\mu^{\prime} 
\ge 
S'_\nu E_k^{\prime |\mu| + |\nu|} S_\nu^{\prime*}.
$
By the identity \eqref{eq:maruone}, we conclude the inequality
$S_\mu^{\prime *} S_\mu^{\prime} \ge S'_\nu S_\nu^{\prime*}$ in $\OLamptwomin$.

(ii)
Assume that 
$1- S_\mu^{\prime*} S_\mu^{\prime} \ge S_\nu^\prime S_\nu^{\prime*}$ 
in  $\OLamptwomin$.
Now suppose that 
$\widehat{S}_\mu^* \widehat{S}_\mu \cdot \widehat{S}_\nu \widehat{S}_\nu^*\ne 0$ 
in $\OhatLamtwomin$.
As in the proof of (i),  we have
\begin{align*}
0\ne & \widehat{S}_\mu^* \widehat{S}_\mu \cdot \widehat{S}_\nu \widehat{S}_\nu^*  \\
= & 
S_{\eta_{m+1} w_{[m+1,L+m-1]}} S_{\eta_1 w_{[1,L+m-1]}}^*S_{\eta_1 w_{[1,L+m-1]}} S_{\eta_{m+1} w_{[m+1,L+m-1]}}^*
\cdot 
S_{\xi_1 y_{[1,L+n-1]}} S_{\xi_1 y_{[1,L+n-1]}}^* 
\end{align*}
so that  
$\xi_1 = \eta_{m+1}, w_{[m+1, L+m-1]} =y_{[1, L-1]}.$ 
One may take an $L+K+m+n-1$-synchronizing word
$z_{[1, L+ k-1]} \in B_*(\Lambda_2)$ such that 
$$
\eta_1w_{[1,L+m-1]}y_{[L,L+n-1]}z_{[L-1, L+k-1]} \in B_*(\Lambda_2),
$$ 
where $n =|\nu|, m=|\mu|$.
By putting
\begin{equation*}
\zeta_1  =\tau^{\varphi_1}(\xi_n, y_{[n,L+n-1]}), \quad
\zeta_2  =\tau^{\varphi_1}(\zeta_1, z_{[1,L]}), \quad
\cdots  \quad
\zeta_k  =\tau^{\varphi_1}(\zeta_{k-1}, z_{[k, L+k-2]}),
\end{equation*}
%\begin{align*}
%\zeta_1&  =\tau^{\varphi_1}(\xi_n, y_{[n,L+n-1]}), \\
%\zeta_2&  =\tau^{\varphi_1}(\zeta_1, z_{[1,L]}), \\
%\cdots & \cdots \\
%\zeta_k&  =\tau^{\varphi_1}(\zeta_{k-1}, z_{[k, L+k-2]}),
%\end{align*}
the word 
\begin{equation*}
 ((\eta_1,w_{[1,L]}),\dots, (\eta_m,w_{[m,L+m-1]}), (\xi_1,y_{[1,L]}),\dots,(\xi_n,y_{[n, L+n-1]}),
 (\zeta_1,z_{[1,L]}), \dots, (\zeta_k,z_{[k,L+k-1]}))
\end{equation*}
%\begin{align*}
%& ((\eta_1,w_{[1,L]}),\dots, (\eta_m,w_{[m,L+m-1]}), (\xi_1,y_{[1,L]}),\dots,(\xi_n,y_{[n, L+n-1]}),\\
%& (\zeta_1,z_{[1,L]}),(\zeta_2,z_{[2,L+1]}), \dots, (\zeta_k,z_{[k,L+k-1]}))
%\end{align*}
belongs to
$B_*(\Lambda'_2).$
Put
$\gamma = ((\zeta_1,z_{[1,L]}),(\zeta_2,z_{[2,L+1]}), \dots, (\zeta_k,z_{[k,L+k-1]}))$
so that we have
$ S_\mu^{\prime*} S_\mu^\prime \cdot S_\nu^\prime S_\nu^{\prime*}
\ge S_{\nu\gamma}^\prime S_{\nu\gamma}^{\prime*} \ne 0$ 
in $\OLamptwomin$, 
a contradiction.
%= & 
%S_{\xi_1 y_{[1,L-1]}} S_{\eta_1 w_{[1,L+m-1]}}^*S_{\eta_1 w_{[1,L+m-1]}} S_{\xi_1 y_{[1,L-1]}}^*
%\cdot S_{\xi_1 y_{[1,L+n-1]}} S_{\xi_1 y_{[1,L+n-1]}}^* \\
%= &  S_{\xi_1 y_{[1,L+n-1]}} S_{\xi_1 y_{[1,L+n-1]}}^* S_{\xi_1 y_{[1,L-1]}} S_{\eta_1 w_{[1,L+m-1]}}^* \\
%& \cdot S_{\eta_1 w_{[1,L+m-1]}} S_{\xi_1 y_{[1,L-1]}}^*
%S_{\xi_1 y_{[1,L+n-1]}} S_{\xi_1 y_{[1,L+n-1]}}^*. \\
%\end{align*}j=1,2,\dots,m'(|\mu|)
%%%%%%%%%%%%%%%%%%%%%%%%%%%%%%%

Conversely, assume  
$1- \widehat{S}_\mu^* \widehat{S}_\mu \ge \widehat{S}_\nu \widehat{S}_\nu^*$ in $\OhatLamtwomin.$
Now suppose that
$ S_\mu^{\prime*} S_\mu^\prime \cdot S_\nu^\prime S_\nu^{\prime*} \ne 0$
in $\OLamptwomin$.
Since
\begin{equation*}
S_\mu^{\prime*}S'_\mu  =
\sum_{j=1}^{m'(|\mu|)}
A'_{0, |\mu|} (0, \mu, j)
 E_j^{\prime |\mu|}
\quad \text{ in } \OLamptwomin,
\end{equation*}
take 
$
j=1,2,\dots,m'(|\mu|)
$ such that 
$A'_{0, |\mu|} (0, \mu, j) =1$ 
and
$ E_j^{\prime |\mu|} \cdot S_\nu^\prime S_\nu^{\prime*} \ne 0$
so that 
$S_\nu^{\prime*}
 E_j^{\prime |\mu|} \cdot S_\nu^\prime 
 \ne 0.
 $
As
\begin{equation*}
 S_\nu^{\prime*} E_j^{\prime|\mu|}S'_\nu 
=
\sum_{k=1}^{m'(|\mu| + |\nu|)}
A'_{|\mu|, |\mu| +|\nu|} (j, \nu, k)
 E_k^{\prime |\mu| + |\nu|},
\end{equation*}
there exists $k =1,2,\dots,m'(|\mu| + |\nu|)$
such that 
$A'_{|\mu|, |\mu| +|\nu|} (j, \nu, k) =1$
and hence 
\begin{equation*} 
S_\nu^{\prime*} E_j^{\prime|\mu|} S'_\nu
\ge 
 E_k^{\prime |\mu| + |\nu|}.
\end{equation*}
One may take an admissible word
$\gamma =((\zeta_1,z_{[1,L]}), \dots, (\zeta_p, z_{[p,L+p-1]})) \in B_*(\Lambda'_2)$
such that 
$v_k^{\prime |\mu| + |\nu|}$ launches $\gamma$
so that 
$E_k^{\prime |\mu| + |\nu|} \ge S'_\gamma S_\gamma^{\prime*}$.
Hence we have
\begin{equation*} 
 E_j^{\prime|\mu|} 
\ge 
S'_{\nu\gamma} S_{\nu\gamma}^{\prime*}
\quad \text{ so that}
\quad
S_\mu^{\prime*}S_\mu^{\prime} \ge 
S'_{\nu\gamma} S_{\nu\gamma}^{\prime*}.
\end{equation*}
By (i) we have 
$\widehat{S}_\mu^* \widehat{S}_\mu \ge 
\widehat{S}_{\nu\gamma} \widehat{S}_{\nu\gamma}^{*}.
$
Since
\begin{equation*}
\widehat{S}_\mu^* \widehat{S}_\mu \cdot
\widehat{S}_{\nu} \widehat{S}_{\nu}^{*} \cdot
\widehat{S}_{\nu\gamma} \widehat{S}_{\nu\gamma}^{*}
=
\widehat{S}_{\nu\gamma} \widehat{S}_{\nu\gamma}^{*}
\ne 0,
\end{equation*}
we get
$\widehat{S}_\mu^* \widehat{S}_\mu \cdot
\widehat{S}_{\nu} \widehat{S}_{\nu}^{*} \ne 0,$
a contradiction to
$1- \widehat{S}_\mu^* \widehat{S}_\mu \ge \widehat{S}_\nu \widehat{S}_\nu^*$.
\end{proof}

\begin{lemma}\label{lem:Snualpha}
For $\alpha\in \Sigma_2^\prime$ and
$\nu \in B_n(\Lambda'_2)$
with
$\nu \alpha \in B_{n+1}(\Lambda'_2)$,
we have
$
\widehat{S}_\alpha^* \widehat{S}_\nu^* \widehat{S}_\nu\widehat{S}_\alpha
= \widehat{S}_{\nu\alpha}^* \widehat{S}_{\nu\alpha}.
$ 
\end{lemma}
\begin{proof}
Let $\alpha\in \Sigma_2^\prime$ and
$\nu \in B_n(\Lambda'_2)$ be 
$\alpha =(\eta_1,w_{[1,L]}) \in \Sigma_2^\prime, 
\nu =((\xi_1,y_{[1,L]}), \dots, (\xi_n,y_{[n, L+n-1]})) \in B_n(\Lambda'_2).$
We put
$\eta_2 = \tau^{\varphi_1}(\eta_1,w_{[1,L]}) \in B_K(\Lambda_2)$.
We then have
\begin{align*}
 \widehat{S}_\alpha^* \widehat{S}_\nu^* \widehat{S}_\nu\widehat{S}_\alpha 
%= & (S_{\eta_1w_{[1,L]}}S_{\eta_2 w_{[2,L]}}^*)^*
%     (S_{\xi_1 y_{[1, L + n-1]}} S_{\xi_{n+1} y_{[n+1, L + n-1]}}^*)^*  \cdot \\
%  &  S_{\xi_1 y_{[1, L + n-1]}} S_{\xi_{n+1} y_{[n+1, L + n-1]}}^* 
%     S_{\eta_1w_{[1,L]}}S_{\eta_2 w_{[2,L]}}^* \\
= & S_{\eta_2 w_{[2,L]}}S_{\eta_1w_{[1,L]}}^* S_{\xi_{n+1} y_{[n+1, L + n-1]}}
     S_{\xi_1 y_{[1, L + n-1]}}^*  \cdot \\
  &  S_{\xi_1 y_{[1, L + n-1]}} S_{\xi_{n+1} y_{[n+1, L + n-1]}}^* 
     S_{\eta_1w_{[1,L]}}S_{\eta_2 w_{[2,L]}}^* \\
\end{align*}
As $S_{\eta_1w_{[1,L]}}^* S_{\xi_{n+1} y_{[n+1, L + n-1]}} \ne $
if and only if  $\eta_1 = \xi_{n+1}, w_{[1,L-1]} =  y_{[n+1, L + n-1]}.$
Hence we have
\begin{align*}
 \widehat{S}_\alpha^* \widehat{S}_\nu^* \widehat{S}_\nu\widehat{S}_\alpha 
= & S_{\eta_2 w_{[2,L]}} S_{w_L}^* S_{\xi_{n+1} y_{[n+1, L + n-1]}}^* 
     S_{\xi_{n+1} y_{[n+1, L + n-1]}} S_{\xi_1 y_{[1, L + n-1]}}^*  \cdot \\
  &  S_{\xi_1 y_{[1, L + n-1]}} S_{\xi_{n+1} y_{[n+1, L + n-1]}}^* 
      S_{\xi_{n+1} y_{[n+1, L + n-1]}} S_{w_L}S_{\eta_2 w_{[2,L]}}^* \\
\end{align*}
As $S_{\xi_1 y_{[1, L + n-1]}} S_{\xi_{n+1} y_{[n+1, L + n-1]}}^* 
      S_{\xi_{n+1} y_{[n+1, L + n-1]}} = S_{\xi_1 y_{[1, L + n-1]}}$,
we have
\begin{align*}
 \widehat{S}_\alpha^* \widehat{S}_\nu^* \widehat{S}_\nu\widehat{S}_\alpha 
= & S_{\eta_2 w_{[2,L]}} S_{w_L}^*  S_{\xi_1 y_{[1, L + n-1]}}^*  \cdot 
    S_{\xi_1 y_{[1, L + n-1]}}  S_{w_L}S_{\eta_2 w_{[2,L]}}^* \\
= & S_{\eta_2 w_{[2,L]}}  S_{\xi_1 y_{[1, L + n-1]} w_L}^*  \cdot 
    S_{\xi_1 y_{[1, L + n-1]}w_L}  S_{\eta_2 w_{[2,L]}}^*.
\end{align*}
The last term above is nothing but
$
 \widehat{S}_{\nu\alpha}^*  \widehat{S}_{\nu\alpha}. 
$
\end{proof}

In the minimal $\lambda$-graph system ${\frak L}_{\Lambda'_2}^{\min}$,
recall that 
$\{ v_1^{\prime l}, \dots, v_{m'(l)}^{\prime l} \}$
denote the vertex set 
$V'_l $ of ${\frak L}_{\Lambda'_2}^{\min}$
of the normal subshift $\Lambda'_2$.
For a vertex $v_i^{\prime l} \in V'_l$, define a function
$f_i^l:B_l(\Lambda'_2) \longrightarrow \{0,1\}$ 
by setting for $\mu \in B_l(\Lambda'_2)$
\begin{equation*}
f_i^l(\mu) =
\begin{cases}
1 & \text{ if } \mu \in \Gamma^-_l(v_i^{\prime l}), \\
-1 & \text{ if } \mu \not\in \Gamma^-_l(v_i^{\prime l}).
\end{cases}
\end{equation*}
Recall that  $S'_\alpha, \alpha \in \Sigma_2'$ and
$E_i^{\prime l}, v_i^{\prime l} \in V'_l$
denote the canonical generating partial isometries and projections 
of the $C^*$-algebra $\OLamptwomin$.
We then have by \eqref{eq:eil}
\begin{equation}\label{eq:eiprimel}
E_i^{\prime l} 
= \prod_{\mu \in B_l(\Lambda'_2)} S_\mu^{\prime*} S_\mu^{\prime f_i^l(\mu)} 
\quad \text{ in }
\quad  \OLamptwomin
\end{equation}
where for $\mu \in B_l(\Lambda'_2)$
$$
S_\mu^{\prime*} S_\mu^{\prime f_i^l(\mu)} =
\begin{cases}
S_\mu^{\prime*} S_\mu^{\prime} & \text{ if } f_i^l(\mu) =1, \\
1 - S_\mu^{\prime*} S_\mu^\prime & \text{ if } f_i^l(\mu) =-1.
\end{cases}
$$
In the $C^*$-algebra $\OhatLamtwomin$,
we define a projection for each $v_i^{\prime l}\in V'_l$ 
by settting
\begin{equation}
\widehat{E}_i^{l} : = \prod_{\mu \in B_l(\Lambda'_2)} \widehat{S}_\mu^* \widehat{S}_\mu^{f_i^l(\mu)} 
\quad \text{ in }
\quad  \OhatLamtwomin \label{eq:eilhat}
\end{equation}
where for $\mu \in B_l(\Lambda'_2)$
$$
\widehat{S}_\mu^{*} \widehat{S}_\mu^{f_i^l(\mu)} =
\begin{cases}
\widehat{S}_\mu^{*} \widehat{S}_\mu & \text{ if } f_i^l(\mu) =1, \\
1 - \widehat{S}_\mu^{*} \widehat{S}_\mu & \text{ if } f_i^l(\mu) =-1.
\end{cases}
$$
Let $(A'_{l,l+1}, I'_{l,l+1})_{l\in \Zp}$ 
be the transition matrix system for  ${\frak L}_{\Lambda'_2}^{\min}$.
\begin{lemma}
For each $v_i^{\prime l} \in V_l^\prime$, we have the identities
\begin{equation*}
\sum_{i=1}^{m'(l)} \widehat{E}_i^{l} = 1, \qquad
\widehat{E}_i^{l} = \sum_{j=1}^{m'(l+1)} I'_{l,l+1}(i,j) \widehat{E}_j^{l+1}.
\end{equation*}
\end{lemma}
\begin{proof}
We will first show 
$\widehat{E}_i^l \ge \widehat{E}_j^{l+1}$ for 
$i=1,2,\dots,m'(l), j=1,2,\dots,m'(l+1)$
with $I'_{l,l+1}(i,j) =1$.
Assume that $I'_{l,l+1}(i,j) =1$.
Hence we have
$\Gamma_l^-(v_j^{\prime l+1}) =\Gamma_l^-(v_i^{\prime l}) $
in $\LLamptwomin$.
By Lemma \ref{lem:Snualpha}, for $\nu \in \Gamma_l^-(v_i^{\prime l})$,
we have 
$
\widehat{S}_\nu^* \widehat{S}_\beta^* \widehat{S}_\beta\widehat{S}_\nu
= \widehat{S}_{\beta\nu}^* \widehat{S}_{\beta\nu}
$ 
for 
$\beta \in \Gamma_1^-(\nu).$
Hence 
\begin{equation}
\widehat{S}_\nu^* \widehat{S}_\nu
\ge \widehat{S}_{\beta\nu}^* \widehat{S}_{\beta\nu}
\quad
\text{ for }
\quad
\beta \in \Gamma_1^-(\nu).
\end{equation}
For 
 $\nu \not\in \Gamma_l^-(v_i^{\prime l})$
and
$\beta_1, \beta_2 \in \Gamma_1^-(\nu),$
we have 
\begin{equation*}
(1 -  \widehat{S}_{\beta_1\nu}^* \widehat{S}_{\beta_1\nu})
(1 -  \widehat{S}_{\beta_2\nu}^* \widehat{S}_{\beta_2\nu})
= 1 -  \widehat{S}_{\nu}^*(
\widehat{S}_{\beta_1}^* \widehat{S}_{\beta_1}
+
\widehat{S}_{\beta_2}^* \widehat{S}_{\beta_2}
-
\widehat{S}_{\beta_1}^* \widehat{S}_{\beta_1}
\widehat{S}_{\beta_2}^* \widehat{S}_{\beta_2}
)
 \widehat{S}_{\nu}.
\end{equation*}
Similarly we know 
\begin{equation}
\prod_{\beta\in \Gamma_1^-(\nu)}
(1 -  \widehat{S}_{\beta\nu}^* \widehat{S}_{\beta\nu})
= 1 -  \widehat{S}_{\nu}^*(
\bigvee_{\beta\in \Gamma_1^-(\nu)}
\widehat{S}_{\beta}^* \widehat{S}_{\beta}
)
 \widehat{S}_{\nu} \label{eq:bigvee}
\end{equation}
where 
$$
\bigvee_{\beta\in \Gamma_1^-(\nu)}
\widehat{S}_{\beta}^* \widehat{S}_{\beta}
= \sum_{\beta\in \Gamma_1^-(\nu)}
\widehat{S}_{\beta}^* \widehat{S}_{\beta}
- \sum_{\beta1\ne \beta_2\in \Gamma_1^-(\nu)}
 \widehat{S}_{\beta_1}^* \widehat{S}_{\beta_1}
\widehat{S}_{\beta_2}^* \widehat{S}_{\beta_2}
+ \cdots -
(-1)^{|\Gamma_1^-(\nu)|} \prod_{\beta\in \Gamma_1^-(\nu)}
\widehat{S}_{\beta}^* \widehat{S}_{\beta}
$$
the projection spanned by 
$
\widehat{S}_{\beta}^* \widehat{S}_{\beta},
{\beta\in \Gamma_1^-(\nu)}.
$
Now 
$
\bigvee_{\beta\in \Sigma'_2}
\widehat{S}_{\beta}^* \widehat{S}_{\beta} = 1
$
so that 
%\begin{equation*}
%\widehat{S}_{\nu} \widehat{S}_{\nu}^* (\bigvee_{\beta\in \Sigma'_2}
%\widehat{S}_{\beta}^* \widehat{S}_{\beta}) \widehat{S}_{\nu} \widehat{S}_{\nu}^*
% \ge \widehat{S}_{\nu} \widehat{S}_{\nu}^*.
%\end{equation*}
%Since
%\begin{equation*}
%\widehat{S}_{\nu} \widehat{S}_{\nu}^* (\bigvee_{\beta\in \Sigma'_2}
%\widehat{S}_{\beta}^* \widehat{S}_{\beta})
%\widehat{S}_{\nu} \widehat{S}_{\nu}^*
%=\bigvee_{\beta\in \Sigma'_2}\widehat{S}_{\nu} \widehat{S}_{\nu}^*
%\widehat{S}_{\beta}^* \widehat{S}_{\beta}
%\widehat{S}_{\nu} \widehat{S}_{\nu}^*
%=\bigvee_{\beta\in \Gamma_1^-(\nu)}\widehat{S}_{\beta}^* \widehat{S}_{\beta}
%\widehat{S}_{\nu} \widehat{S}_{\nu}^*\widehat{S}_{\beta}^* \widehat{S}_{\beta}
%\le \bigvee_{\beta\in \Gamma_1^-(\nu)}\widehat{S}_{\beta}^* \widehat{S}_{\beta},
%\end{equation*}
%we have 
%\begin{equation*}
%\bigvee_{\beta\in \Gamma_1^-(\nu)}\widehat{S}_{\beta}^* \widehat{S}_{\beta}
%\ge \widehat{S}_{\nu} \widehat{S}_{\nu}^* \label{eq:veeGamma}
%\end{equation*}
%Hence 
we have
\begin{equation}
\widehat{S}_{\nu}^*
\widehat{S}_{\nu}
=
\widehat{S}_{\nu}^*(
\bigvee_{\beta\in \Sigma'_2}
\widehat{S}_{\beta}^* \widehat{S}_{\beta}
) \widehat{S}_{\nu}
=
\widehat{S}_{\nu}^* ( \bigvee_{\beta\in \Gamma_1^-(\nu)}
\widehat{S}_{\beta}^* \widehat{S}_{\beta})
\widehat{S}_{\nu}. \label{eq:Snuvee}
\end{equation}
By \eqref{eq:bigvee} and \eqref{eq:Snuvee},
we have
\begin{equation}
\prod_{\beta\in \Gamma_1^-(\nu)}
(1 -  \widehat{S}_{\beta\nu}^* \widehat{S}_{\beta\nu})
= 1 -  \widehat{S}_{\nu}^* \widehat{S}_{\nu} \label{eq:betaprod}
\end{equation}
For
$\nu \in \Gamma_l^-(v_i^{\prime l})(=\Gamma_l^-(v_j^{\prime l+1}))$
and $\beta \in \Gamma_1^-(\nu)$
with
 $\mu =\beta \nu \in \Gamma_{l+1}^-(v_j^{\prime l+1})$,
we have 
\begin{equation*}
\widehat{S}_\nu^*\widehat{S}_\nu \ge 
\widehat{S}_\mu^*\widehat{S}_\mu. 
%\label{eq:maru6}
\end{equation*}
For
$\nu \not\in \Gamma_l^-(v_i^{\prime l})(=\Gamma_l^-(v_j^{\prime l+1}))$
and $\beta \in \Gamma_1^-(\nu)$
with
 $\mu =\beta \nu \not\in \Gamma_{l+1}^-(v_j^{\prime l+1})$,
we have by \eqref{eq:betaprod}
\begin{equation*}
\prod_{\beta\in \Gamma_1^-(\nu)}
(1 -  \widehat{S}_{\beta\nu}^* \widehat{S}_{\beta\nu})
= 1 -  \widehat{S}_{\nu}^* \widehat{S}_{\nu}.
\end{equation*}
Hence  we have
\begin{align*}
\widehat{E}_i^l
= & \prod_{\nu \in \Gamma_l^-(v_i^{\prime l})} \widehat{S}_\nu^*\widehat{S}_\nu \cdot
 \prod_{\nu \not\in \Gamma_l^-(v_i^{\prime l})}(1- \widehat{S}_\nu^*\widehat{S}_\nu) \\
\ge &  \prod_{\mu \in \Gamma_{l+1}^-(v_j^{\prime l+1})} \widehat{S}_\mu^*\widehat{S}_\mu \cdot
 \prod_{\mu \not\in \Gamma_{l+1}^-(v_j^{\prime l+1})}(1- \widehat{S}_\mu^*\widehat{S}_\mu) \\
\ge &
\widehat{E}_j^{l+1}.
\end{align*}

We will next see that 
$\widehat{E}_{j}^{l+1} \cdot \widehat{E}_{j'}^{l+1}  =0$
for $j \ne j'$.
As 
$v_j^{\prime l+1} \ne v_{j'}^{\prime l+1}$ in $V'_{l+1},$
we know 
$\Gamma_{l+1}^-(v_j^{\prime l+1}) \ne \Gamma_{l+1}^-(v_{j'}^{\prime l+1})$
because 
$\LLamptwomin$ is predecessor-separated.
Hence we have two cases:

Case (1): There exists $\mu \in \Gamma_{l+1}^-(v_j^{\prime l+1})$ 
such that $\mu \not\in  \Gamma_{l+1}^-(v_{j'}^{\prime l+1})$.

Case (2): There exists $\mu \in \Gamma_{l+1}^-(v_{j'}^{\prime l+1})$ 
such that $\mu \not\in  \Gamma_{l+1}^-(v_{j}^{\prime l+1})$.

In both cases, it is easy to see that 
$\widehat{E}_{j}^{l+1} \cdot \widehat{E}_{j'}^{l+1}  =0$
by its definition \eqref{eq:eilhat}.

Since 
$\widehat{E}_i^l \ge \widehat{E}_j^{l+1}$ for 
$i=1,2,\dots,m'(l), j=1,2,\dots,m'(l+1)$
with $I'_{l,l+1}(i,j) =1$,
and 
$\widehat{E}_{j}^{l+1} \cdot \widehat{E}_{j'}^{l+1}  =0$
for $j \ne j'$,
we have the inequality
\begin{equation}\label{eq:wideinequal}
\widehat{E}_i^l \ge \sum_{j=1}^{m'(l+1)} I'_{l,l+1}(i,j) \widehat{E}_j^{l+1}.
\end{equation}
We will next show that 
$\sum_{i=1}^{m'(l)} \widehat{E}_i^l =1.$
Denote by $\{1,-1\}^{B_l(\Lambda'_2)}$ the set of function
$f:B_l(\Lambda'_2) \longrightarrow \{1,-1\}$.
For $f \in \{1,-1\}^{B_l(\Lambda'_2)}$, we set  
$$
\widehat{S}_\mu^*\widehat{S}_\mu^{f(\mu)} =
\begin{cases}
\widehat{S}_\mu^*\widehat{S}_\mu & \text{ if } f(\mu) =1, \\
1 - \widehat{S}_\mu^*\widehat{S}_\mu  & \text{ if } f(\mu) =-1
\end{cases}
$$
and put
$$
\widehat{E}_f = \prod_{\mu \in B_l(\Lambda'_2)}\widehat{S}_\mu^*\widehat{S}_\mu^{f(\mu)}
\quad \text{ in } \OhatLamtwomin.
$$
For $\mu \in B_l(\Lambda'_2)$, 
the identity
$1= \widehat{S}_\mu^*\widehat{S}_\mu + (1- \widehat{S}_\mu^*\widehat{S}_\mu)$
implies
\begin{equation*}
1= \prod_{\mu \in B_l(\Lambda'_2)} \left( \widehat{S}_\mu^*\widehat{S}_\mu + (1- \widehat{S}_\mu^*\widehat{S}_\mu)\right)
 = \sum_{f \in \{1,-1\}^{B_l(\Lambda'_2)}} \widehat{E}_f.   
\end{equation*}
Fix $f \in \{1,-1\}^{B_l(\Lambda'_2)}$
such that $\widehat{E}_f\ne 0$.
Since
$
1=
\bigvee_{\mu\in B_l(\Lambda'_2)}
\widehat{S}_{\mu}^* \widehat{S}_{\mu},
$
we have
\begin{equation*}
\widehat{E}_f =
\widehat{E}_f \cdot\bigvee_{\mu\in B_l(\Lambda'_2)}
\widehat{S}_{\mu}^* \widehat{S}_{\mu}=
\bigvee_{\mu\in B_l(\Lambda'_2)}
 \widehat{E}_f \widehat{S}_{\mu}^* \widehat{S}_{\mu}.
\end{equation*}
If $f(\mu) = -1$ for all $\mu \in B_l(\Lambda'_2)$, then 
$\widehat{E}_f\cdot  \widehat{S}_{\mu}^* \widehat{S}_{\mu}=0$ for all $\mu \in B_l(\Lambda'_2)$.
Hence the condition $\widehat{E}_f\ne 0$
implies that there exists $\mu\in B_l(\Lambda'_2)$
such that $f(\mu) = 1$.
We fix such $\mu$ and write it as 
$\mu =((\eta_1,w_{[1,L]}), \dots, (\eta_m,w_{[m, L+m-1]})).$
Since 
\begin{equation*}
 \widehat{S}_\mu^* \widehat{S}_\mu 
= 
S_{\eta_{m+1} y_{[m+1,L+m-1]}}
S_{\eta_1 y_{[1,L+m-1]}}^*
S_{\eta_1 y_{[1,L+m-1]}}
S_{\eta_{m+1} y_{[m+1,L+m-1]}}^* \quad \text{ in } \OLamtwomin, 
\end{equation*}
one may find a vertex $v_j^{K+L+m-1} \in V_{K+L+m-1}^{\Lambda_2^\min}$
in the $\lambda$-graph system 
$\LLamtwomin$, 
such that 
there exists $\gamma \in E_{m,K+L+m-1}^{\Lambda_2^\min}$ satisfying
\begin{equation} \label{eq:vjklm}
\begin{cases}
& \eta_1 w_{[1,L+m-1]} \in \Gamma_{K+L+m-1}^-(v_j^{K+L+m-1}),\\
& \lambda(\gamma) = \eta_{m+1}w_{[m+1, L+m-1]}, \\ 
& \lambda(\gamma) \in \Gamma_{K+L-1}^-(v_j^{K+L+m-1}). 
\end{cases} 
\end{equation}
We may further find a vertex 
$v_j^{K+L+m-1}$ in $V_{K+L+m-1}^{\Lambda_2^\min}$
for the word $\mu$ satisfying  \eqref{eq:vjklm}
but not satisfying the conditions for the words $\nu\in B_l(\Lambda'_2)$
such that $f(\nu) =-1$.
We take such a vertex and write it as 
 $v_{j_0}^{K+L+m-1}$.
 Let 
$E_{j_0}^{K+L+m-1}$ be the corresponding projection 
in the $C^*$-algebra $\OLamtwomin$
so that we have
$$
\widehat{E}_f \ge E_{j_0}^{K+L+m-1}.
$$
Since 
$\LLamtwomin$ is $\lambda$-synchronizing,
there exists an admissible  word 
$(a_1,\dots,a_k) \in B_*^+(\Lambda_2)$
such that 
$v_{j_0}^{K+L+m-1}$ launches $(a_1,\dots,a_k)$.
We may assume that $k >K+L$
for the length $k$ of the word $(a_1,\dots,a_k).$
Hence we see that 
$$
 E_{j_0}^{K+L+m-1} \ge S_{(a_1,\dots,a_k)}S_{(a_1,\dots,a_k)}^* \quad \text{ in }
\OLamtwomin.
$$
Let $n = k-(K+L) +1$.
Put
$$
\xi_1 =(a_1,\dots,a_K), \quad
y_{[1, L + n-1]} = (a_{K+1},\dots,a_k)
$$
%so that 
%$$
%y_{[1,L]} = (a_{K+1},\dots,a_{K+L}), \qquad
%y_{[2,L+1]} = (a_{K+2},\dots,a_{K+L+1}), \qquad \dots
%y_{[n,L+n-1]} = (a_{K+n},\dots,a_{K+L+n-1}).
%$$
and
$\xi_{i+1} = \tau^{\varphi_1}(\xi_i,y_{[i, L+i-1]}), i=1,2,\dots, n$. 
Define the word
$$
\nu = ((\xi_1, y_{[1,L]}), \dots, (\xi_n, y_{[n,L+n-1]})) \in B_n(\Lambda'_2).
$$ 
By Lemma \ref{lem:5.7} (ii), we know
$$
\widehat{S}_\nu \widehat{S}_\nu^* 
= S_{\xi_1y_{[1,L+n-1]}}S_{\xi_1y_{[1,L+n-1]}}^*
= S_{(a_1,\dots,a_k)}S_{(a_1,\dots,a_k)}^*.
$$
Hence we have
\begin{equation}\label{eq:widehatEf}
\widehat{E}_f \ge \widehat{S}_\nu \widehat{S}_\nu^* \quad \text{ in }
 \OhatLamtwomin.
\end{equation}
We put
\begin{equation*}
E_f^{\prime} 
= \prod_{\mu \in B_l(\Lambda'_2)} S_\mu^{\prime*} S_\mu^{\prime f(\mu)} 
\quad \text{ in }
 \OLamptwomin.
\end{equation*}
By applying Lemma \ref{lem:iff} for \eqref{eq:widehatEf}, 
we have 
\begin{equation}\label{eq:Eprimef}
E^{\prime}_f \ge S_\nu^\prime S_\nu^{\prime*} \quad \text{ in }
\OLamptwomin
\end{equation}
and hence 
$E^{\prime}_f \ne 0$.
Since
$$
1 = \sum_{f \in \{1,-1\}^{B_l(\Lambda'_2)}} E^\prime_f
  = \sum_{i=1}^{m'(l)} E_i^{\prime l}
$$
and 
$E^{\prime}_{f_i^l} =E_i^{\prime l}$,
we know that 
$E^{\prime}_f \ne 0$ if and only if $f = f_i^l$ for some $v_i^{\prime l} \in V_l^\prime$.
Hence the condition
$E^{\prime}_f \ne 0$ implies that $f = f_i^l$ for some $v_i^{\prime l} \in V_l^\prime$.
Therefore we see that 
$\widehat{E}_f = \widehat{E}_{f_i^l}$
for some $v_i^{\prime l} \in V_l^\prime$.
Since 
$\widehat{E}_{f_i^l} = \widehat{E}_i^l$
and $1 =\sum_{f \in \{1,-1\}^{B_l(\Lambda'_2)}} \widehat{E}_f$,
we conclude that 
\begin{equation}\label{eq:oneeil}
1 =\sum_{i=1}^{m'(l)} \widehat{E}_i^l \quad \text{ in }
 \OhatLamtwomin.
\end{equation}
Since for each $j =1,2,\dots,m'(l+1)$,
there exists a unique $i=1,2,\dots,m'(l)$ such that 
$I'_{l,l+1}(i,j) =1$, 
we have
$ \widehat{E}_j^{l+1} = \sum_{i=1}^{m'(l)} I'_{l,l+1}(i,j)\widehat{E}_j^{l+1} $.
As the identity \eqref{eq:oneeil} holds for all $l \in \Zp$,
we have
\begin{equation*}
1 = \sum_{j=1}^{m'(l+1)}\widehat{E}_j^{l+1}
  %= \sum_{j=1}^{m'(l+1)}\sum_{i=1}^{m'(l)} I'_{l,l+1}(i,j)\widehat{E}_j^{l+1}
  = \sum_{i=1}^{m'(l)}\sum_{j=1}^{m'(l+1)} I'_{l,l+1}(i,j)\widehat{E}_j^{l+1}
\end{equation*}
so that 
\begin{equation}
1 =\sum_{i=1}^{m'(l)}\widehat{E}_i^l 
  = \sum_{i=1}^{m'(l)}\sum_{j=1}^{m'(l+1)} I'_{l,l+1}(i,j)\widehat{E}_j^{l+1}.
\end{equation}
By the inequality \eqref{eq:wideinequal},
we conclude that 
\begin{equation*}
\widehat{E}_i^l  = \sum_{j=1}^{m'(l+1)} I'_{l,l+1}(i,j) \widehat{E}_j^{l+1}.
\end{equation*}
\end{proof}
Define the commutative $C^*$-subalgebras: 
\begin{align*}
\ALLamptwomin
= & C^*( S_\mu^{\prime*}S_\mu^{\prime} : \mu \in B_*(\Lambda'_2)) 
\subset \OLamptwomin, \\
\AhatLLamtwomin
= & C^*( \widehat{S}_\mu^{*} \widehat{S}_\mu : \mu \in B_*(\Lambda'_2)) 
\subset \OhatLamtwomin.
\end{align*}
\begin{lemma}Keep the above notation.
\begin{enumerate}
\renewcommand{\theenumi}{\roman{enumi}}
\renewcommand{\labelenumi}{\textup{(\theenumi)}}
\item
\begin{align*}
\ALLamptwomin
= & C^*( E_i^{\prime l} : i=1,2,\dots,m'(l), \, l\in\Zp), \\
\AhatLLamtwomin
= & C^*( \widehat{E}_i^l : i=1,2,\dots,m'(l), \, l\in\Zp).
\end{align*}
\item
There exists an isomorphism
$\Phi_\A:\ALLamptwomin\longrightarrow \AhatLLamtwomin
$
of $C^*$-algebras such that 
\begin{equation} \label{eq:SalphaPhi}
\widehat{S}_\alpha^*\Phi_\A(X)\widehat{S}_\alpha 
= \Phi_\A(S_\alpha^{\prime*} X S_\alpha^\prime), 
\qquad X \in \ALLamptwomin, \, \alpha \in \Sigma'_2.
\end{equation}
\end{enumerate}
\end{lemma}
\begin{proof}
(i) 
By the identity \eqref{eq:eiprimel},
$E_i^{\prime l}$ is written in terms of  $S_\mu^{\prime*}S_\mu^{\prime f_i^l(\mu)}$.
Conversely for any word $\mu \in B_l(\Lambda'_2)$
we set
$$J'(\mu,i) =
\begin{cases}
1 & \text{ if } \mu \in \Gamma_l^-(v_i^{\prime l}), \\
0 & \text{ otherwise}.
\end{cases}
$$
By the formula $ 1 = \sum_{i=1}^{m'(l)}E_i^{\prime l}$, we have
\begin{equation*}
S_\mu^{\prime*}S_\mu^{\prime} 
=\sum_{i=1}^{m'(l)} S_\mu^{\prime*}S_\mu^{\prime}E_i^{\prime l}
=\sum_{i=1}^{m'(l)} J'(\mu,i) E_i^{\prime l},
\end{equation*}
so that  $S_\mu^{\prime*}S_\mu^{\prime f_i^l(\mu)}$
 is written in terms of $E_i^{\prime l}$.
Hence we have 
$$  
\ALLamptwomin
= C^*( E_i^{\prime l} : i=1,2, \dots,m'(l), \, l\in\Zp).
$$
The other equality
$$\AhatLLamtwomin
=  C^*( \widehat{E}_i^l : i=1,2, \dots,m'(l), \, l\in\Zp).
$$
is similarly proved.

 (ii)
The identities 
\begin{gather*}
E_i^{\prime l} = \sum_{j=1}^{m'(l+1)} I'_{l,l+1}(i,j) E_j^{\prime l+1}, \qquad
1 = \sum_{i=1}^{m'(l)}E_i^{\prime l} \quad \text{ in } \quad \ALLamptwomin, \\
\widehat{E}_i^{l} = \sum_{j=1}^{m'(l+1)} I'_{l,l+1}(i,j) \widehat{E}_j^{l+1}, \qquad
1 = \sum_{i=1}^{m'(l)} \widehat{E}_i^{ l} \quad \text{ in } \quad \AhatLLamtwomin
\end{gather*}
hold.
Since the projections
$E_i^{\prime l}, \widehat{E}_i^l$ are all nonzero,
the correspondence 
$E_i^{\prime l} \longrightarrow  \widehat{E}_i^l$
 extends to an isomorphism 
$\Phi_\A:\ALLamptwomin\longrightarrow \AhatLLamtwomin
$
of $C^*$-algebras such that
$\Phi_\A(E_i^{\prime l})=\widehat{E}_i^l$
and hence 
$\Phi_\A(S_\mu^{\prime*}S_\mu^{\prime} )
=\widehat{S}_\mu^{*} \widehat{S}_\mu 
$ for $\mu \in B_*(\Lambda'_2)$.
We then have 
\begin{equation*}
\widehat{S}_\alpha^*\Phi_\A(S_\mu^{\prime*}S_\mu^{\prime} )
\widehat{S}_\alpha
=\widehat{S}_\alpha^* \widehat{S}_\mu^{*} \widehat{S}_\mu \widehat{S}_\alpha
=\widehat{S}_{\mu\alpha}^{*} \widehat{S}_{\mu\alpha}
=\Phi_\A(S_{\mu\alpha}^{\prime*}S_{\mu\alpha}^{\prime} )
=\Phi_\A(S_{\alpha}^{\prime*}S_{\mu}^{\prime*}S_{\mu}^{\prime}S_{\alpha}^{\prime} )
\end{equation*}
proving the identity \eqref{eq:SalphaPhi}.
\end{proof}
Recall that  $A'_{l,l+1}(i,\alpha,j)$ is the the matrix component of the transition matrix system
$(A'_{l,l+1}, I'_{l,l+1})_{l\in \Zp}$ of ${\frak L}_{\Lambda'_2}^\min$.
\begin{lemma}
\begin{equation*}
\widehat{S}_\alpha^* \widehat{E}_i^{l} \widehat{S}_\alpha
= \sum_{j=1}^{m'(l+1)} A'_{l,l+1}(i,\alpha,j) \widehat{E}_j^{l+1} \quad 
\text{ for } \alpha=(\xi_1,y_{[1,L]}) \in \Sigma'_2. 
\end{equation*}
\end{lemma}
\begin{proof}
We note that the identity 
\begin{equation*}
S_\alpha^{\prime*} E_i^{\prime l} S_\alpha^{\prime}
= \sum_{j=1}^{m'(l+1)} A'_{l,l+1}(i,\alpha,j) E_j^{\prime l+1} \quad 
\text{ for } \alpha=(\xi_1,y_{[1,L]}) \in \Sigma'_2 
\end{equation*}
holds.
By using the preceding lemma,  we have 
\begin{align*}
\widehat{S}_\alpha^* \widehat{E}_i^{l} \widehat{S}_\alpha
=& 
\widehat{S}_\alpha^* \Phi_\A({E}_i^{\prime l}) \widehat{S}_\alpha 
= \Phi({S}_\alpha^{\prime* }{E}_i^{\prime l}{S}_\alpha^\prime) \\
=& \Phi(\sum_{j=1}^{m'(l+1)} A'_{l,l+1}(i,\alpha,j) E_j^{\prime l+1} ) \\
=& \sum_{j=1}^{m'(l+1)} A'_{l,l+1}(i,\alpha,j) \widehat{E}_j^{l+1}. 
\end{align*}
\end{proof}

%%%%%%%%%%%%%%%%%%%%%%%%%%%%%%%%%%%%%%%%%%%%%%%%%%%%%%%%%%%%%%

Recall that 
\begin{align*}
\OhatLamtwomin
=& C^*(\widehat{S}_{(\xi_1, y_{[1,L]})} : (\xi_1, y_{[1,L]})\in \Sigma'_2 ), \\\DhatLLamtwomin
= & C^*(
\widehat{S}_\mu \widehat{S}_\nu^* \widehat{S}_\nu \widehat{S}_\mu^*,
: \mu, \nu \in B_*(\Lambda'_2)), \\
\DhatLamtwo
= & C^*(
\widehat{S}_\mu \widehat{S}_\mu^* : \mu \in B_*(\Lambda'_2)).
\end{align*}
Then the inclusion relations 
\begin{equation*}
\OhatLamtwomin \subset \OLamtwomin, \qquad
\DhatLLamtwomin \subset \DLLamtwomin, \qquad
\DhatLamtwo \subset \DLamtwo
\end{equation*}
are obvious.
Let
$\rhohatLamtwo_t$ be the restriction of the gauge action 
$\rhoLamtwo_t$ on $\OLamtwomin$ to the subalgebra  $\OhatLamtwomin$.
The gauge action on $\mathcal{O}_{\Lambda_2^{\prime \min}}$
is denoted by $\rho_t^{\Lambda'_2}$.
\begin{lemma}\label{lem:OLampOlam}
Keep the above notation. 
There exists an isomorphism
$\Phi: \OLamptwomin \longrightarrow  \OhatLamtwomin$ of $C^*$-algebras such that 
\begin{equation*}
\Phi(\DLLamptwomin) = \DhatLLamtwomin, \qquad
\Phi(\DLamptwo) =\DhatLamtwo, \qquad
\Phi\circ \rhoLamptwo_t= \rhohatLamtwo_t \circ \Phi.
\end{equation*}
\end{lemma}
\begin{proof}
By the universal property and its uniqueness of the $C^*$-algebra
 $\OLamptwomin$, the correspondence 
\begin{equation*}
\Phi: S_\alpha^\prime, \, E_i^{\prime l}  \in \OLamptwomin 
\longrightarrow 
\widehat{S}_\alpha, \, \widehat{E}_i^l 
\in \OhatLamtwomin \subset \OLamtwomin
\end{equation*}
yields an isomorphism 
$\Phi: \OLamptwomin \longrightarrow  \OhatLamtwomin$ of $C^*$-algebras such that 
\begin{equation*}
\Phi(\DLLamptwomin) = \DhatLLamtwomin, \qquad
\Phi(\DLamptwo) =\DhatLamtwo.
\end{equation*}
Since
\begin{align*}
(\rhohatLamtwo_t \circ \Phi) (S_\alpha^\prime)
&=\rhohatLamtwo_t (\widehat{S}_\alpha)
=\rhoLamtwo_t(S_{\xi_1 y_{[1,L]}}S_{\xi_2 y_{[2,L]}}^*) \\
&= e^{2\pi\sqrt{-1}t}S_{\xi_1 y_{[1,L]}}S_{\xi_2 y_{[2,L]}}^*
= e^{2\pi\sqrt{-1}t}\widehat{S}_\alpha
=\Phi\circ \rhoLamptwo_t(S_\alpha^\prime),
\end{align*}
the equality
$\Phi\circ \rhoLamptwo_t= \rhohatLamtwo_t \circ \Phi
$ holds.
\end{proof}

\medskip
We will finally prove that 
the $C^*$-subalgebra
$\OhatLamtwomin$ of $\OLamtwomin$
actually coincides with the ambient algebra $\OLamtwomin$,
that is the final step to prove Theorem \ref{thm:main1.4}.

Let 
$\FLamtwomin
$ 
be the canonical AF algebra of $\OLamtwomin$
that is realized as the fixed point subalgebra of $\OLamtwomin$
under the gauge action $\rhoLamtwo_t, t \in \T$ of $\OLamtwomin$. 
Let
$\FhatLamtwomin$
be the $C^*$-subalgebra
of $\OhatLamtwomin$
  generated by elements of the form:
$$
\widehat{S}_\nu \widehat{S}_\mu^* \widehat{S}_\mu \widehat{S}_\gamma^*,
\qquad \mu, \nu,\gamma \in B_*(\Lambda'_2) \text{ with } |\nu| = |\gamma|.
$$
The subalgebra $\FhatLamtwomin$
is nothing but the  $C^*$-subalgebra
of $\OhatLamtwomin$
  generated by elements of the form:
$$
\widehat{S}_\nu \widehat{E}_i^l \widehat{S}_\gamma^*,
\qquad i=1,2,\dots,m'(l), l\in \Zp, \, 
\nu,\gamma \in B_*(\Lambda'_2) \text{ with } |\nu| = |\gamma|.
$$
\begin{lemma}\label{lem:AF}
$\FhatLamtwomin =\FLamtwomin.
$ 
\end{lemma}
\begin{proof}
Let
$\nu =((\xi_1,y_{[1,L]}), \dots, (\xi_n,y_{[n, L+n-1]})),
\mu =((\eta_1,w_{[1,L]}), \dots, (\eta_m,w_{[m, L+m-1]}))
$
and
$
\gamma = ((\zeta_1,z_{[1,L]}), \dots, (\zeta_k,z_{[k,L+k-1]}))
\in B_*(\Lambda'_2)
$
with $k=n$.
Since
$\widehat{S}_{(\xi_1, y_{[1,L]})} = S_{\xi_1 y_{[1,L]}}S_{\xi_2 y_{[2,L]}}^*
\in \OLamtwomin,
$
we have
\begin{align*}
& \widehat{S}_\nu \widehat{S}_\mu^* \widehat{S}_\mu \widehat{S}_\gamma^* \\
= & \widehat{S}_{(\xi_1,y_{[1,L]})} \cdots \widehat{S}_{(\xi_n,y_{[n, L+n-1]})} \cdot
(\widehat{S}_{(\eta_1,w_{[1,L]})} \cdots \widehat{S}_{(\eta_m,w_{[m, L+m-1]})})^* \cdot \\
 & (\widehat{S}_{(\eta_1,w_{[1,L]})} \cdots \widehat{S}_{(\eta_m,w_{[m, L+m-1]})}) \cdot
(\widehat{S}_{(\zeta_1,z_{[1,L]})}\cdots \widehat{S}_{(\zeta_n,z_{[n,L+n-1]})})^* \\
= & S_{\xi_1y_{[1,L+n-1]}} S_{\xi_{n+1}y_{[n+1, L+n-1]}}^* \cdot
 S_{\eta_{m+1}w_{[m+1,L+m-1]}} S_{\eta_1w_{[1, L+m-1]}}^* \cdot \\
 & S_{\eta_1w_{[1, L+m-1]}} S_{\eta_{m+1}w_{[m+1,L+m-1]}}^* \cdot
S_{\zeta_{n+1}z_{[n+1,L+n-1]}}S_{\zeta_1 z_{[1,L+n-1]}}^* \\
= &
\begin{cases}
 S_{\xi_1y_{[1,L+n-1]}}  S_{\eta_1w_{[1, L+m-1]}}^* \cdot
  S_{\eta_1w_{[1, L+m-1]}} S_{\zeta_1 z_{[1,L+n-1]}}^*  & \\
\quad \text{ if } \xi_{n+1} = \eta_{m+1} = \zeta_{n+1}, \, 
y_{[n+1, L+n-1]} = w_{[m+1,L+m-1]} =z_{[n+1,L+n-1]}, \\
0 \quad \text{ otherwise. } &
\end{cases}
\end{align*}
Hence 
$\widehat{S}_\nu \widehat{S}_\mu^* \widehat{S}_\mu \widehat{S}_\gamma^* 
$ belongs to 
$\FLamtwomin$,
so that $\FhatLamtwomin \subset \FLamtwomin.$

Conversely,
for admissible words $a, b, c \in B_*(\Lambda_2)$ with $|a| = |c|$, 
by considering the identity
\begin{equation}
S_a S_b^* S_b S_c^* = 
\sum_{\delta\in B_{N}(\Lambda_2)}
S_{a\delta} S_{b\delta}^* S_{b\delta} S_{c\delta}^*  \label{eq:Sabc}
\end{equation}
for any $N\in \N$, one may assume that 
$|a| (=|c|), |b| >K$.
For $\rho \in \Gamma_\infty^+(a) \cap\Gamma_\infty^+(b) \cap\Gamma_\infty^+(c)$,
we define $(y_i)_{i\in \N}, (w_i)_{i\in \N}, (z_i)_{i\in \N} \in X_{\Lambda_2}$ by setting
\begin{align*}
 y_1 & =a_{K+1}, \, \dots,\,  y_{|a|-K} = a_{|a|},
\quad  y_{|a| -K +i} = \rho_i, \quad i=1,2,\dots,\\
 w_1 & =b_{K+1}, \, \dots, \, w_{|b|-K} = b_{|b|},
\quad  w_{|b| -K +i} = \rho_i, \quad i=1,2,\dots,\\
 z_1 & =c_{K+1}, \, \dots, \, z_{|b|-K} = c_{|c|},
\quad  z_{|c| -K +i} = \rho_i, \quad i=1,2,\dots. 
\end{align*}
Define sequences $(\xi_i)_{i\in \N}, (\eta_i)_{i\in \N},  (\zeta_i)_{i\in \N}$
 of $B_K(\Lambda_2)$
by setting:
\begin{align*}
\xi_1 =& a_{[1,K]}, \qquad \xi_{i+1} = \tau^{\varphi_1}(\xi_i, y_{[i,L+i-1]}), \quad i=1,2 \dots,\\
\eta_1 =& b_{[1,K]}, \qquad \eta_{i+1} = \tau^{\varphi_1}(\eta_i, w_{[i,L+i-1]}), \quad i=1,2 \dots,\\
\zeta_1 =& c_{[1,K]}, \qquad \zeta_{i+1} = \tau^{\varphi_1}(\xi_i, z_{[i,L+i-1]}), \quad i=1,2 \dots.
\end{align*}
Define elements $x=(x_i)_{i\in \N}, x'=(x'_i)_{i\in \N}, 
x^{\prime\prime}=(x^{\prime\prime}_i)_{i\in \N} \in X_{\Lambda_1}$
by setting
\begin{equation*}
x = h^{-1}(\xi_1 y), \qquad 
x' = h^{-1}(\eta_1 w), \qquad 
x^{\prime\prime} = h^{-1}(\zeta_1 z). 
\end{equation*}
By previous discussions, we know that 
\begin{equation*}
\xi_i =\varphi_1(x_{[i, M+i-1]}), \qquad 
\eta_i=\varphi_1(x'_{[i, M+i-1]}), \qquad 
\zeta_i=\varphi_1(x^{\prime\prime}_{[i, M+i-1]}), \qquad 
i=1,2,\dots. 
\end{equation*}
Put
$p = |a| - |b|\in \Z.$
Since 
$\sigma_{\Lambda_1}^{K}\circ h^{-1} :
X_{\Lambda_1}\longrightarrow X_{\Lambda_1}
$
is a sliding block code, there exists $N_1, N_2\in \N$
such that 
$w_{j+p} = y_j (=z_j)$ for all $j \ge N_2$
implies
$x'_{[i+p, M+i+p-1]} =x_{[i, M+i-1]} (= x^{\prime\prime}_{[i, M+i-1]})$
for all $i \ge N_1.$
Hence we have
$$
\eta_i = \xi_{i+p} = \zeta_{i+p} \qquad \text{ for } i \ge N_1.
$$
Let $n =\max\{N_1, N_2\}, m=n +p$.
By putting
\begin{align*}
\nu =& ((\xi_1,y_{[1,L]}), \dots, (\xi_n,y_{[n, L+n-1]})) \in B_n(\Lambda'_2), \\
\mu = &((\eta_1,w_{[1,L]}), \dots, (\eta_m,w_{[m, L+m-1]}))\in B_m(\Lambda'_2), \\
\gamma =& ((\zeta_1,z_{[1,L]}), \dots, (\zeta_n,z_{[k,L+n-1]})) \in B_n(\Lambda'_2),
\end{align*}
we have 
$$
\eta_{m+1} = \xi_{n+1} = \zeta_{n+1}, \qquad
w_{[m, L+m-1]} =y_{[n, L+n-1]} =z_{[k,L+n-1]}.
$$
Let
$\delta = (\rho_1,\dots,\rho_{K+L+n-|a|-1}) \in B_*(\Lambda_2).$
We then have 
\begin{align*}
S_{a\delta} S_{b\delta}^* S_{b\delta} S_{c\delta}^*
& =S_{\xi_1y_{[1,L+n-1]}}  S_{\eta_1w_{[1, L+m-1]}}^* \cdot
  S_{\eta_1w_{[1, L+m-1]}} S_{\zeta_1 z_{[1,L+n+1]}}^*  \\
&= \widehat{S}_\nu \widehat{S}_\mu^* \widehat{S}_\mu \widehat{S}_\gamma^*. 
\end{align*}
By the formula
\eqref{eq:Sabc} for $N = K+L+n-|a|-1$, 
we know that 
$S_a S_b^* S_b S_c^* $
with $|a| = |c|$ 
belongs to the AF-algebra
$\FhatLamtwomin$, so that 
we have
$\FLamtwomin\subset \FhatLamtwomin$. 
\end{proof}

\begin{lemma}\label{lem:OhatOlam}
$\widehat{\mathcal{D}}_{\Lambda'_2} = D_{\Lambda_2}$
and
$\OhatLamtwomin =\OLamtwomin$.
\end{lemma}
\begin{proof}
The equality 
$\widehat{\mathcal{D}}_{\Lambda'_2} = D_{\Lambda_2}$
is easily obtained by Lemma \ref{lem:5.7} (ii).

The inclusion relation
$\OhatLamtwomin \subset \OLamtwomin$
is obvious.
To prove $\OhatLamtwomin = \OLamtwomin$,
it suffices to show that 
for any $\alpha \in \Sigma_2 = B_1(\Lambda_2)$,
the partial isometry $S_\alpha $ belongs to 
$\OhatLamtwomin.$
For $(\xi_1, y_{[1,L]}) \in \Sigma'_2$, we have 
$
\widehat{S}_{(\xi_1, y_{[1,L]})} = S_{\xi_1y_{[1,L]}} S_{\xi_2y_{[2,L]}}^* 
\in  \OLamtwomin,
$
so that  for $t \in \T = \mathbb{R}/\Z$
\begin{equation*}
\rhoLamtwo_t(\widehat{S}_{(\xi_1, y_{[1,L]})}^*S_\alpha)
=
\rhoLamtwo_t( S_{\xi_2y_{[2,L]}} S_{\xi_1y_{[1,L]}}^*S_\alpha)
=
S_{\xi_2y_{[2,L]}} S_{\xi_1y_{[1,L]}}^*S_\alpha
=\widehat{S}_{(\xi_1, y_{[1,L]})}^*S_\alpha.
\end{equation*}
This implies that 
$\widehat{S}_{(\xi_1, y_{[1,L]})}^*S_\alpha
\in \FLamtwomin.
$
By Lemma \ref{lem:AF},
we see that 
$\widehat{S}_{(\xi_1, y_{[1,L]})}^*S_\alpha
$
belongs to
$\FhatLamtwomin$
for any 
$(\xi_1, y_{[1,L]}) \in \Sigma'_2$.
By the identity
\begin{equation*}
S_\alpha =
\sum_{(\xi_1, y_{[1,L]}) \in \Sigma'_2} 
\widehat{S}_{(\xi_1, y_{[1,L]})}\cdot \widehat{S}_{(\xi_1, y_{[1,L]})}^*S_\alpha,
\end{equation*} 
we obtain that 
$S_\alpha$ belongs to $\OhatLamtwomin,$
and hence 
$\OLamtwomin \subset \OhatLamtwomin.$
\end{proof}
We thus have 
\begin{proposition}\label{prop:mainevent2}
There exists an isomorphism
$\Phi_2: \OLamptwomin \longrightarrow  \OLamtwomin$ of $C^*$-algebras such that 
\begin{equation*}
\Phi_2(\DLamptwo) =\DLamtwo, \qquad
\Phi_2\circ \rhoLamptwo_t= \rhoLamtwo_t \circ \Phi_2.
\end{equation*}
\end{proposition}
\begin{proof} The assertion follows from 
Lemma \ref{lem:OLampOlam} and Lemma \ref{lem:OhatOlam}.
\end{proof}
Therefore we reach the following theorem.

\begin{theorem}\label{thm:eventconj}
Let $\Lambda_1$ and $\Lambda_2$ be normal subshifts.
If their one-sided subshifts 
$(X_{\Lambda_1},\sigma_{\Lambda_1})$ 
and 
$(X_{\Lambda_2},\sigma_{\Lambda_2})$ 
are eventually conjugate, then there exists an isomorphism
$\Phi:\OLamonemin\longrightarrow\OLamtwomin$ of $C^*$-algebras such that 
$\Phi({\mathcal{D}}_{\Lambda_1}) ={\mathcal{D}}_{\Lambda_2}$
and
$\Phi\circ\rho^{\Lambda_1}_t = \rho^{\Lambda_2}_t\circ\Phi, t \in \T$.
\end{theorem}
\begin{proof}
By Lemma \ref{lem:LamptwoLamone},
the one-sided subshifts 
$ (X_{\Lambda'_2}, \sigma_{\Lambda'_2})$ and 
$(X_{\Lambda_1}, \sigma_{\Lambda_1})$
are topologically conjugate, so that by Theorem \ref{thm:one-sidedconj4}
 there exists an isomorphism
$\Phi_1:\OLamonemin\longrightarrow\OLamptwomin$ of $C^*$-algebras such that 
$\Phi_1({\mathcal{D}}_{\Lambda_1}) ={\mathcal{D}}_{\Lambda'_2}$
and
$\Phi_1\circ\rho^{\Lambda_1}_t = \rho^{\Lambda'_2}_t\circ\Phi_1, t \in \T$.
By Proposition \ref{prop:mainevent2},
there exists an isomorphism
$\Phi_2: \OLamptwomin \longrightarrow  \OLamtwomin$ of $C^*$-algebras such that 
\begin{equation*}
\Phi_2(\DLamptwo) =\DLamtwo, \qquad
\Phi_2\circ \rhoLamptwo_t= \rhoLamtwo_t \circ \Phi_2.
\end{equation*}
Therefore we have a desired isomorphism of $C^*$-algebras 
between $\OLamonemin$ and $\OLamtwomin$.
\end{proof}
\medskip

Let ${\frak L}_1,{\frak L}_2$
be left-resolving $\lambda$-graph systems that present subshifts
$\Lambda_1, \Lambda_2$, respectively.
In \cite{MaDynam2020}, the author introduced 
the notion of $({\frak L}_1,{\frak L}_2)$-eventually conjugacy 
between one-sided subshifts 
$(X_{\Lambda_1}, \sigma_{\Lambda_1})$
and
$(X_{\Lambda_2}, \sigma_{\Lambda_2})$.
\begin{definition}[{\cite[Definition 5.1]{MaDynam2020}}]
Let ${\frak L}_1,{\frak L}_2$
be left-resolving $\lambda$-graph systems that present subshifts
$\Lambda_1, \Lambda_2$, respectively.
One sided subshifts 
$(X_{\Lambda_1},\sigma_{\Lambda_1})$
and
$(X_{\Lambda_2},\sigma_{\Lambda_2})$
are said to be $({\frak L}_1,{\frak L}_2)$-{\it eventually conjugate}\/
if 
there exist homeomorphisms
$h_{\frak L}: X_{{\frak L}_1}\longrightarrow X_{{\frak L}_2}$,
$h_{\Lambda}: X_{\Lambda_1}\longrightarrow X_{\Lambda_2}$
and an integer $K\in \Zp$ such that 
\begin{equation} \label{eq:heventfrak}
\begin{cases}
&\sigma_{\L_2}^K(h_\L(\sigma_{\L_1}(x)))
  =\sigma_{\L_2}^{K+1}(h_\L(x)) , \qquad x \in X_{\L_1}, \\ 
&\sigma_{\L_1}^K(h_\L^{-1}(\sigma_{\L_2}(y)))
 =\sigma_{\L_1}^{K+1}(h_\L^{-1}(y)), \qquad y \in X_{\L_2},
\end{cases}
\end{equation}
and 
\begin{equation} \label{eq:heventpi}
\pi_{\L_2} \circ h_\L = h_{\Lambda}\circ \pi_{\L_1}.
\end{equation}
We remark that the equalities \eqref{eq:heventfrak} and \eqref{eq:heventpi}
automatically imply the equalities
\begin{equation} \label{eq:heventlambda}
\begin{cases}
&\sigma_{\Lambda_2}^K(h_\Lambda(\sigma_{\Lambda_1}(a))) 
 =\sigma_{\Lambda_2}^{K+1}(h_\Lambda(a)) , \qquad a \in X_{\Lambda_1}, \\ 
&\sigma_{\Lambda_1}^K(h_\Lambda^{-1}(\sigma_{\Lambda_2}(b)))
 =\sigma_{\Lambda_1}^{K+1}(h_\Lambda^{-1}(b)), \qquad b \in X_{\Lambda_2}.
\end{cases}
\end{equation}
\end{definition}
In {\cite{MaDynam2020}}, the following proposition  was proved.
\begin{proposition}[{\cite[Theorem 1.3]{MaDynam2020}}] \label{prop:eventconj} 
Suppose that two left-resolving $\lambda$-graph systems 
${\L_1}, {\L_2}$ satisfy condition (I). 
Then $(X_{\Lambda_1},\sigma_{\Lambda_1})$
and
$(X_{\Lambda_2},\sigma_{\Lambda_2})$
are 
$({\frak L}_1,{\frak L}_2)$-eventually conjugate
if and only if 
there exists an isomorphism
$\Phi:{\mathcal{O}}_{{\frak L}_1}\longrightarrow {\mathcal{O}}_{{\frak L}_2}
$ 
of $C^*$-algebras such that 
\begin{equation*}
%\Phi({\mathcal{D}}_{{\frak L}_1})={\mathcal{D}}_{{\frak L}_2},\qquad
\Phi({\mathcal{D}}_{{\Lambda}_1})={\mathcal{D}}_{{\Lambda}_2}
\quad
\text{ and }
\quad
\Phi \circ \rho_t^{{\frak L}_1} =\rho_t^{{\frak L}_2}\circ \Phi,
\quad
t \in \T.
\end{equation*}
\end{proposition}

\medskip

{\it Proof of Theorem \ref{thm:main1.4}.}
Let $\Lambda_1, \Lambda_2$ be two normal susbshifts.
Assume that 
$(X_{\Lambda_1},\sigma_{\Lambda_1})$
and
$(X_{\Lambda_2},\sigma_{\Lambda_2})$
are eventually conjugate.
By Theorem \ref{thm:eventconj},
there exists an isomorphism
$\Phi:\OLamonemin\longrightarrow\OLamtwomin$ of $C^*$-algebras such that 
$\Phi({\mathcal{D}}_{\Lambda_1}) ={\mathcal{D}}_{\Lambda_2}$
and
$\Phi\circ\rho^{\Lambda_1}_t = \rho^{\Lambda_2}_t\circ\Phi, t \in \T$.

Conversely, suppose that 
there exists an isomorphism
$\Phi:\OLamonemin\longrightarrow\OLamtwomin$ of $C^*$-algebras such that 
$\Phi({\mathcal{D}}_{\Lambda_1}) ={\mathcal{D}}_{\Lambda_2}$
and
$\Phi\circ\rho^{\Lambda_1}_t = \rho^{\Lambda_2}_t\circ\Phi, t \in \T$.
Let $\L_1, \L_2$ be their minimal $\lambda$-graph systems
${\frak L}_{\Lambda_1}^{\min}, {\frak L}_{\Lambda_2}^{\min}$, respectively.
The associated $C^*$-algebras 
$\mathcal{O}_{{\frak L}_{\Lambda_1}^{\min}}, \mathcal{O}_{{\frak L}_{\Lambda_2}^{\min}}
$
are nothing but the $C^*$-algebras
$\OLamonemin, \OLamtwomin$, respectively.
By Proposition \ref{prop:eventconj},
we know that 
$(X_{\Lambda_1},\sigma_{\Lambda_1})$
and
$(X_{\Lambda_2},\sigma_{\Lambda_2})$
are 
$({\frak L}_1,{\frak L}_2)$-eventually conjugate,
in particular, 
 $(X_{\Lambda_1},\sigma_{\Lambda_1})$
and
$(X_{\Lambda_2},\sigma_{\Lambda_2})$
are eventually conjugate.
\qed

%\newpage

%%%%%%%%%%%%%%%%%%%%%%%%%%%%%%%%%%%%%%%%%%%%%%%%%%%%%%%%%%%%%
\section{Two-sided topological conjugacy}
%%%%%%%%%%%%%%%%%%%%%%%%%%%%%%%%%%%%%%%%%%%%%%%%%%%%%
In this section,  we study two-sided topological conjugacy of normal subshifts 
in terms of the associated stabilized $C^*$-algebras with its diagonals and gauge actions.  
Following \cite{MaDynam2020}, we will consider the compact Hausdorff space
\begin{equation*}
\bar{X}_{\frak L} =\{
(\alpha_i,u_i) _{i \in \Z} \in \prod_{i\in \Z}(\Sigma\times\Omega_{\frak L})
\mid
(\alpha_{i+k},u_{i+k}) _{i \in \Z}
\in X_{\frak L} \text{ for all } k \in \Z\} %\label{eq:barXfrakL}
\end{equation*}
with the shift homeomorphism $\bar{\sigma}_{\frak L}$
\begin{equation*}
\bar{\sigma}_{\frak L}((\alpha_i,u_i) _{i \in \Z})=
 (\alpha_{i+1},u_{i+1})_{i \in \Z},
\qquad (\alpha_i,u_i) _{i \in \Z} \in \bar{X}_{\frak L} %\label{eq:barsigmafrakL}
\end{equation*}
on 
$\bar{X}_{\frak L}$,
where 
$\bar{X}_{\frak L}$ is endowed with the relative topology from the infinite product topology of $\prod_{i\in \Z}(\Sigma\times\Omega_{\frak L})$.
For
$x =(\alpha_i,u_i)_{i\in \Z}\in \bar{X}_{\frak L}$,
$\alpha =(\alpha_i)_{i\in \Z}\in \Lambda$
and
$k \in \Z$,
we set 
\begin{equation*}
%x_{[k,l]} =(\alpha_i,u_i)_{i=k}^l,\quad
%\alpha_{[k,l]} =(\alpha_i)_{i=k}^l,\quad
x_{[k,\infty)} =(\alpha_i,u_i)_{i=k}^\infty,\quad
\alpha_{[k,\infty)} =(\alpha_i)_{i=k}^\infty.
\end{equation*}
\begin{definition}[{\cite[Definition 7.1]{MaDynam2020}}]
The topological dynamical systems
$(\bar{X}_{{\frak L}_1}, \bar{\sigma}_{{\frak L}_1})$
and
$(\bar{X}_{{\frak L}_2}, \bar{\sigma}_{{\frak L}_2})$
are said to be {\it right asymptotically conjugate}\/ 
if there exists a homeomorphism
$\psi:\bar{X}_{{\frak L}_1}\longrightarrow \bar{X}_{{\frak L}_2}$
such that  
$\psi\circ\bar{\sigma}_{{\frak L}_1}=\bar{\sigma}_{{\frak L}_2}\circ\psi$
and 
\begin{enumerate}
\renewcommand{\theenumi}{\roman{enumi}}
\renewcommand{\labelenumi}{\textup{(\theenumi)}}
\item
for $m \in \Z$, there exists $M \in \Z$ such that 
$x_{[M,\infty)} = z_{[M,\infty)} $ implies 
$\psi(x)_{[m,\infty)} = \psi(z)_{[m,\infty)} $
for $x, z \in \bar{X}_{{\frak L}_1}$,
\item
for $n \in \Z$, there exists $N \in \Z$ such that 
$y_{[N,\infty)} = w_{[N,\infty)} $ implies 
$\psi^{-1}(y)_{[n,\infty)} = \psi^{-1}(w)_{[n,\infty)} $
for $y,w \in \bar{X}_{{\frak L}_2}$.
\end{enumerate}
We  call
$\psi:\bar{X}_{{\frak L}_1}\longrightarrow \bar{X}_{{\frak L}_2}$
{\it a right asymptotic conjugacy.}\/
\end{definition}
Let us denpte by 
$\bar{\pi}_i:\bar{X}_{{\frak L}_i}\longrightarrow \Lambda_i$
the factor map defined by
$
\bar{\pi}_i ((\alpha_i,u_i) _{i \in \Z}) = (\alpha_i)_{i\in\Z} \in\Lambda_i
$
for $i=1,2.$
\begin{definition}[{\cite[Definition 7.2]{MaDynam2020}}]
Two subshifts 
$\Lambda_1$ and 
$\Lambda_2$
are said to be $({\frak L}_1,{\frak L}_2)$-{\it conjugate}\/ if
there exists a right asymptotic conjugacy
$\psi_{\frak L}:\bar{X}_{{\frak L}_1}\longrightarrow \bar{X}_{{\frak L}_2}$
and a topological conjugacy
$\psi_\Lambda:\Lambda_1\longrightarrow \Lambda_2$
such that 
 $\bar{\pi}_2\circ\psi_{\frak L} = \psi_{\Lambda}\circ\bar{\pi}_1$.
\end{definition}

\begin{proposition}
Let  $\Lambda_1, \Lambda_2$ be normal subshifts
and $\L_1, \L_2$ be their minimal $\lambda$-graph systems, respectively.
Suppose that 
 $\Lambda_1, \Lambda_2$ are topologically conjugate, 
then 
they are $(\L_1, \L_2)$-conjugate.
\end{proposition}
\begin{proof}
We may assume that $\Lambda_1$ and $\Lambda_2$ are bipartitely related by a bipartite subshift $\widehat{\Lambda}$ over alphabet 
$\Sigma=C\sqcup D$ (see \cite{Nasu}, \cite{Nasu2}).
Hence  
there exist specifications 
$\kappa_1: \Sigma_1\longrightarrow C\cdot D$
and
$\kappa_2: \Sigma_2\longrightarrow D\cdot C$
such that  the $2$-higher block shift  
$\widehat{\Lambda}^{[2]}$ of 
$\widehat{\Lambda}$ is decomposed into two disjoint subshifts
$\widehat{\Lambda}^{[2]} 
= \widehat{\Lambda}^{CD} \sqcup \widehat{\Lambda}^{DC}$,
where 
\begin{align*}
\widehat{\Lambda}^{[2]} 
&= \{ (x_i x_{i+1})_{i\in \Z} \mid
(x_i)_{i\in \Z} \in \widehat{\Lambda} \}, \\
\widehat{\Lambda}^{CD}  
&= \{ (c_i d_i)_{i\in \Z} \in \widehat{\Lambda}^{[2]}  \mid
c_i \in C, \, d_i \in D, i \in \Z \}, \\
\widehat{\Lambda}^{DC}  
&= \{ (d_i c_{i+1})_{i\in \Z} \in \widehat{\Lambda}^{[2]}  \mid
d_i \in D, \, c_{i+1} \in C, i \in \Z \},
\end{align*}
and specifications
$\kappa_1: \Sigma_1\longrightarrow C\cdot D$,
$\kappa_2: \Sigma_2\longrightarrow D\cdot C$
mean injective maps.
The notion that 
 two subshifts $\Lambda_1, \Lambda_2$ are bipartitely related
 mean that $\Lambda_1, \Lambda_2$ are identified with
$\widehat{\Lambda}^{CD}, \widehat{\Lambda}^{DC}$
through $\kappa_1, \kappa_2$, respectively.

The specifications $\kappa_1$ and $\kappa_2$ naturally extend
to the maps
$B_*(\Lambda_1) \longrightarrow B_*(\widehat{\Lambda}^{CD})$
and
$B_*(\Lambda_2) \longrightarrow B_*(\widehat{\Lambda}^{DC})$,
respectively.
We still denote them by $\kappa_1$ and $\kappa_2$, respectively.
We write 
$\L_i =(V^i, E^i,\lambda^i,\iota^i), i=1,2.$
Let $(\alpha_i, u_i)_{i\in \Z} \in \bar{X}_{\L_1}$.
In the $\lambda$-graph system 
$\L_1$,
take  a vertex $u_i^l \in V^1_l$ such that
$(u_i^l)_{l\in \N} = u_i \in \Omega_{\L_1}, i\in \Zp$.
There exists an $l$-synchronizing word
$\mu_i^l \in S_l(\Lambda_1), i\in \N$
 such that  $u_i^l = [\mu_i^l]_l \in S_l(\Lambda_1)/\underset{l}{\sim}$.
Let $\kappa_1(\alpha_i) = c_i d_i$ for $c_i \in C, d_i \in D$.
As $\kappa_1(\mu_i^l) \in B_*(\widehat{\Lambda}^{CD})$,
take
$c_i^l \in C$ such that $\kappa_1(\mu_i^l) c_i^l \in B_*(\widehat{\Lambda}).$
We put
$\nu_i^l = \kappa_2^{-1}(d_i \kappa_1(\mu_i^{l+1}) c_i^{l+1}) \in V_l^{2}$
and
$\beta_i = \kappa_2^{-1}(d_{i-1} c_i) \in \Sigma_2.$  
We then have 
$\beta_i \nu_i^l \underset{l-1}{\sim} \nu_{i-1}^{l-1}$
and
$ \nu_i^l \underset{l-1}{\sim} \nu_{i-1}^{l-1}.$
Define
$w_i^l = [\nu_i^l] \in S_l(\Lambda_2)$
so that $w_i^l \in V_l^{2}$.
Since $\iota(w_i^{l+1}) =w_i^l$ for $l \in \N$,
we have $w_i =(w_i^l)_{l\in \N} \in \Omega_{\L_2}$ for $i \in \Z$
and
$(\beta_i, w_i)_{i\in \Z} \in \bar{X}_{\L_2}$. 
Under the identification
between $\widehat{\Lambda}^{DC}$ and $\Lambda_2$,
we know that the correspondence
\begin{equation*}
(\alpha_i, u_i)_{i\in \Z} \in \bar{X}_{\L_1} \longrightarrow 
(\beta_i, w_i)_{i\in \Z} \in \bar{X}_{\L_2}
\end{equation*}
written $\psi: \bar{X}_{\L_1} \longrightarrow \bar{X}_{\L_2}$
gives rise to a topological conjugacy 
between
$(\bar{X}_{\L_1}, \bar{\sigma}_{\L_1})$ 
and
$(\bar{X}_{\L_2}, \bar{\sigma}_{\L_2})$ 
such that 
$\psi: \bar{X}_{\L_1} \longrightarrow \bar{X}_{\L_2}$
is a right asymptotic conjugacy 
and
there exists a topological conjugacy 
$\psi_{\Lambda}: \Lambda_1 \longrightarrow  \Lambda_2$
such that 
$\pi_2\circ\psi = \psi_\Lambda\circ \pi_1$.
Therefore 
the two-sided subshifts
$(\Lambda_1,\sigma_{\Lambda_1})$ and
$(\Lambda_2,\sigma_{\Lambda_2})$ are
$(\L_1, \L_2)$-conjugate.
\end{proof}
Therefore we have
\begin{proposition}\label{prop:twosidedconjugate}
Let  $\Lambda_1, \Lambda_2$ be normal subshifts
and $\L_1, \L_2$ be their minimal $\lambda$-graph systems, respectively.
Then the following two conditions are equivalent.
\begin{enumerate}
\renewcommand{\theenumi}{\roman{enumi}}
\renewcommand{\labelenumi}{\textup{(\theenumi)}}
\item
The two-sided subshifts
$(\Lambda_1,\sigma_{\Lambda_1})$ and
$(\Lambda_2,\sigma_{\Lambda_2})$ are
$(\L_1, \L_2)$-conjugate.
\item
$(\Lambda_1,\sigma_{\Lambda_1})$ and
$(\Lambda_2,\sigma_{\Lambda_2})$ are
 topologically conjugate.
\end{enumerate}
\end{proposition}

\medskip

Let us recall that  $\K$ denotes the $C^*$-algebra of compact operators on the separable 
infinite dimensional Hilbert space $\ell^2(\N)$ and
$\C$denotes  its commutative $C^*$-subalgebra of diagonal operators.

{\it Proof of Theorem \ref{thm:main1.5}.}
Let $\Lambda_1, \Lambda_2$ be two normal susbshifts.
Suppose that the two-sided subshifts
$(\Lambda_1,\sigma_{\Lambda_1})$ and
$(\Lambda_2,\sigma_{\Lambda_2})$ are
 topologically conjugate.
By Proposition \ref{prop:twosidedconjugate},
they are $(\L_1, \L_2)$-conjugate, so that 
\cite[Theorem 1.4]{MaDynam2020} ensures us that 
there exists an isomorphism
$\widetilde{\Phi}:\OLamonemin\otimes\K\longrightarrow\OLamtwomin\otimes\K$ of $C^*$-algebras such that 
$\widetilde{\Phi}({\mathcal{D}}_{\Lambda_1}\otimes\C) ={\mathcal{D}}_{\Lambda_2}\otimes\C$
and
$\widetilde{\Phi}\circ(\rho^{\Lambda_1}_t\otimes\id) 
= (\rho^{\Lambda_2}_t\otimes\id) \circ \widetilde{\Phi}, t \in \T$.

Conversely suppose that 
there exists an isomorphism
$\widetilde{\Phi}:\OLamonemin\otimes\K\longrightarrow\OLamtwomin\otimes\K$ of $C^*$-algebras such that 
$\widetilde{\Phi}({\mathcal{D}}_{\Lambda_1}\otimes\C) ={\mathcal{D}}_{\Lambda_2}\otimes\C$
and
$\widetilde{\Phi}\circ(\rho^{\Lambda_1}_t\otimes\id) 
= (\rho^{\Lambda_2}_t\otimes\id) \circ \widetilde{\Phi}, t \in \T$.
By \cite[Theorem 1.4]{MaDynam2020}
the two-sided subshifts
$(\Lambda_1,\sigma_{\Lambda_1})$ and
$(\Lambda_2,\sigma_{\Lambda_2})$ are
 $(\L_1, \L_2)$-conjugate, and hence 
they are topologically conjugate.
\qed

%%%%%%%%%%%%%%%%%%%%%%%%%%%%%%%%%%%%%%%%%%%%%%
%%%%%%%%%%%%%%%%%%%%%%%%%%%%%%%%%%%%%%%%%%%%%%%

\medskip

%\newpage

%%%%%%%%%%%%%%%%%%%%%%%%%%%%%%%%%%%%%%
{\it Acknowledgments:}
The author would like to thank Wolfgang Krieger
for various discussions and constant encouragements.  
This work was  supported by 
JSPS KAKENHI Grant Numbers 15K04896, 19K03537.

%%%%%%%%%%%%%%%%%%%%%%%%%%%%%%

%%%%%%%%%%%%%%%%%%%%%%%%
%%%%%%%%%%%%%%%%%%%%%%%%%%%%%%%%%%%%%%%%%%%%%%%%%%%%%%%%%%%%%%%%%%%%%%%%%%%%%%

\end{document}